\documentclass{amsart}

\usepackage[margin=1.5in]{geometry}
\usepackage{amssymb}
\usepackage{amsthm}
\usepackage{arydshln}
\usepackage{bbm}
\usepackage{bm}
\usepackage{datetime2}
\usepackage{enumitem}
\usepackage{graphicx}
\usepackage{mathrsfs}
\usepackage{mathtools}
\usepackage{pdflscape}
\usepackage{tikz-cd}
\usepackage[bookmarkstype=toc, bookmarks=true, bookmarksnumbered=true, colorlinks=false, pdfborder={0 0 1}]{hyperref}

\DeclareMathOperator{\Ad}{Ad}
\DeclareMathOperator{\cInd}{c-Ind}
\DeclareMathOperator{\depth}{depth}

\DeclareMathOperator{\FT}{FT}
\DeclareMathOperator{\GL}{GL}
\DeclareMathOperator{\Hom}{Hom}
\DeclareMathOperator{\Ind}{Ind}
\DeclareMathOperator{\Inf}{Inf}

\DeclareMathOperator{\Irr}{Irr}
\DeclareMathOperator{\Lie}{Lie}

\DeclareMathOperator{\pInd}{pInd}
\DeclareMathOperator{\pr}{pr}
\DeclareMathOperator{\Sp}{Sp}

\DeclareMathOperator{\Tr}{Tr}

\newcommand{\bbB}{\mathbb{B}}
\newcommand{\bbG}{\mathbb{G}}

\newcommand{\bbM}{\mathbb{M}}

\newcommand{\bbT}{\mathbb{T}}
\newcommand{\bbU}{\mathbb{U}}
\newcommand{\bbW}{\mathbb{W}}
\newcommand{\C}{\mathbb{C}}
\newcommand{\F}{\mathbb{F}}
\newcommand{\Q}{\mathbb{Q}}
\newcommand{\R}{\mathbb{R}}
\newcommand{\Z}{\mathbb{Z}}

\newcommand{\bfG}{\mathbf{G}}
\newcommand{\bfH}{\mathbf{H}}
\newcommand{\bfJ}{\mathbf{J}}
\newcommand{\bfM}{\mathbf{M}}
\newcommand{\bfT}{\mathbf{T}}
\newcommand{\bfU}{\mathbf{U}}
\newcommand{\bfV}{\mathbf{V}}
\newcommand{\bfZ}{\mathbf{Z}}
\newcommand{\x}{\mathbf{x}}
\newcommand{\y}{\mathbf{y}}

\newcommand{\cB}{\mathcal{B}}
\newcommand{\cF}{\mathcal{F}}
\newcommand{\cFT}{\mathcal{FT}}
\newcommand{\cL}{\mathcal{L}}
\newcommand{\cO}{\mathcal{O}}
\newcommand{\cR}{\mathcal{R}}

\newcommand{\mfg}{\mathfrak{g}}
\newcommand{\mfj}{\mathfrak{j}}
\newcommand{\mfp}{\mathfrak{p}}
\newcommand{\mft}{\mathfrak{t}}

\newcommand{\bmfg}{\boldsymbol{\mathfrak{g}}}
\newcommand{\bmfj}{\boldsymbol{\mathfrak{j}}}
\newcommand{\bmft}{\boldsymbol{\mathfrak{t}}}

\newcommand{\sfQ}{\mathsf{Q}}

\newcommand{\AFMO}{\mathrm{AFMO}}
\newcommand{\der}{\mathrm{der}}
\newcommand{\FKS}{\mathrm{FKS}}

\newcommand{\KY}{\mathrm{KY}}
\newcommand{\nilp}{\mathrm{nilp}}
\newcommand{\nreg}{\mathrm{nreg}}
\newcommand{\nvreg}{\mathrm{nvreg}}
\newcommand{\ram}{\mathrm{ram}}
\newcommand{\red}{\mathrm{red}}
\newcommand{\reg}{\mathrm{reg}}
\newcommand{\sep}{\mathrm{sep}}
\renewcommand{\ss}{\mathrm{ss}}

\newcommand{\unip}{\mathrm{unip}}
\newcommand{\ur}{\mathrm{ur}}
\newcommand{\vreg}{\mathrm{vreg}}
\newcommand{\Yu}{\mathrm{Yu}}

\newcommand{\cc}{{}^{\circ}}
\newcommand{\from}{\colon}

\newcommand{\Ga}{\mathbb{G}_{\mathrm{a}}}
\newcommand{\ol}{\overline}
\newcommand{\Qlb}{\overline{\Q}_{\ell}}
\newcommand{\mapsfrom}{\mathrel{\reflectbox{$\mapsto$}}}

\makeatletter
\newcommand{\dashover}[2][\mathop]{#1{\mathpalette\df@over{{\dashfill}{#2}}}}
\newcommand{\fillover}[2][\mathop]{#1{\mathpalette\df@over{{\solidfill}{#2}}}}
\newcommand{\df@over}[2]{\df@@over#1#2}
\newcommand\df@@over[3]{%
  \vbox{
    \offinterlineskip
    \ialign{##\cr
      #2{#1}\cr
      \noalign{\kern1pt}
      $\m@th#1#3$\cr
    }
  }%
}
\newcommand{\dashfill}[1]{%
  \kern-.5pt
  \xleaders\hbox{\kern.5pt\vrule height.4pt width \dash@width{#1}\kern.5pt}\hfill
  \kern-.5pt
}
\newcommand{\dash@width}[1]{%
  \ifx#1\displaystyle
    2pt
  \else
    \ifx#1\textstyle
      1.5pt
    \else
      \ifx#1\scriptstyle
        1.25pt
      \else
        \ifx#1\scriptscriptstyle
          1pt
        \fi
      \fi
    \fi
  \fi
}
\newcommand{\solidfill}[1]{\leaders\hrule\hfill}
\makeatother

\theoremstyle{plain}
\newtheorem{thm}{Theorem}[section]
\newtheorem*{thm*}{Theorem}
\newtheorem{prop}[thm]{Proposition}
\newtheorem{lem}[thm]{Lemma}
\newtheorem{cor}[thm]{Corollary}
\newtheorem{conj}[thm]{Conjecture}

\newtheorem{quest}[thm]{Question}

\theoremstyle{definition}
\newtheorem{defn}[thm]{Definition}
\newtheorem{definition}[thm]{Definition}

\theoremstyle{remark}
\newtheorem{rem}[thm]{Remark}
\newtheorem*{claim*}{Claim}
\newtheorem{remark}[thm]{Remark}

\numberwithin{equation}{section}

\title{Green functions \\ for positive-depth Deligne--Lusztig induction}

\author{Charlotte Chan}
\address{Department of Mathematics, University of Michigan, 2074 East Hall, 530 Church Street, Ann Arbor, MI 48105, USA.}
\email{charchan@umich.edu}

\author{Masao Oi}
\address{Department of Mathematics, National Taiwan University, Astronomy Mathematics Building 5F, No.\ 1, Sec.\ 4, Roosevelt Rd., Taipei 10617, Taiwan}
\email{masaooi@ntu.edu.tw}

\begin{document}

\begin{abstract}
    Under a largeness assumption on the size of the residue field, we give an explicit description of the positive-depth Deligne--Lusztig induction of unramified elliptic pairs $(\bfT,\theta)$. When $\theta$ is regular, we show that positive-depth Deligne--Lusztig induction gives a geometric realization of Kaletha's Howe-unramified regular $L$-packets. This is obtained as an immediate corollary of a very simple ``litmus test'' characterization theorem which we foresee will have interesting future applications to small-$p$ constructions. We next define and analyze Green functions of two different origins: Yu's construction (algebra) and positive-depth Deligne--Lusztig induction (geometry). Using this, we deduce a comparison result for arbitrary $\theta$ from the regular setting. As a further application of our comparison isomorphism, we prove the positive-depth Springer hypothesis in the $0$-toral setting and use it to give a geometric explanation for the appearance of orbital integrals in supercuspidal character formulae.
\end{abstract}


\maketitle


\tableofcontents

\newpage

\section{Introduction}\label{sec:intro}

The representation theory of finite groups of Lie type play a central role in the representation theory of $p$-adic groups, and as such, Deligne--Lusztig varieties are important objects which arise in a wide range of geometric contexts related to the Langlands program. In recent years, there has been a surge of interest in geometric realizations of Langlands phenomena, especially where strong conditions on level structures can be relaxed; one particular developing area is on establishing a Deligne--Lusztig theory extending the classical depth-zero setting to arbitrary integral depth. The first definition of \textit{positive-depth Deligne--Lusztig induction} is due to Lusztig \cite{Lus04}, who defined a functor
\begin{equation*}
    R_{\bbT_r}^{\bbG_r} \from \Z[\Irr(\bbT_r(\F_q))] \to \Z[\Irr(\bbG_r(\F_q))]
\end{equation*}
where $\bbG_r$ and $\bbT_r$ denote the $r$th jet schemes
of a connected reductive group $\bbG$ and a maximal torus $\bbT \subset \bbG$ defined over $\F_q$. This was extended to the mixed-characteristic setting by Stasinski \cite{Sta09} and further generalized in the context of parahoric subgroups of $p$-adic groups by the first author and Ivanov \cite{CI21-RT}. This functor recovers classical Deligne--Lusztig induction when $r = 0$. 

In general, much less is known about the $r>0$ setting, but there have been many rapid developments \cite{Sta09,CS17,CI21-RT,CO25,CS23,DI24,Cha24,INT24,Nie24}. These works either study endomorphism properties (e.g., for what characters $\theta$ is $R_{\bbT_r}^{\bbG_r}(\theta)$ irreducible?) as in \cite{Sta09,CI21-RT,DI24,INT24,Cha24} or study $R_{\bbT_r}^{\bbG_r}$ from the perspective of explicit algebraic constructions \cite{CS17,CO25,CS23,Nie24,IN25} (e.g., what is its relationship to G\'erardin \cite{Ger75-LNM}, or more generally Yu \cite{Yu01}?). These works study $R_{\bbT_r}^{\bbG_r}$ under at least one of the following  constraints: Coxeter or elliptic torus, specific choice of the Borel, genericity conditions on $\theta$. 

Let $\bfT$ be an unramified elliptic maximal torus of a connected reductive group $\bfG$ defined over a non-archimedean local field $F$ with residual characteristic $p$ and residual cardinality $q$. Let $\theta$ be a smooth character of $\bfT(F)$. 
In the present paper, we resolve for large $q$, the explicit comparison between $R_{\bbT_r}^{\bbG_r}(\theta)$ and types for representations of $p$-adic groups. We  describe exhaustion results for Kim--Yu types \cite{KY17} (see Sections \ref{subsec:types}, \ref{subsec:geom types}, and \ref{subsec:KY exhaust}), but for the purpose of the introduction, let us focus on supercuspidal types \cite{Yu01}. To obtain our results, we enact a three-step plan:
\begin{quote}
\begin{enumerate}[label={Step \arabic*.}]
    \item Prove a characterization theorem for representations associated to $\theta$ with trivial Weyl group stabilizer (Theorem \ref{thm:intro litmus})
    \item Compare $R_{\bbT_r}^{\bbG_r}(\theta)$ for $\theta$ with trivial Weyl group stabilizer and (a twist of) Kaletha's regular supercuspidal representations (Theorem \ref{thm:intro compare}).
    \item Analyze character formulae using Green functions to deduce the result for arbitrary $\theta$ (Theorem \ref{thm:intro exhaust}).
\end{enumerate}
\end{quote}
Step 1 requires us to make a largeness assumption on $q$ due to the analytic nature of the statement. Step 2 requires us to make a non-badness assumption on $p$ in order to guarantee the validity of Kaletha's construction of \textit{regular supercuspidal representations} \cite{Kal19}.

In some sense, we can think of this entire paper as justification for the applicability of our characterization theorem (Theorem \ref{thm:intro compare}). One comment is that the bound we give (\textit{Henniart's inequality} \eqref{eq:Henniart}) is sufficient but not necessary in general; in Section \ref{sec:small q}, we discuss some techniques that can be applied in small-$q$ settings. As an additional application of our results, we prove \textit{Springer's hypothesis} for a class of positive-depth representations (Section \ref{sec:springer}); we discuss this more at the end of the introduction.

Let us now expand on the three steps above, beginning with our comparison result for \textit{regular} $\theta$---those with trivial stabilizer in the Weyl group.


\begin{thm}[Theorems \ref{thm:reg comparison}]\label{thm:intro compare}
    Assume $p \neq 2$ is not bad for $\bfG$ and that $q$ is sufficiently large. For any regular $\theta$,
    \begin{equation*}
        \cInd(R_{\bbT_r}^{\bbG_r}(\theta)) \cong \pi_{(\bfT,\theta \cdot \varepsilon^{\ram})}^{\Yu},
    \end{equation*}
    where $\varepsilon^{\ram}$ is a quadratic character associated to $(\bfT,\theta)$ and $\pi_{(\bfT,\theta \cdot \varepsilon^{\ram})}^{\Yu}$ is the supercuspidal representation of $\bfG(F)$ arising from Kaletha's reparametrization of Yu's construction. In particular, the assignment
    \begin{equation*}
        (\bfT,\theta) \mapsto \cInd(R_{\bbT_r}^{\bbG_r}(\theta))
    \end{equation*}
    is well behaved with respect to the local Langlands correspondence (see Theorem \ref{thm:geom packet}).
\end{thm}

The first introduction of the quadratic character $\varepsilon^{\ram}$ is due to DeBacker--Spice \cite{DS18}, where they prove that for Howe-unramified 0-toral supercuspidal representations---the easiest class of supercuspidal representations of arbitrary positive depth---Adler's construction \cite{Adl98} (which generalizes to \cite{Yu01}) is not well behaved with respect to the local Langlands correspondence \textit{unless} one twists the parametrization by $\varepsilon^{\ram}$. The question of how $\varepsilon^{\ram}$ is related to the functor $R_{\bbT_r}^{\bbG_r}$ was first investigated by the authors in \cite{CO25}---there, we established Theorem \ref{thm:intro compare} in exactly the same context as \cite{DS18} (see \cite[Theorem 7.2, Theorem 8.2]{CO25}). This result pioneered a number of directions, one of which was the revelation that positive-depth Deligne--Lusztig induction seemed compatible with Langlands phenomena on the nose, without requiring any external twists.

In \cite{Kal19}, Kaletha established a remarkable reparametrization of a large proportion of the supercuspidal representations constructed by Yu. Kaletha's parametrization relies on establishing that any $\theta$ with trivial Weyl group stabilizer has a \textit{Howe factorization}, which then can be used to construct a datum to which Yu associates a supercuspidal representation; these are \textit{regular supercuspidal representations}. Kaletha extends DeBacker--Spice's definition of $\varepsilon^{\ram}$ and constructs $L$-packets of regular supercuspidal $L$-packets, which for unramified $\bfT$ have the same shape as the $L$-packets established by DeBacker--Spice: the elements are parametrized by $(\bfT,\theta \cdot \varepsilon^{\ram})$. Stability and endoscopic character identities for these $L$-packets was recently established (under additional hypothesis on $F$) by Fintzen--Kaletha--Spice \cite{FKS23}. Moreover, Fintzen--Kaletha--Spice define a twisted Yu construction which exactly encodes $\varepsilon^{\ram}$. The contribution of Theorem \ref{thm:intro compare} is then that positive-depth Deligne--Lusztig induction gives a geometric realization of Howe-unramified regular supercuspidal $L$-packets.

Although Theorem \ref{thm:intro compare} specializes to \cite[Theorem 7.2, 8.2]{CO25}, the methodology in \cite{CO25} relies more on fine representation-theoretic structure than the methodology in the present paper (see Remark \ref{rem:toral comparison} for more comments). The methodology in the present paper is \textit{analytic} in nature. Theorem \ref{thm:intro compare} is in fact obtained as a corollary of the following characterization criterion, which is completely independent of representation-theoretic constructions. Let $G_{\x,0}$ be the parahoric subgroup of $\bfG(F)$ determined by $\bfT \hookrightarrow \bfG$.

\begin{thm}[Theorem \ref{thm:vreg characterization}]\label{thm:intro litmus}
    Assume $q$ is sufficiently large. For any regular $\theta$, there exists at most one irreducible $G_{\x,0}$-representation $\pi$ of finite depth whose character satisfies
    \begin{equation*}
        \Theta_\pi(\gamma) = c \cdot \sum_{w} \theta^w(\gamma) \qquad \text{for all very regular $\gamma\in G_{\x,0}$}.
    \end{equation*}
\end{thm}

Here, \textit{very regular} essentially means that the image of $\gamma$ is regular semisimple in the reductive quotient of $G_{\x,0}$ and $c$ is a sign constant. 
The proof of Theorem \ref{thm:intro litmus} is analytic, short, and simple: the key ingredient is the Cauchy--Schwarz inequality. 
In fact, our approach is heavily inspired by Henniart's arguments in \cite[Section 2.6]{Hen92}, in which he discusses the case of $\GL_{n}$. Over the last few years, the authors have investigated generalizations of Henniart's characterization arguments (see Remark \ref{rem:litmus}); in short, our finding is that Henniart's arguments remain remarkably effective even in much broader contexts, i.e., for any unramified $p$-adic reductive group.

The requirement on $q$ is governed by \textit{Henniart's inequality} \eqref{eq:Henniart}. Up to a first approximation, one can think of this as a requirement on the proportion of regular elements in the finite-field torus associated to $\bfT$. (When $\bfG$ is split and the parahoric associated to $\bfT$ is hyperspecial, this is true on the nose: one can replace the denominator of the left-hand side of \eqref{eq:Henniart} with the number of non-regular elements of $\bbG_0(\F_q)$ which lie in $\bbT_0(\F_q)$.) As such, the strictness of this largeness condition depends on the nature of the chosen torus $\bfT$; for example, when $\bfT$ is Coxeter, the constraint is actually very mild (see Table \ref{table:Henniart}). In general, one can calculate what Henniart's inequality \eqref{eq:Henniart} demands on $q$---by hand for classical types and by using L\"ubeck's tables \cite{Lub19} for exceptional types.

The depth-zero version of Theorem \ref{thm:intro litmus} is one of the key steps in establishing a characterization theorem of this type for regular supercuspidal representations; this was done by the authors in the recent paper \cite{CO23-sc}. Because of this, it is of interest to understand the nature of the failures of Theorem \ref{thm:intro litmus} for small $q$, especially in the depth-zero setting. In Section \ref{sec:small q} we investigate this setting. In Section \ref{subsec:G2}, we present a case-by-case analysis of every small-$q$ setting excluded by the largeness condition required for Theorem \ref{thm:intro litmus} in the case of $G_{2}$. This analysis revealed the following remarkable observation: that small-$q$ violations of Theorem \ref{thm:intro litmus} seem to exclusively come from \textit{unipotent} representations. In Section \ref{sec:rss and unip}, we explain this phenomenon with a general fact: for certain $\theta$, one can formulate a characterization theorem for non-unipotent representations \textit{for all $q$} (Theorem \ref{thm:litmus with unip}).

Let us return now to the setting of Theorem \ref{thm:intro litmus}. It is worth noting that Theorem \ref{thm:intro litmus} works equally well for \textit{all} $p$; in particular, this includes characters $\theta$ for $\bfT(F)$ which do not have a Howe factorization. The potential for new supercuspidal representations to arise in this setting has been discussed in \cite[Section 7.4]{CO25} ($p=2$) and more generally in \cite{INT24} (specifically in the interesting setting of having no Howe factorization). As progress develops on both sides---understanding $R_{\bbT_r}^{\bbG_r}(\theta)$ for non-Howe-factorizable $\theta$ (e.g., the recent work of \cite{INT24}) and constructing supercuspidal representations in arbitrary representations (e.g., the recent work of Fintzen--Schwein \cite{FS25})---Theorem \ref{thm:intro litmus} will continue to be a powerful tool in establishing comparison isomorphisms for ``100\% 
of the time.'' 

Up to this point, we have explained that Theorem \ref{thm:intro litmus} implies Theorem \ref{thm:intro compare}; these are results for $\theta$ which have trivial Weyl group stabilizer. In Sections \ref{sec:DL} and \ref{sec:theta general}, we drop this assumption and consider $\theta$ in general. 

In Section \ref{sec:DL} we finally come to the namesake of this paper. 
Section \ref{sec:DL} is devoted giving a character formula for $R_{\bbT_r}^{\bbG_r}(\theta)$ in the spirit of the Deligne--Lusztig character formula: the formula is given in terms of $\theta$ and a \textit{positive-depth Green function} $Q_{\bbT_r}^{\bbG_r}(\theta_+)$---the restriction of the character of $R_{\bbT_r}^{\bbG_r}(\theta)$ to the unipotent elements. We establish, using the inner product formula for \cite{Cha24}, orthogonality relations for positive-depth Green functions. This structural result may be of independent interest. In the present paper, we give two applications of the study of positive-depth Green functions, which we now describe.

We establish in Section \ref{sec:Yu green} that in fact certain virtual representations ${}^\circ \tau_{(\bfT,\theta \cdot \epsilon)}$ arising from Yu's construction \textit{also} have character formulae in terms of intrinsically defined Green functions $\sfQ_{\bbT_r}^{\bbG_r}(\theta_+)$. Our result (Theorem \ref{thm:green Yu}) seems to be the first to investigate structure of this type. 

In our context, Theorem \ref{thm:green Yu} has the immediate implication that the coincidence of $R_{\bbT_r}^{\bbG_r}(\theta)$ and ${}^\circ \tau_{(\bfT,\theta \cdot \epsilon)}$ for arbitrary $\theta$ immediately follows from the coincidence of their unipotent restrictions $Q_{\bbT_r}^{\bbG_r}(\theta_+)$ and $\sfQ_{\bbT_r}^{\bbG_r}(\theta_+)$. On the other hand, for large $q$, the latter identity is a direct consequence of Theorem \ref{thm:intro compare}! (See Theorem \ref{thm:Q comparison}.) In other words, studying Green functions allows us to deduce results for arbitrary $\theta$ from results for $\theta$ with trivial Weyl group stabilizer (see Theorem \ref{thm:gen comparison}). 
This logical structure to obtain results for all $\theta$ from results for regular $\theta$ was inspired by Lusztig's work \cite{Lus90}. 
In particular, we obtain:

\begin{thm}[Corollary \ref{cor:inner product depth zero}]\label{thm:intro exhaust}
    Assume $q$ is sufficiently large. Let $\Psi = (\vec \bfG, \vec \phi, \vec r, \x, \rho_0)$ be a cuspidal Yu datum where $\x$ corresponds to an unramified elliptic maximal torus $\bfT$. Then
    \begin{equation*}
        \langle \pi_\Psi^{\Yu-\FKS}, R_{\bbT_r}^{\bbG_r}(\theta) \rangle = \langle \rho_{0}, R_{\bbT_0}^{\bbG_0^0}(\phi_{-1}) \rangle,
    \end{equation*}
    where $\theta = \prod_{i=-1}^d (\phi_i)|_{T_0}$. Here, $\pi_\Psi^{\Yu-\FKS}$ denotes the supercuspidal obtained from the Fintzen--Kaletha--Spice twist of Yu's construction.
\end{thm}

Ultimately, one would want to have a more satisfying answer as to why a theorem like Theorem \ref{thm:intro compare} holds, without a ``cheat'' criterion like Theorem \ref{thm:intro litmus}. The present paper has been in preparation for the last few years, much of this delay due to awaiting the resolution of the results in \cite{Cha24}. In the meantime, there have been rapid developments of a different flavor towards understanding $R_{\bbT_r}^{\bbG_r}(\theta)$'s relation to algebraic constructions. These methods are geometric in nature, and were first pioneered by the very interesting recent work of Chen--Stasinski \cite{CS23} (sequel to \cite{CS17}). This was followed by Nie's subsequent generalization \cite{Nie24} of this work using some structural methods of \cite{Cha24} and the very recent work of Ivanov--Nie \cite{IN25}. In particular, Nie \cite{Nie24} proved, with no assumption on $q$, that $R_{\bbT_r}^{\bbG_r}(\theta)$ exhausts all Howe-unramified supercuspidal types. At present, works in this direction do not quite pin down the appearance of the twist $\varepsilon^{\ram}$.


We end this introduction with an application of Theorem \ref{thm:intro compare} to computing character formulae of supercuspidal representations. 
The study of Green functions for finite groups of Lie type is a crucial ingredient in establishing stability and endoscopic character identities for depth-zero supercuspidal representations \cite{KV06,DR09}. \textit{Springer's hypothesis}, proved by Kazhdan \cite{Kaz77} by elementary methods and in \cite[Appendix A]{KV06} using character sheaves, posits that Green functions can be expressed as the Fourier transform of the delta function of the coadjoint orbit of a(ny) semisimple element of the relevant (dual) Lie algebra. Assume for the rest of this introduction that $F$ has characteristic $0$ and large residue characteristic to guarantee convergence of the exponential map. Using $p$-adic harmonic analysis, DeBacker--Reeder \cite{DR09} use the Springer hypothesis to compute the character formula for depth-zero supercuspidal representations; the orbital integrals that appear exactly come from the aforementioned delta functions.

Fintzen--Kaletha--Spice establish for positive-depth regular supercuspidal representations an analogous character formula \cite[Theorem 4.3.5]{FKS23}, which is based on the preceding works by Adler--DeBacker \cite{AS09}, DeBacker--Spice \cite{DS18}, and Spice \cite{Spi18, Spi21}.
The orbital integrals that appear there now come from positive-depth data; their role in the character formula is established using the classical Springer hypothesis together with positive-depth algebraic methods. It is natural to ask: can this formula be obtained using a \textit{positive-depth Springer's hypothesis}? In Section \ref{sec:springer}, we prove this for $0$-toral characters $\theta$ (Theorem \ref{thm:pos Springer}); this proof relies on the recent construction of character sheaves on parahoric subgroups due to Bezrukavnikov and the first author \cite{BC24}. We then use $p$-adic harmonic analysis (Section \ref{subsec:DR}), generalizing DeBacker--Reeder's depth-zero context, to then obtain a \textit{geometric} proof of the character formula for $0$-toral supercuspidal representations. We remark that the constraint on $0$-toral characters only comes into play due to the fact that the computation of the trace-of-Frobenius of the relevant character sheaf is at present only written down for $0$-toral $\theta$ \cite[Theorem 10.9]{BC24}.

\subsection*{Acknowledgements}

The authors thank Jessica Fintzen and Ju-Lee Kim for helpful discussions about Kim--Yu types, Meinolf Geck for helpful discussions about Frank L\"ubeck's database of information on semisimple classes, and Yakov Varshavsky for helpful discussions about Springer's hypothesis. 
The authors also thank Guy Henniart for his encouragement.
The first author was partially supported by NSF grants DMS-1802905, DMS-2101837, NSF grant DMS-2401114, and a Sloan Research Fellowship. She also thanks the Sydney Mathematical Research Institute for excellent working conditions. 
The second author was partially supported by JSPS KAKENHI Grant Number JP20K14287 and JP24K16899, Hakubi Project at Kyoto University, and the Yushan Young Fellow Program, Ministry of Education, Taiwan. This paper resolves several open threads the authors began discussing in 2019; we gratefully thank the many conferences and climbing gyms that have supported our collaboration.

\section{Notation and assumptions}\label{sec:notations}

Let $F$ be a non-archimedean local field with finite residue field $\cO_F/\mfp_F \cong \F_q$ of prime characteristic $p$, where we write $\cO_{F}$ and $\mfp_{F}$ for the ring of its integers and the maximal ideal, respectively.
We let $F^{\ur}$ denote the maximal unramified extension of $F$.
We write $\Gamma$ for the absolute Galois group of $F$.

For an algebraic variety $\bfJ$ over $F$, we denote the set of its $F$-valued points by $J$.
When $\bfJ$ is an algebraic group, we write $\bfZ_{\bfJ}$ for its center.
We use the bold fraktur letter for the Lie algebra of an algebraic group, e.g., $\bmfj=\Lie\bfJ$.
Similarly to the group case, we denote the set of $F$-valued points of $\bmfj$ by $\mfj$.

For a connected reductive group  $\bfJ$ over $F$, let $\cB(\bfJ,F)$ (resp.\ $\cB^{\red}(\bfJ,F)$) denote the enlarged (resp.\ reduced) Bruhat--Tits building of $\bfJ$ over $F$.
For a point $\x\in\cB(\bfJ,F)=\cB^{\red}(\bfJ,F)\times X_{\ast}(\bfZ_{\bfJ})_{\R}$, we write $\bar{\x}$ for the image of $\x$ in $\cB^{\red}(\bfJ,F)$, and $J_{\bar{\x}}$ for the stabilizer of $\bar{\x}$ in $J$.
We define $\widetilde{\R}$ to be the set $\R\sqcup\{r+\mid r\in\R\}\sqcup\{\infty\}$ with a natural order.
Then, for each point $\x\in\cB(\bfJ,F)$, the Moy--Prasad filtration $\{J_{\x,r}\}_{r\in\widetilde{\R}_{\geq0}}$ and also $\{\mfj_{\x,r}\}_{r\in\widetilde{\R}_{\geq0}}$ for the Lie algebra $\mfj$ of $\bfJ$ are defined.
For any $r,s\in\widetilde{\R}_{\geq0}$ satisfying $r<s$, we write $J_{\x,r:s}$ (resp.\ $\mfj_{\x,r:s}$) for the quotient $J_{\x,r}/J_{\x,s}$ (resp.\ $\mfj_{\x,r}/\mfj_{\x,s}$).

Suppose that $\bfT$ is a tamely ramified maximal torus of $\bfJ$.
By fixing a $T$-equivariant embedding of $\cB(\bfT,F)$ into $\cB(\bfJ,F)$, we may regard $\cB(\bfT,F)$ as a subset of $\cB(\bfJ,F)$.
Then, for any point $\x\in\cB(\bfJ,F)$, the property that ``$\x$ belongs to the image of $\cB(\bfT,F)$'' does not depend on the choice of such an embedding (see the second paragraph of \cite[Section 3]{FKS23} for details).
For any point $\x\in\cB(\bfJ,F)$ which belongs to $\cB(\bfT,F)$, we have $T_{\mathrm{b}}\subset G_{\x}$, where $T_{\mathrm{b}}$ denotes the maximal bounded subgroup of $T$.
When $\bfT$ is elliptic in $\bfJ$, the image of $\cB(\bfT,F)$ in $\cB^{\red}(\bfJ,F)$ consists of only one point.
When the image of a point $\x\in\cB(\bfJ,F)$ in $\cB^{\red}(\bfJ,F)$ coincides with this point (or, equivalently, $\x$ belongs to $\cB(\bfT,F)$), we say that $\x$ is associated with $\bfT$.
Note that, in this case, we have $T \subset J_{\bar{\x}}$.

For any element $x$ of a group $H$, we let ${}^{x}(-)$ or $(-)^{x}$ denote the conjugation by $y$.
For example, for any $x,y\in H$, we put ${}^{x}y\coloneqq xyx^{-1}$ and ${y}^{x}\coloneqq x^{-1}yx$; for any $x\in H$ and a subset $H'\subset H$, we put ${}^{x}H'\coloneqq xH'x^{-1}$ and $H'^{x}\coloneqq x^{-1}H'x$; for any representation $\rho$ of a subgroup $H'$ of $H$ and $x\in H$, we define a representation ${}^{x}\rho$ of ${}^{x}H'$ (resp.\ $\rho^{x}$ of $H'{}^{x}$) by ${}^{x}\rho({}^{x}y)\coloneqq \rho(y)$ (resp.\ $\rho^{x}(y^{x})\coloneqq \rho(y)$).
For any group $H$ and its subgroups $H_{1}$ and $H_{2}$, we put
\[
N_{H}(H_{1},H_{2})\coloneqq \{n\in H \mid {}^{n}H_{1}\subset H_{2}\}.
\]
When $H_{1}=H_{2}$, we shortly write $N_{H}(H_{1})\coloneqq N_{H}(H_{1},H_{2})$.

For any two virtual representations $\rho_{1}, \rho_{2}$ of any finite group $G$, we let $\langle\rho_{1},\rho_{2}\rangle_{G}$ denote the scalar product of $\rho_{1}$ and $\rho_{2}$, i.e.,
\[
\langle\rho_{1},\rho_{2}\rangle_{G}
\coloneqq
\frac{1}{|G|}\sum_{g\in G}\Theta_{\rho_{1}}(g)\cdot\overline{\Theta_{\rho_{2}}(g)}.
\]
We often simply write $\langle\rho_{1},\rho_{2}\rangle$ when the group $G$ is clear from the context.

\subsection{Assumptions on $p$ and $q$}\label{subsec:assumptions}

In this paper, we impose several assumptions on $p$ and $q$ depending on the context.
These assumptions will be explicitly stated each time, at the risk of being repetitive.
But for the sake of clarity, we also summarize them here.
\begin{enumerate}
    \item (Assumption on $p$.) Whenever we appeal to the theory of algebraic construction of supercuspidal representations, we assume that $p$ is odd and not bad for $\bfG$. When we discuss regular supercuspidal representations, we additionally assume $p \nmid |\pi_1(\bfG_{\der})|$ and $p \nmid |\pi_1(\widehat \bfG_{\der})|$ (which are implied by the non-badness of $p$ unless a component of type $A_n$ is present). This is to guarantee the existence of a Howe factorization for all $\theta$ and simplifies our exposition on regular supercuspidal representations, but it is not necessary. In general, our methods hold for any tame pair $(\bfT,\theta)$ that has a Howe factorization; we call such pairs \textit{Howe-factorizable}.

    \item (Assumption on $p$.) In Section \ref{sec:springer}, we assume that $p$ is sufficiently large so that a logarithm map is available; see Section \ref{subsec:log} for details.
    \item (Assumption on $q$.) In Theorem \ref{thm:vreg characterization}, we assume that $q$ is sufficiently large so that the Henniart inequality \eqref{eq:Henniart} holds. This theorem is crucial for our comparison results between algebraic and geometric constructions of supercuspidal representations, hence we always assume this in the context of comparison results.
    \item (Assumption on $q$.) In Theorem \ref{thm:pos Springer}, we assume that $q$ is large so that a maximal torus $\bbT_0$ of a connected reductive group $\bbG_0$ possesses at least one regular semisimple element. This assumption is much weaker than the previous one.
    \item (Assumption on $q$.) In Section \ref{sec:theta general}, we assume that $q$ is sufficiently large so that a maximal torus $\bbT_0$ of a connected reductive group $\bbG_0$ possesses at least one regular character. This is equivalent to the existence of a regular semisimple element in the dual torus of $\bbT_0$, hence the resulting restriction on $q$ is very close to the previous one, but not exactly the same; see Remark \ref{rem:reg-depth-zero}.
\end{enumerate}

\section{Very regular elements and characterization of parahoric representations}\label{sec:litmus}

In this section, we let $\bfG$ be a connected reductive group over $G$ and fix an unramified elliptic maximal torus $\bfT$ of $\bfG$.
Let $\x\in\cB(\bfG,F)$ be a point associated to $\bfT$.
Note that then we have $TG_{\x,0}=Z_{\bfG}G_{\x,0}$ (\cite[Lemma 7.1.1]{Kal11-0}).

Recall the notion of \textit{very regular elements}, as introduced in \cite[Definition 4.2]{CO25}.

\begin{definition}
    A regular semisimple element $\gamma \in TG_{\x,0}$ is \textit{unramified very regular} if
    \begin{enumerate}[label=\textbullet]
        \item the connected centralizer $\bfT_\gamma$ of $\gamma$ in $\bfG$ is an unramified maximal torus, and
        \item $\alpha(\gamma) \not\equiv 1 \pmod{\mathfrak{p}_{\overline{F}}}$ for any root $\alpha$ of $\bfT_\gamma$ in $\bfG$.
    \end{enumerate}
\end{definition}

Let $T_{0}=T\cap G_{\x,0}$ be the parahoric subgroup of $T$.
Let $\bbG_{0}$ denote the reductive quotient of the special fiber of the parahoric subgroup scheme of $G$ associated to $\x$; then $\bbG_{0}(\F_{q})\cong G_{\x,0:0+}$ (see \cite[Section 3.2]{MP96}).
If we similarly define $\bbT_{0}$, then $\bbT_{0}$ is an elliptic maximal torus of $\bbG_{0}$ (see \cite[Section 3.2]{CO25}).
Note that the Weyl groups $W_{G_{\x,0}}(T_{0})$ and $W_{\bbG_{0}}(\bbT_{0})$ are naturally identified through the reduction map $G_{\x,0}\twoheadrightarrow G_{\x,0:0+}\cong\bbG_{0}(\F_{q})$ (see \cite[Lemma 3.4.10 (2)]{Kal19}).
In fact, more generally, we have the following identifications for any unramified very regular element $\gamma\in G_{\x,0}$:
\[
    W_{G_{\x,0}}(T_{\gamma,0},T_{0})
    \xrightarrow{\cong}
    W_{G_{\x,0:r+}}(T_{\gamma,0:r+},T_{0:r+})
    \xrightarrow{\cong}
    W_{\bbG_{0}(\F_q)}(\bbT_{\gamma,0},\bbT_{0}),
\]
where $\bbT_{\gamma,0}$ denotes the elliptic maximal torus determined by $\bfT_{\gamma}$ in the same way as $\bbT_{0}$.
(Whenever these sets are nonempty, $\gamma$ must be elliptic.)
Here, both the maps are natural ones induced by the reductions maps $G_{\x,0}\twoheadrightarrow G_{\x,0:r+} \twoheadrightarrow G_{\x,0:0+}$.
It is explained that the composition of these two maps is bijective in \cite[Lemma 6.22]{CO23-sc}; the same proof can be applied to showing that the first map is bijective.
We also note that $W_{G_{\x,0}}(T_{\gamma,0},T_{0})$ equals $W_{G_{\x,0}}(\bfT_{\gamma},\bfT)$.

We write $G_{\x,0,\vreg}$ (resp.\ $T_{0,\vreg}$) for the set of unramified very regular elements of $G_{\x,0}$ (resp.\ $T_{0}$).
Let $\bbT_{0}(\F_{q})_{\vreg}$ be the image of $T_{0,\vreg}$ under the reduction map $T_{0}\twoheadrightarrow T_{0:0+}\cong\bbT_{0}(\F_{q})$.
We put $\bbT_{0}(\F_{q})_{\nvreg}:=\bbT_{0}(\F_{q})\smallsetminus\bbT_{0}(\F_{q})_{\vreg}$. We point out that the elements of $\bbT_0(\F_q)_{\vreg}$ are regular semisimple in $\bbG_0(\F_q)$, but in general it is possible for $\bbT_0(\F_q)_{\vreg}$ to be strictly smaller than the set of regular elements in $\bbT_0(\F_q)$.

The main result of this section is an elementary characterization of certain irreducible representations of parahoric subgroups.

\begin{thm}\label{thm:vreg characterization}
    Let $\theta \from T_0 \to \C^\times$ be a smooth character with trivial $W_{G_{\x,0}}(\bfT)$-stabilizer. Assume that $q$ is large enough so that
    \begin{equation}\tag{$*$}\label{eq:Henniart}
        \frac{|\bbT_{0}(\F_{q})|}{|\bbT_{0}(\F_{q})_{\nvreg}|} > 2\cdot|W_{\bbG_{0}}(\bbT_{0})(\F_{q})|.
    \end{equation}
    Then there exists at most one irreducible smooth representation $\pi$ of $G_{\x,0}$ whose character satisfies
    \begin{equation*}
        \Theta_\pi(\gamma) = c \cdot \sum_{w \in W_{G_{\x,0}}(\bfT_{\gamma},\bfT)} \theta^w(\gamma) \qquad \text{for all $\gamma \in G_{\x,0,\vreg}$,}
    \end{equation*}
    for some sign constant $c \in \{\pm 1\}$.
\end{thm}

\begin{proof}
    Assume that $\pi,\pi'$ are irreducible smooth representations of $G_{\x,0}$ such that
    \begin{align*}
        \Theta_\pi(\gamma) &= c \cdot \sum_{w \in W_{G_{\x,0}}(\bfT_{\gamma},\bfT)} \theta^w(\gamma) \qquad \text{for all $\gamma \in G_{\x,0,\vreg}$,}\\
        \Theta_{\pi'}(\gamma) &= c' \cdot \sum_{w \in W_{G_{\x,0}}(\bfT_{\gamma},\bfT)} \theta^w(\gamma) \qquad \text{for all $\gamma \in G_{\x,0,\vreg}$,}
    \end{align*}
    where the two sign constants $c,c'$ \textit{a priori} may not agree. Our task is to prove $\pi \cong \pi'$.
    Since $\pi,\pi'$ are both irreducible, this is equivalent to proving 
    \begin{equation}\label{eq:nonzero}
        \langle \pi, \pi' \rangle \neq 0.
    \end{equation}
    Here, note that all $\theta$, $\pi$, and $\pi'$ are irreducible smooth representations of compact groups, hence we can choose $r\in\R_{\geq0}$ such that they factors through the quotient by the depth $(r+)$-th filtration.
   Thus we are choosing the smallest such $r$ and taking the inner product in the finite group $G_{\x,0:r+}$, i.e., 
   \[
   \langle\pi,\pi'\rangle=\frac{1}{|G_{\x,0:r+}|}\sum_{\gamma\in G_{\x,0:r+}}\Theta_{\pi}(\gamma)\cdot\overline{\Theta_{\pi'}(\gamma)}.
   \]
   We let $(G_{\x,0:r+})_{\vreg}$ be the image of $G_{\x,0,\vreg}$ and $(G_{\x,0:r+})_{\nvreg}:=G_{\x,0:r+}\smallsetminus(G_{\x,0:r+})_{\vreg}$.
   If we define 
   \[
   \langle\pi,\pi'\rangle_{\mathrm{(n)vreg}}
   :=
   \frac{1}{|G_{\x,0:r+}|}\sum_{\gamma\in (G_{\x,0:r+})_{\mathrm{(n)vreg}}}\Theta_{\pi}(\gamma)\cdot\overline{\Theta_{\pi'}(\gamma)},\quad
   \]
   then obviously we have $\langle \pi, \pi' \rangle = \langle \pi, \pi' \rangle_{\nvreg} + \langle \pi, \pi' \rangle_{\vreg}$.
   Hence, to prove \eqref{eq:nonzero}, it suffices to show
    \begin{equation}\label{eq:inequality}
        |\langle \pi, \pi' \rangle_{\nvreg}| < |\langle \pi, \pi' \rangle_{\vreg}|.
    \end{equation}

    To this end, we begin by estimating $|\langle \pi, \pi' \rangle_{\nvreg}|$. By the Cauchy--Schwarz inequality, we have
    \begin{equation*}
        |\langle \pi, \pi' \rangle_{\nvreg}| \leq \langle \pi, \pi \rangle_{\nvreg}^{\frac{1}{2}} \cdot \langle \pi', \pi' \rangle_{\nvreg}^{\frac{1}{2}}.
    \end{equation*}
    By assumption, we have $\langle \pi, \pi \rangle = 1 = \langle \pi', \pi' \rangle$ and $\langle \pi, \pi \rangle_{\vreg} = \langle \pi', \pi' \rangle_{\vreg}$. 
    It therefore follows that $\langle \pi, \pi \rangle_{\nvreg} = \langle \pi', \pi' \rangle_{\nvreg}.$ So we obtain the inequality
    \begin{equation*}
        |\langle \pi, \pi' \rangle_{\nvreg}| \leq \langle \pi, \pi \rangle_{\nvreg}.
    \end{equation*}
    To show \eqref{eq:inequality}, it is enough to show that $\langle \pi, \pi \rangle_{\vreg} > \frac{1}{2}$; indeed, if this inequality holds, then $\langle \pi, \pi \rangle_{\nvreg} < \frac{1}{2}$ and
    \begin{equation*}
        |\langle \pi, \pi' \rangle_{\nvreg}| \leq \langle \pi, \pi \rangle_{\nvreg} < \tfrac{1}{2} < \langle \pi, \pi \rangle_{\vreg} = |\langle \pi, \pi' \rangle_{\vreg}|.
    \end{equation*}
    Hence the theorem follows from the following claim:
    \begin{equation}\label{eq:claim}
        \text{If \eqref{eq:Henniart} holds, then $\langle \pi, \pi \rangle_{\vreg} > \tfrac{1}{2}$}.
    \end{equation}
    
    By the assumption on $\Theta_{\pi}$, we have
    \[
     \langle \pi, \pi \rangle_{\vreg}
     =
     \frac{1}{|G_{\x,0:r+}|} \sum_{\gamma \in (G_{\x,0:r+})_{\vreg}} \sum_{w,w' \in W_{G_{\x,0}}(\bfT_{\gamma},\bfT)} \theta^{w}(\gamma) \cdot \overline{\theta^{w'}(\gamma)}.
    \]
    Note that the index set of the first sum can be replaced with 
    \[
    (G_{\x,0:r+})'_{\vreg}
    :=
    \{\gamma\in(G_{\x,0:r+})_{\vreg} \mid \text{conjugate to an element of $(T_{0:r+})_{\vreg}$}\}.
    \]
    By this definition, we have a surjective map
    \[
    G_{\x,0:r+}\times (T_{0:r+})_{\vreg}\twoheadrightarrow (G_{\x,0:r+})'_{\vreg}\colon (g,t)\mapsto {}^{g}t.
    \]
    We claim that each fiber of this map is of order $|N_{G_{\x,0:r+}}(T_{0:r+})|$.
    Indeed, fix $(g,t)\in G_{\x,0:r+}\times (T_{0:r+})_{\vreg}$ and let us take another element $(g',t')\in G_{\x,0:r+}\times (T_{0:r+})_{\vreg}$ satisfying ${}^{g}t={}^{g'}t'$.
    Then we have ${}^{g^{-1}g'}t'=t$.
    In other words, if we take lifts $g_0, g'_0\in G_{\x,0}$ of $g,g'\in G_{\x,0:r+}$ and also lifts $t_0, t'_0\in T_{0,\vreg}$ of $t,t'\in (T_{0:r+})_\vreg$, then there exists an element $g_{r+}\in G_{\x,r+}$ satisfying ${}^{g_0^{-1}g'_0}t'_0=t_0g_{r+}$.
    By \cite[Lemma 5.1]{CO25}, there exists $k_{r+}\in G_{\x,r+}$ satisfying ${}^{k_{r+}}t_0=t_0g_{r+}$ (the case where $r=0$ is discussed in \textit{loc.\ cit.}, but the same argument works).
    Hence we get ${}^{g_0^{-1}g'_0}t'_0={}^{k_{r+}}t_0$, which implies that $k_{r+}^{-1}g_0^{-1}g'_0\in N_{G_{\x,0}}(\bfT)$ since both $t'_0$ and $t_0$ are regular semisimple elements of $T$.
    Thus we get $g'\in g\cdot N_{G_{\x,0:r+}}(T_{0:r+})$.
    Conversely, for any $n\in N_{G_{\x,0:r+}}(T_{0:r+})$, the element $(g',t'):=(gn,{}^{n^{-1}}t)$ satisfies ${}^{g}t={}^{g'}t'$.
    Thus we get the claim.

    By this observation, the above expression of $\langle\pi,\pi\rangle_\vreg$ equals
    \begin{align*}
        & \frac{1}{|G_{\x,0:r+}|} \cdot \frac{|G_{\x,0:r+}|}{|N_{G_{\x,0:r+}}(T_{0:r+})|} \sum_{t \in (T_{0:r+})_{\vreg}} \sum_{w,w' \in W_{G_{\x,0}}(\bfT)} \theta^{w}(t) \cdot \overline{\theta^{w'}(t)} \\
        &= \frac{1}{|N_{G_{\x,0:r+}}(T_{0:r+})|} \sum_{w,w' \in W_{G_{\x,0}}(\bfT)}\left(|T_{0:r+}| \langle \theta^{w}, \theta^{w'} \rangle_{T_{0:r+}} - \sum_{t \in (T_{0:r+})_{\nvreg}}\theta^{w}(t) \cdot \overline{\theta^{w'}(t)}\right) \\
        &\geq \frac{1}{|N_{G_{\x,0:r+}}(T_{0:r+})|} \left(|T_{0:r+}| \cdot |W_{G_{\x,0}}(\bfT)| - |(T_{0:r+})_{\nvreg}| \cdot |W_{G_{\x,0}}(\bfT)|^2\right).
    \end{align*}
    By using the identifications of Weyl groups $W_{G_{\x,0}}(\bfT)=W_{G_{\x,0}}(T_{0})\cong W_{G_{\x,0:r+}}(T_{0:r+})\cong W_{\bbG_{0}(\F_q)}(\bbT_{0})$ mentioned before, we see that the right-hand-most side of the above equalities is
    \[
    1 - \frac{|(T_{0:r+})_{\nvreg}|}{|T_{0:r+}|} \cdot |W_{\bbG_{0}}(\bbT_{0})(\F_{q})|.
    \]
    Finally noting that the ratio $|(T_{0:r+})_{\nvreg}|/|T_{0:r+}|$ is equal to $|\bbT_{0}(\F_{q})_{\nvreg}|/|\bbT_{0}(\F_{q})|$, this proves the theorem. 
    
    We remark that since $|\bbT_{0}(\F_{q})_{\nvreg}|$ is a polynomial in $q$ of degree strictly less than the degree of $|\bbT_{0}(\F_{q})|$, then the ratio $\frac{|\bbT_{0}(\F_{q})_{\nvreg}|}{|\bbT_{0}(\F_{q})|}$ tends to 0 as $q \to \infty$. 
    In particular, since $|W_{\bbG_{0}}(\bbT_{0})(\F_{q})|$ is a number independent of $q$, we see that $\langle \pi, \pi \rangle_{\vreg} > \frac{1}{2}$ for $q$ sufficiently large;  moreover, by the same argument as in the proof of \cite[Proposition 5.8]{CO25}, we can show that there exists a constant $C$ which depends only on the absolute rank of $\bfG$ such that \eqref{eq:Henniart} is satisfied when $q>C$.
\end{proof}

\begin{rem}\label{rem:litmus}
This is not the first instance of a characterization theorem based on values on very regular elements. The first such characterization appeared for $p$-adic $\GL_n$ in work of Henniart \cite{Hen92,Hen93}. Investigating generalizations of Henniart's characterization theorem has led the two authors of this paper on a multi-year journey:
\begin{enumerate}
    \item We used a more structured version of this argument in \cite{CO25} to achieve comparison theorems between algebraic and geometric parametrizations of certain supercuspidal representations. A much more primitive incarnation of the ideas in \cite{CO25} appeared in the techniques in \cite[\S6.3]{Cha20}.
    \item We established a characterization theorem for certain irreducible representations of finite groups of Lie type in \cite[\S4.1]{CO23-sc}. We used this to establish a characterization theorem for supercuspidal representations of $p$-adic groups \cite[\S8]{CO23-sc} by introducing yet another characterization theorem: one that allowed a reduction to depth-zero data.
\end{enumerate}
As is the case for Henniart's pioneering $\GL_n$ work, it is our belief that characterization theorems of this type will continue to find many interesting applications. One could view the entirety of the present paper as justification of the impact of Theorem \ref{thm:vreg characterization}.
\end{rem}

\begin{rem}
    In our characterization theorem in the finite field setting \cite[\S4.1]{CO23-sc}, the maximal torus is allowed to be arbitrary (i.e., no ellipticity assumption).
    It should therefore be reasonable to expect that Theorem \ref{thm:vreg characterization} can also be formulated for non-elliptic $\bfT$.
    Since the majority of the above proof is quite formal (especially, the estimate using the Cauchy--Schwarz inequality), we expect that essentially the same proof works.
    The only subtleties are the lemmas on unramified very regular elements such as \cite[Lemma 6.22]{CO23-sc} or \cite[Lemma 5.1]{CO25}, where the ellipticity assumption on $\bfT$ is heavily used.
\end{rem}

\section{Regular supercuspidal representations and their characters}\label{sec:Yu}

Let $\bfG$ be a tamely ramified connected reductive group over $F$. In this section, we assume $p \neq 2$ is not bad for $\bfG$ and additionally that  $p \nmid |\pi_1(\bfG_{\der})| \cdot |\pi_1(\widehat \bfG_{\der})|$ (see the comment in Section \ref{subsec:assumptions}). In the following, we write $(\bfT,\theta)$ to mean a pair of an $F$-rational maximal torus of $\bfG$ and a smooth character $\theta\colon T\rightarrow\C^\times$.
We simply refer to such $(\bfT,\theta)$ a \textit{pair (of $\bfG$)} as long as there is no risk of confusion, depending on the context.
We say that a pair $(\bfT,\theta)$ is \textit{tame} if $\bfT$ is tamely ramified, \textit{unramified} if $\bfT$ is unramified, \textit{elliptic} if $\bfT$ is elliptic in $\bfG$.
When the depth of $\theta$ is $r$, we say that $(\bfT,\theta)$ is of \textit{depth $r$}.
We additionally have the adjectives \textit{regular}, \textit{toral}, and \textit{0-toral} concerning conditions on $\theta$; see Section \ref{subsec:rsc}. If $(\bfT,\theta)$ has a Howe factorization, we say that it is \textit{Howe-factorizable}. Note that the imposed assumptions on $p$ imply that any tame pair is Howe-factorizable.

\subsection{Regular supercuspidal representations}\label{subsec:rsc}

In \cite{Kal19}, Kaletha introduced a class of supercuspidal representations called \textit{regular supercuspidal representations} (\cite[Definition 3.7.3]{Kal19}).
He proved that the equivalence classes of regular supercuspidal representations are parametrized by the $G$-conjugacy classes of \textit{tame elliptic regular pairs} of $\bfG$ (\cite[Corollary 3.7.10]{Kal19}).
A tame elliptic regular pair of $\bfG$ is a tame elliptic pair $(\bfT,\theta)$ whose $\theta$ satisfies certain conditions (see \cite[Definition 3.7.5]{Kal19} for the details); under the assumptions on $p$ imposed above, regularity is equivalent to the condition that $\theta$ has trivial stabilizer under the action of the Weyl group $W_G(\bfT)$.
For a tame elliptic regular pair $(\bfT,\theta)$, we write $\pi_{(\bfT,\theta)}$ for the corresponding regular supercuspidal representation under Kaletha's parametrization:
\begin{align*}
\{\text{tame elliptic regular pairs of $\bfG$}\}/\text{$G$-conj.} &\xrightarrow{1:1} \{\text{regular s.c.\ rep'ns of $G$}\}/{\sim}\\
(\bfT,\theta)&\mapsto\pi_{(\bfT,\theta)}.
\end{align*}

The regular supercuspidal representations arising from unramified elliptic regular pairs $(\bfT,\theta)$ are particularly of interest.
We call regular supercuspidal representations arising from such pairs \textit{Howe-unramified regular supercuspidal representations} \cite[Definition 3.6]{CO25}.

Kaletha's construction of regular supercuspidal representation relies on Yu's theory of \textit{tame supercuspidal representations} (\cite{Yu01}).
Recall that Yu's theory attaches to each \textit{cuspidal $\bfG$-datum} $\Psi=(\vec{\bfG},\vec{\phi},\vec{r},\x,\rho'_{0})$, which consists of
\begin{itemize}
\item
a sequence $\vec{\bfG}=(\bfG^{0}\subsetneq\bfG^{1}\subsetneq\cdots\subsetneq\bfG^{d}=\bfG)$ of tame twisted Levi subgroups,
\item
a point $\x$ of the Bruhat--Tits building $\cB(\bfG^{0},F)$ of $\bfG^{0}$,
\item
a sequence $\vec{r}=(0\leq r_{0}<\cdots<r_{d-1}\leq r_{d})$ of real numbers,
\item
a sequence $\vec{\phi}=(\phi_{0},\ldots,\phi_{d})$ of characters $\phi_{i}$ of $G^{i}$, and
\item
an irreducible $\rho_{0}'$ cuspidal representation of $G^{0}_{\bar{\x}}$
\end{itemize}
satisfying certain conditions, an irreducible supercuspidal representation $\pi_{\Psi}$ of $G$ of the form $\cInd_{K}^{G}\rho_{\Psi}^{\Yu}$, where $K$ and $\rho_{\Psi}^{\Yu}$ are an explicit open compact-mod-center subgroup of $G$ and its irreducible representation (see \cite{Yu01} and also \cite[Section 3.1]{CO25} for the details).
In this paper, we refer to a cuspidal $\bfG$-datum as a \textit{Yu datum} of $\bfG$. 
We say that a Yu datum is \textit{Howe-unramified} if $\bfG^0$ (hence every $\bfG^i$) splits over an unramified extension.

To each tame elliptic regular pair $(\bfT,\theta)$ of $\bfG$, Kaletha first associates a Yu datum $\Psi$ (\cite[Proposition 3.7.8]{Kal19}) and then defines the regular supercuspidal representation $\pi_{(\bfT,\theta)}$ to be $\pi_{\Psi}$.
The process of constructing a Yu datum $\Psi$ from a tame elliptic regular pair $(\bfT,\theta)$ is called the \textit{Howe factorization}.
In fact, the Howe factorization itself is considerable for any tame pair $(\bfT,\theta)$ (especially, possibly non-elliptic and non-regular).
More precisely, any tame pair $(\bfT,\theta)$ defines 
\begin{itemize}
    \item tame twisted Levi subgroups $(\bfG^{-1}=\bfT\subset\bfG^{0}\subsetneq\cdots\subsetneq\bfG^{d}=\bfG)$ and
    \item characters $\vec{\phi}=(\phi_{-1},\ldots,\phi_d)$ of $G^i$ satisfying certain ``genericity'' conditions
\end{itemize} 
such that $\theta=\prod_{i=-1}^d \phi_i|_T$ (see \cite[Section 3.6]{Kal19} or \cite[Section 3.3]{CO25} for the details).

\begin{defn}[{\cite[Definition 3.7]{CO25}}]\label{defn:toral}
Let $(\bfT,\theta)$ be a tame pair of $\bfG$ with sequence of tame twisted Levi subgroups $\vec{\bfG}=(\bfG^{0}\subsetneq\cdots\subsetneq\bfG^{d}=\bfG)$ associated by the Howe factorization.
\begin{enumerate}
\item
We call $\theta$ a \textit{toral} character if $\bfG^{0}=\bfT$.
\item
We call $\theta$ a \textit{$0$-toral} character if $d=1$ and $\bfG^{0}=\bfT$.
\end{enumerate}
\end{defn}

When a regular supercuspidal representation $\pi_{(\bfT,\theta)}$ arises from a tame elliptic toral (resp.\ $0$-toral) regular pair $(\bfT,\theta)$ of $\bfG$, we say that $\pi_{(\bfT,\theta)}$ is a \textit{toral} (resp.\ \textit{$0$-toral}) supercuspidal representation of $G$.

\subsection{Adler--DeBacker--Spice character formula}\label{subsec:ADS}

In \cite{CO25}, we obtained an explicit character formula of Howe-unramified regular supercuspidal representations (not necessarily toral) on unramified very regular elements as a consequence of the theory of Adler--DeBacker--Spice (\cite{AS08, AS09, DS18}).
We review it here.

In the following, let $\pi_{(\bfT,\theta)}$ be a Howe-unramified regular supercuspidal representation, i.e., $(\bfT,\theta)$ is an unramified elliptic regular pair of $\bfG$.

Let $\Psi=(\vec{\bfG},\vec{\theta},\vec{r},\x,\rho'_{0})$ be a Howe factorization of $(\bfT,\theta)$ so that by construction \cite{Kal19} we have $\pi_{(\bfT,\theta)}\cong\cInd_{K}^{G}\rho_{\Psi}^{\Yu}$.
Then Yu's subgroup $K$ associated to $\Psi$ is contained in $G_{\bar{\x}}$ and contains a subgroup $\cc K$.
Moreover, Yu's representation $\rho_{\Psi}^{\Yu}$ of $K$ is realized as the induction of a representation $\cc\rho_{\Psi}^{\Yu}$ of $\cc K$ to $K$.
On the other hand, the subgroup $TG_{\x,0}$ of $G_{\bar{\x}}$ contains $\cc{K}$.
We let $\tau_{\Psi}^{\Yu}$ (resp.\ $\cc\tau_{\Psi}^{\Yu}$) denote the induction of $\rho_{\Psi}^{\Yu}$ to $G_{\bar{\x}}$ (resp.\ $\cc\rho_{\Psi}^{\Yu}$ to $TG_{\x,0}$).
The situation is summarized in the following diagram (each dashed arrow indicates the induction; see \cite[Section 3.4]{CO25} for the details):
\[
\begin{tikzcd}[row sep=15pt, column sep=25pt]
G_{\bar{\x}} & TG_{\x,0} & \tau_{\Psi}^{\Yu} & \cc\tau_{\Psi}^{\Yu} \\
K & \cc{K} & \rho_{\Psi}^{\Yu} & \cc\rho_{\Psi}^{\Yu}
  \arrow[phantom, from=1-2, to=1-1, "\supset"]
  \arrow[phantom, from=2-2, to=2-1, "\supset"]
  \arrow[phantom, from=2-1, to=1-1, "\cup"]
  \arrow[phantom, from=2-2, to=1-2, "\cup"]
  \arrow[dotted, from=2-3, to=1-3]
  \arrow[dotted, from=1-4, to=1-3]
  \arrow[dotted, from=2-4, to=2-3]
  \arrow[dotted, from=2-4, to=1-4]
\end{tikzcd}
\]

For convenience, we also introduce a further intermediate subgroup $K\subset K_{\sigma}\subset G_{\bar{\x}}$ and its representation $\sigma_{\Psi}^{\Yu}$ defined by inducing $\rho_{\Psi}^{\Yu}$ (see \cite[Section 3.1]{CO25} and also \cite[Section 2]{AS09}).

Following \cite{DS18}, we introduced a sign character $\varepsilon^{\ram}[\theta]\colon T\rightarrow\C^{\times}$ in \cite[Definition 4.6]{CO25}.
Our character formula of $\pi_{(\bfT,\theta)}$ on unramified very regular elements is stated using $\varepsilon^{\ram}[\theta]$ as follows:

\begin{prop}[{\cite[Proposition 4.11]{CO25}}]\label{prop:AS-vreg}
Let $\gamma\in TG_{\x,0}$ be an unramified very regular element.
\begin{enumerate}
\item
If $\gamma$ is not $TG_{\x,0}$-conjugate to an element of $T$, then we have $\Theta_{\cc\tau_{\Psi}^{\Yu}}(\gamma)=0$.
\item
If $\gamma$ is $TG_{\x,0}$-conjugate to an element of $T$ (in this case, we assume that $\gamma$ itself belongs to $T$ by taking conjugation), we have
\[
\Theta_{\cc\tau_{\Psi}^{\Yu}}(\gamma)
=
(-1)^{r(\bfG^{0})-r(\bfT)+r(\bfT,\theta)}
\sum_{w\in W_{G_{\x,0}}(\bfT)}
\varepsilon^{\ram}[\theta]({}^{w}\gamma)\theta({}^{w}\gamma),
\]
where $r(\bfG^{0})$, $r(\bfT)$, and $r(\bfT,\theta)$ are integers determined by $(\bfT,\theta)$ (see \cite[Proposition 4.9]{CO25} for the details) and $W_{G_{\x,0}}(\bfT)\coloneqq N_{G_{\x,0}}(\bfT)/T_{0}$.
\end{enumerate}
\end{prop}

Here let us mention the modified construction of tame supercuspidal representations due to Fintzen--Kaletha--Spice \cite{FKS23}.
They introduced a sign character $\epsilon_{\Psi}\colon K\rightarrow\C^{\times}$ associated to a Yu datum $\Psi$ (see \cite[Definition 4.1.10]{FKS23}).
Then they proved that $\rho_{\Psi}^{\Yu}\otimes\epsilon_{\Psi}$ compactly induces to an irreducible supercuspidal representation of $G$.
Let us write $\pi_{\Psi}^{\FKS}$ for this representation:
\[
\pi_{\Psi}^{\FKS}\coloneqq \cInd_{K}^{G}(\rho_{\Psi}^{\Yu}\otimes\epsilon_{\Psi}).
\]
We write $\cc\rho_{\Psi}^{\FKS}:=\cc\rho_{\Psi}^{\Yu}\otimes\epsilon_{\Psi}$ and $\rho_{\Psi}^{\FKS}:=\rho_{\Psi}^{\Yu}\otimes\epsilon_{\Psi}$.
We define $\cc\tau_{\Psi}^{\FKS}:=\Ind_{\cc K}^{TG_{\x,0}}(\cc\rho_{\Psi}^{\FKS})$ and $\tau_{\Psi}^{\FKS}:=\Ind_{K}^{G_{\bar{\x}}}(\rho_{\Psi}^{\FKS})$.
We also define $\sigma_{\Psi}^{\FKS}:=\Ind_{K}^{K_{\sigma}}(\rho_{\Psi}^{\FKS})$.

Suppose that a Yu datum $\Psi$ is regular and corresponds to an unramified elliptic regular pair $(\bfT,\theta)$.
Let us write $\pi_{(\bfT,\theta)}^{\FKS}$ for $\pi_{\Psi}^{\FKS}$.
In this case, the identical proof as in Proposition \ref{prop:AS-vreg} works to compute the character of $\cc\tau_{\Psi}^{\FKS}$.
Only the difference between the formulas for $\Theta_{\cc\tau_{\Psi}^{\Yu}}$ and $\Theta_{\cc\tau_{\Psi}^{\FKS}}$ is that the summand is multiplied by $\epsilon_{\Psi}({}^{w}\gamma)$.
The point is that the restriction of the character $\epsilon_{\Psi}$ to $T$ decomposes into the product of three characters $\epsilon_{\sharp,\x}$, $\epsilon_{\flat}$, and $\epsilon_{f}$ introduced in \cite[Definition 3.1]{FKS23}.
In fact, $\epsilon_{\sharp,\x}$ and $\epsilon_{f}$ are nothing but the quantities $\epsilon^{\ram}$ and $\epsilon_{f,\ram}$ introduced in \cite{Kal19}, which appear naturally in the character formula of tame supercuspidal representations (see \cite[Remark 3.3]{FKS23}).
With our notation, $\epsilon^{\ram}$ is equal to $\varepsilon^{\ram}[\theta]$.
On the other hand, $\epsilon_{\flat}$ is trivial whenever $T$ is unramified.
Therefore, in summary, Proposition \ref{prop:AS-vreg} can be restated as follows:

\begin{prop}\label{prop:AS-vreg-FKS}
Let $\gamma\in TG_{\x,0}$ be an unramified very regular element.
\begin{enumerate}
\item
If $\gamma$ is not $TG_{\x,0}$-conjugate to an element of $T$, then we have $\Theta_{\cc\tau_{\Psi}^{\FKS}}(\gamma)=0$.
\item
If $\gamma$ is $TG_{\x,0}$-conjugate to an element of $T$ (in this case, we assume that $\gamma$ itself belongs to $T$ by taking conjugation), we have
\[
\Theta_{\cc\tau_{\Psi}^{\FKS}}(\gamma)
=
(-1)^{r(\bfG^{0})-r(\bfT)+r(\bfT,\theta)}
\sum_{w\in W_{G_{\x,0}}(\bfT)}
\theta({}^{w}\gamma).
\]
\end{enumerate}
\end{prop}

\subsection{Regular Kim--Yu types}\label{subsec:types}

Yu's theory of tame supercuspidal representations has been furthermore generalized by Kim--Yu \cite{KY17} by appealing to the theory of types (\cite{BK98}).
They introduced the notion of a ``datum'', which generalizes a Yu datum and associated to it a type in the sense of Bushnell-Kutzko; here let us call them a \textit{Kim--Yu datum} and a \textit{Kim--Yu type}.
We do not review the details of their definition and construction; see \cite[Section 7]{KY17} and also \cite[Section 4]{AFMO24-depth0}.

Let us consider an unramified regular pair $(\bfT,\theta)$ of $\bfG$  of depth $r$.
Note that $\bfT$ is not necessarily elliptic; the regularity of $\theta$ is in the sense of \cite[Definition 3.7.5]{Kal19} without the ellipticity assumption.

Kaletha's Howe factorization gives rise to a sequence of tame twisted Levi subgroups $\vec{\bfG}=(\bfG^0\subsetneq\cdots\subsetneq\bfG^d=\bfG)$ such that $\bfT$ is a maximally unramified maximal torus of $\bfG^0$.
Define the Levi subgroup $\bfM^i$ of $\bfG^i$ to be the centralizer of the maximal split subtorus of $\bfT$ in $\bfG^i$.
Note that then $\bfT$ is elliptic in $\bfM^0$ and $\vec{\bfM}=(\bfM^0\subseteq\cdots\subseteq\bfM^d=\bfM)$ forms a ``generalized'' (i.e., $\bfM^i$ might be equal to $\bfM^{i+1}$) tame twisted Levi sequence of $\bfM$ (see \cite[Section 2]{KY17} for details).
In particular, the pair $(\bfT,\theta)$ also defines a tame elliptic regular pair in $\bfM$, hence we get an associated cuspidal $\bfM$-datum $\Psi_M$ of $\bfM$ by \cite[Proposition 3.7.8]{Kal19} (see Section \ref{subsec:rsc}).
Moreover, one can upgrade it to a Kim--Yu datum $\Psi$ of $\bfG$ as in \cite[Section 7.2]{KY17}. Analogously to the setting of Section \ref{subsec:rsc}, we say a Kim--Yu datum is \textit{Howe-unramified} if $\bfG^0$ (hence every $\bfG^i$ and $\bfM^i$) splits over an unramified extension.


From this data, by Yu's and Kim--Yu's constructions \cite{Yu01}, we obtain a type $(K_{M,0},\rho_M)$ of $M$ and also a type $(K_0,\rho)$ of $G$ which is a $G$-cover of $(K_{M,0},\rho_M)$ in the sense of Bushnell--Kutzko (see \cite[Section 8]{BK98} and also \cite[Section 4.2]{KY17}).
Here we give two remarks about our convention.
First, we put $K_{M,0}:=K_{M}\cap M_{\x,0}$, where $K_M$ denotes the open subgroup of $M_{\bar{\x}}$ associated to the tame elliptic regular pair $(\bfT,\theta)$ of $\bfM$ as in the manner of Section \ref{subsec:ADS} (note that a subgroup constituting a type must be open and \textit{compact}).
Then we define the group $K_0$ using $K_{M,0}$ and the datum $\Sigma$ following \cite[Section 7.4]{KY17}.
Second, $\rho_M$ denotes the representation of $K_{M,0}$ \textit{twisted} by the sign character of Fintzen--Kaletha--Spice; with the notation in Section \ref{subsec:ADS}, $\rho_M$ is $\cc\rho^\FKS_{\Psi_M}|_{K_{M,0}}$ (not $\cc\rho^\Yu_{\Psi_M}|_{K_{M,0}}$).
The point is that it is also possible to find a suitable twist of the original (untwisted) Kim--Yu type so that it again becomes a ``quasi''-$G$-cover of the twisted type $(K_{M,0},\rho_M)$ of $M$.
This is what is discussed in \cite[Section 4]{AFMO24-depth0}; our $\rho$ denotes the representation of $K_0$ twisted by Adler--Fintzen--Mishra--Ohara.

\begin{rem}
In \cite[Remark 4.1.5]{AFMO24-depth0}, a very subtle point of the above twisted construction is discussed.
To be more precise, we write $\rho^\KY$ and $\rho_M^\Yu$ for the untwisted types associated to $\Psi$ and $\Psi_M$.
If we write $\rho=\rho^\AFMO$ and $\rho_M=\rho_M^\FKS$ for the twisted types by AFMO and FKS, respectively, then we have $\rho^\AFMO=\rho^\KY\otimes\epsilon^G$ and $\rho_M^\FKS=\rho_M^\Yu\otimes\epsilon^M$, where $\epsilon^M$ and $\epsilon^G$ are the sign characters of FKS.
In fact, what is proved in \cite{AFMO24-depth0} is that $\rho^\AFMO=\rho^\KY\otimes\epsilon^G$ is a quasi-$G$-cover of $\rho_M^\AFMO:=\rho_M^\Yu\otimes(\epsilon^G|_{K_{M,0}})$.
The problem is that $(\epsilon^G|_{K_{M,0}})$ could be different to $\epsilon^M$.
(Indeed, a counterexample is provided in \cite[Remark A.2.6]{AFMO24-depth0}).

However, this is not the case in our situation.
The possible discrepancy between $(\epsilon^G|_{K_{M,0}})$ and $\epsilon^M$ on $T_0$ comes from the characters $\epsilon_f$ and $\epsilon_\flat$ of $T_0$ (see \cite[Definition 3.1]{FKS23}), to which only ``potentially-ramified'' roots (i.e., itself or its associated restricted root is ramified) contribute.
Hence, as our torus $\bfT$ is unramified, at least we have that $(\epsilon^G|_{K_{M,0}})\equiv\epsilon^M$ on $T_0\subset K_{M,0}$.
Here, let $\Psi_M=(\vec{\bfM},\vec{\theta},\vec{r},\x,\rho'_{M,0})$ be a cuspidal $\bfM$-datum corresponding to $(\bfT,\theta)$.
Since $\theta$ is supposed to be regular, we may choose $\rho'_{M,0}$ to be (the inflation of) the regular Deligne--Lusztig representation $(-1)^{r(\bbM^0_0)-r(\bbT_0)}R_{\bbT_0}^{\bbM^0_0}(\theta_{-1})$.
Then twisting the type $\rho_M^\Yu$ by a character $(\epsilon^G)|_{K_{M,0}}$ or $\epsilon^M$ amounts to tensoring $(-1)^{r(\bbM^0_0)-r(\bbT_0)}R_{\bbT_0}^{\bbM^0_0}(\theta_{-1})$ with those characters.
However, since these tensor products are only determined by the restrictions of $(\epsilon^G)|_{K_{M,0}}$ or $\epsilon^M$ to $T_0$,  
we conclude that $(-1)^{r(\bbM^0_0)-r(\bbT_0)}R_{\bbT_0}^{\bbM^0_0}(\theta_{-1})\otimes(\epsilon^G|_{K_{M,0}})=(-1)^{r(\bbM^0_0)-r(\bbT_0)}R_{\bbT_0}^{\bbM^0_0}(\theta_{-1})\otimes\epsilon^M$, which implies that $\rho_M^\AFMO=\rho_M^\FKS$. (This twisting property can be found in \cite[Lemma 2.3.13]{GM20}; note that there, one must assume that the character is trivial on unipotent elements, which holds automatically for the sign character here since any unipotent element is of $p$-power order and we have assumed $p$ is odd.)
\end{rem}

We fix a parabolic subgroup $P$ of $G$ with Levi decomposition $P=MN$.
We put $P_{\x,0}:=G_{\x,0}\cap P$ and $N_{\x,0}:=G_{\x,0}\cap N$, hence we have $P_{\x,0}=M_{\x,0}N_{\x,0}$.
Using this decomposition, we define a depth-$r$ parabolic induction at parahoric level $I_{P_{\x,0}}^{G_{\x,0}}$ for any smooth representation of $M_{\x,0}$ trivial on $M_{\x,r+}$ by
\[
    I_{P_{\x,0}}^{G_{\x,0}}
    :=
    \Ind_{P_{\x,0}G_{\x,r+}}^{G_{\x,0}}\circ\Inf_{M_{\x,0}}^{P_{\x,0}G_{\x,r+}}
    \colon
    \mathrm{Rep}(M_{\x,0:r+})\rightarrow\mathrm{Rep}(G_{\x,0:r+}).
\]
Similarly, we define a depth-$r$ Jacquet functor at parahoric level $J_{P_{\x,0}}^{G_{\x,0}}$ to be the $N_{\x,0}$-coinvariant functor:
\[
    J_{P_{\x,0}}^{G_{\x,0}}
    :=
    (-)_{N_{\x,0}}
    \colon
    \mathrm{Rep}(G_{\x,0:r+})\rightarrow\mathrm{Rep}(M_{\x,0:r+}).
\]
Then the first adjoint theorem holds also in this context, i.e., for any smooth representations $\sigma_M$ of $M_{0:r+}$ and $\sigma$ of $G_{0:r+}$,
\[
    \Hom_{G_{\x,0}}\bigl(\sigma,I_{P_{\x,0}}^{G_{\x,0}}(\sigma_M)\bigr)
    \cong
    \Hom_{M_{\x,0}}\bigl(J_{P_{\x,0}}^{G_{\x,0}}(\sigma),\sigma_M\bigr).
\]

\begin{lem}\label{lem:BK}
We have an $M_{\x,0}$-equivariant surjective homomorphism
\[
    \Phi\colon
    J_{P_{\x,0}}^{G_{\x,0}}(\Ind_{K_0}^{G_{\x,0}}\rho)
    \twoheadrightarrow
    \Ind_{K_{M,0}}^{M_{\x,0}}\rho_M.
\]
\end{lem}

\begin{proof}
The proof is the same as (in fact, even easier than) \cite[Lemma 10.3]{BK98}.
For the sake of completeness, we reproduce it here.

We define a homomorphism
\[
\Phi\colon\Ind_{K_0}^{G_{\x,0}}\rho
\rightarrow
\Ind_{K_{M,0}}^{M_{\x,0}}\rho_M
\]
by
\[
\Phi(f)(m)
:=
\int_{N_{\x,0}} f(mn)\,dn,
\]
where $dn$ denotes the Haar measure of $N_{\x,0}$ with total volume $1$.
Note that the well-definedness follows from that $(K_0,\rho)$ is a quasi-$G$-cover of $(K_{M,0},\rho_M)$ (in the sense of Adler--Fintzen--Mishra--Ohara, \cite[Definition 3.3.2]{AFMO24-Hecke}), hence, in particular, $K_{M,0}=K_0\cap M$ and $\rho=\rho^{N_{\x,0}}=\rho_M$.

Since $\Phi$ is defined by averaging $\rho(n)(f)$ over $n\in N_{\x,0}$, it is obvious that $\Phi$ factors through the $N_{\x,0}$-coinvariant of $\Ind_{K_0}^{G_{\x,0}}\rho$.
Thus let us show the surjectivity of the map $\Phi$.

Let $\cB$ be a $\C$-basis of the representation space (say $V$) of $\rho$ and $\rho_M$.
We fix a set $\Gamma_M$ of representatives of $K_{M,0}\backslash M_{\x,0}$.
For any $(\gamma,v)\in \Gamma_M\times\cB$, we let $f^M_{(\gamma,v)}\colon M_{\x,0}\rightarrow V$ be the function such that the support is $K_{M,0}\gamma$ and $f^M_{(\gamma,v)}(\gamma)=v$.
Then $\{f^M_{(\gamma,v)}\}_{(\gamma,v)\in\Gamma_M\times\cB}$ forms a $\C$-basis of $\Ind_{K_{M,0}}^{M_{\x,0}}(\rho_M)$.

For $(\gamma,v)\in \Gamma_M\times\cB$, we let $f_{(\gamma,v)}\colon G_{\x,0}\rightarrow V$ be the function such that the support is $K_{0}m$ and $f^M_{(\gamma,v)}(\gamma)=v$.
Then $\Phi(f_{\gamma,v})$ is equal to $f^M_{(\gamma,v)}$ up to non-zero scalar multiple.
Indeed, we can easily check that $\Phi(f_{\gamma,v})(\gamma)$ equals $v$ up to non-zero scalar multiple by noting that $K_{0}\cap N_{\x,0}$ acts trivially on $\rho$.
Moreover, for any $m\in M_{\x,0}$,
\[
\Phi(f_{\gamma,v})(m)
=
\int_{N_{\x,0}} f_{\gamma,v}(mn)\,dn
\]
is not zero only when $mn$ belongs to $K_0\gamma$ for some $n\in N_{\x,0}$.
As $mn=(mnm^{-1})m$ and $M_{\x,0}$ normalizes $N_{\x,0}$, this implies that $m\gamma^{-1}\in N_{\x,0}K_0$.
However, since the product map $N\times M\times \overline{N}\rightarrow G$ is injective and induces a bijection $(K_0\cap N)\times K_{M,0}\times (K_0\cap\overline{N})\rightarrow K_0$), we have $N_{\x,0}K_0\cap M=K_{M,0}$, which implies that $m\in K_{M,0} \gamma$.
This completes the proof.
\end{proof}


\begin{lem}\label{lem:red-depth-zero}
    The induced representation $\Ind_{K_0}^{G_{\x,0}}\rho$ is irreducible.
\end{lem}

\begin{proof}
    Recall that, by Mackey decomposition, $\Ind_{K_0}^{G_{\x,0}}\rho$ is irreducible if and only if the intertwining group $I_{G_{\x,0}}(K_0,\rho)$ of $(K_0,\rho)$ in $G_{\x,0}$ equal to $K_0$ itself.
    The inclusion $K_0\subset I_{G_{\x,0}}(K_0,\rho)$ is obvious, hence we have to show its converse.
    By \cite[Section 8]{KY17}, we have $I_{G}(K_0,\rho)\subset K_0G^{0}K_0$ (note that the support of Hecke algebra is nothing but the intertwining group in $G$; see, e.g., \cite[p590]{BK98}).
    Hence 
    \[
        I_{G_{\x,0}}(K_0,\rho)
        =G_{\x,0}\cap I_{G}(K_0,\rho) 
        \subset G_{\x,0}\cap K_0G^{0}K_0
    \]
    As $K_0$ is contained in $G_{\x,0}$, we have 
    \[
    G_{\x,0}\cap K_0G^{0}K_0
    =K_0(G_{\x,0}\cap G^{0})K_0
    =K_0G^0_{\x,0}K_0
    =K_0. \qedhere
    \]
\end{proof}

\begin{prop}\label{prop:types induction}
The representation $\Ind_{K_0}^{G_{\x,0}}\rho$ is contained in $I_{P_{\x,0}}^{G_{\x,0}}(\Ind_{K_{M,0}}^{M_{\x,0}}\rho_M)$.
\end{prop}

\begin{proof}
By Lemma \ref{lem:BK}, we have
\[
    \Hom_{M_{\x,0}}\bigl(J_{P_{\x,0}}^{G_{\x,0}}(\Ind_{K_0}^{G_{\x,0}}\rho),\Ind_{K_{M,0}}^{M_{\x,0}}\rho_M\bigr)
    \neq0.
\]
Hence the adjunction implies that 
\[
    \Hom_{G_{\x,0}}\bigl(\Ind_{K_0}^{G_{\x,0}}\rho,I_{P_{\x,0}}^{G_{\x,0}}(\Ind_{K_{M,0}}^{M_{\x,0}}\rho_M)\bigr)
    \neq0.
\]
Since $\Ind_{K_0}^{G_{\x,0}}\rho$ is irreducible by Lemma \ref{lem:red-depth-zero}, this means that $I_{P_{\x,0}}^{G_{\x,0}}(\Ind_{K_{M,0}}^{M_{\x,0}}\rho_M)$ contains $\Ind_{K_0}^{G_{\x,0}}\rho$.
\end{proof}

\begin{rem}
As suggested by the above proof, we can furthermore show that $\Ind_{K_0}^{G_{\x,0}}\rho$ and $I_{P_{\x,0}}^{G_{\x,0}}(\Ind_{K_{M,0}}^{M_{\x,0}}\rho_M)$ are isomorphic whenever the latter is irreducible.
Also note that Proposition \ref{prop:types induction} holds if we replace the depth-$r$ parabolic induction with the depth-$r'$ parabolic induction for any $r'>r$.
\end{rem}

\section{Geometric realization of Howe-unramified regular supercuspidal $L$-packets}\label{sec:reg comparison}

In this section, we use the main results of Sections \ref{sec:litmus} and \ref{sec:Yu} to deduce the compatibility between positive-depth Deligne--Lusztig induction and algebraic constructions of representations of $p$-adic groups. After recalling known results on positive-depth Deligne--Lusztig induction in Section \ref{subsec:parahoricDL}, we apply the characterization theorem (Theorem \ref{thm:vreg characterization}) to obtain a geometric realization of regular supercuspidal $L$-packets. For example, due to known stability results for these $L$-packets for mixed-characteristic $F$ (\cite{DS18,FKS23}), we obtain that positive-depth Deligne--Lusztig induction preserves stability.

Analogously to the depth-zero setting, the supercuspidal part of the relationship between positive-depth Deligne--Lusztig induction and  representations of $G$ comes from the setting that the torus $\bfT$ is taken to be \textit{elliptic}. In Section \ref{subsec:geom types}, we use Section \ref{subsec:types} to give a description of positive-depth Deligne--Lusztig induction for arbitrary unramified $\bfT$.

\subsection{Positive-depth Deligne--Lusztig representations}\label{subsec:parahoricDL}

We review a generalization of Deligne--Lusztig varieties for $G_{\x,0}$ defined in \cite{CI21-RT}. In the setting that $F$ has equal characteristic and $G_{\x,0} = \bbG(\cO_F)$ for a reductive group $\bbG$ over $\F_q$, these varieties were originally defined and studied by Lusztig in \cite{Lus04}. Later this set-up was generalized using the Greenberg functor by Stasinski in \cite{Sta09} for $F$ mixed characteristic. In this section, we make impose no assumptions on $p$ and $q$.

Let $\bfT \subset \bfG$ be an unramified maximal torus (not necessarily elliptic) defined over $F$ and let $\x \in \cB(\bfG, F)$ be a point in the apartment of $\bfT$.
Say $\bfT$ splits over the degree-$n$ unramified extension $F_n$ of $F$ and let $\bfU$ be the unipotent radical of a $F_n$-rational Borel subgroup of $\bfG_{F_n}$ containing $\bfT_{F_{n}}$. 
Following, \cite[Section 2.6]{CI21-RT}, for $r \in \Z_{\geq 0}$, we have group schemes $\bbT_r \subset \bbG_r$ defined over $\F_q$ such that $\bbG_r(\F_q) = G_{\x,0:r+}$ and $\bbT_r(\F_q) = T_{0:r+}$, and a group scheme $\bbU_r \subset \bbG_{r, \F_{q^n}}$. (We warn the reader that there is a change of convention between our present work and the papers \cite{Lus04, Sta09, CI21-RT}: in our normalization, $\bbG_0$ is a reductive group over a finite field.) 
Let $\sigma \from \bbG_r \to \bbG_r$ denote the geometric Frobenius endomorphism associated to the $\F_q$-rational structure on $\bbG_r$.

\begin{definition}\label{def:Xr}
For $r \in \Z_{\geq 0}$, we define the following $\F_{q^n}$-subscheme of $\bbG_r$ associated to the triple $(\bfT, \bfU, \x)$:
\begin{equation*}
X_{\bbT_r \subset \bbG_r} \coloneqq \{x \in \bbG_r \mid x^{-1} \sigma(x) \in \bbU_r\}.
\end{equation*}
\end{definition}

By \cite[Lemma 3.1]{CI21-RT}, $X_{\bbT_r \subset \bbG_r}$ is separated, smooth, and of finite type over $\F_{q^n}$. 
For $(g,t) \in G_{\x,0:r+} \times T_{0:r+}$ and $x \in X_{\bbT_r \subset \bbG_r}$, the assignment $(g,t) * x = g x t$ defines an action of $G_{\x,0:r+} \times T_{0:r+}$ on $X_{\bbT_r \subset \bbG_r}$ which pulls back to an action of $G_{\x,0} \times T_0$.

We fix a prime number $\ell$ which is not equal to $p$.
By functoriality, the cohomology groups $H_c^i(X_{\bbT_r \subset \bbG_r}, \overline \Q_\ell)$ are representations of $G_{\x,0} \times T_0$.
We take a smooth character $\theta \from T \to \C^\times$ trivial on $T_{r+}$ and regard it as a $\Qlb^{\times}$-valued character by fixing an isomorphism $\C\cong\Qlb$. 
Then the subspace $H_c^i(X_{\bbT_r \subset \bbG_r}, \overline \Q_\ell)_\theta$ of $H_c^i(X_{\bbT_r \subset \bbG_r}, \overline \Q_\ell)$ on which $T_0$ acts by multiplication by $\theta|_{T_{0}}$ is a representation of $G_{\x,0}$.
We define a virtual representation $R_{\bbT_r, \bbU_r}^{\bbG_r}(\theta)$ of $G_{\x,0}$ with $\Qlb$-coefficient by
\begin{equation*}
R_{\bbT_r, \bbU_r}^{\bbG_r}(\theta) \coloneqq \sum_{i \geq 0} (-1)^i H_c^i(X_{\bbT_r \subset \bbG_r}, \overline \Q_\ell)_\theta.
\end{equation*}
We regard this as a $\C$-representation using the fixed isomorphism $\C\cong\Qlb$; note that the resulting $\C$-representation is independent of the choices of $\ell$ and the isomorphism $\C\cong\Qlb$. 

In the special case that $r = 0$, the variety $X_0$ is an affine fibration over a classical Deligne--Lusztig variety and hence their cohomology is the same up to an even degree shift. Hence the $R_{\bbT_0, \bbU_0}^{\bbG_0}(\theta)$ is exactly the usual Deligne--Lusztig representation of the finite reductive group $\bbG_0(\F_q)$ attached to the character $\theta|_{T_{0}}$ of $\bbT_0(\F_q)$.

We present two important results on the positive-depth Deligne--Lusztig representations, both of which are $r>0$ analogues of known $r=0$ theorems. 
The first result is the $r>0$ version of the Mackey formula for Deligne--Lusztig representations \cite[Theorem 6.8]{DL76} (in the case where a maximal torus is elliptic): 

\begin{thm}[{\cite[Theorem 6.2]{Cha24}}]\label{thm:Cha24}
    Let $(\bfT,\theta)$ be an unramified elliptic Howe-factorizable pair of $\bfG$ with triple $(\bfT,\bfU,\x)$ of depth $\leq r$.
    For any other unramified pair $(\bfT',\theta')$ with triple $(\bfT',\bfU',\x)$ of depth $\leq r$, 
    \begin{equation*}
        \langle R_{\bbT_r,\bbU_r}^{\bbG_r}(\theta), R_{\bbT'_{r},\bbU'_r}^{\bbG_{r}}(\theta') \rangle_{\bbG_r(\F_{q})} = |\{w \in W_{\bbG_{r}(\F_{q})}(\bbT_{r}, \bbT'_{r}) \mid \theta^{w} = \theta'\}|,
    \end{equation*}
    where we put $W_{\bbG_{r}(\F_{q})}(\bbT_{r}, \bbT'_{r}):=N_{\bbG_{r}(\F_{q})}(\bbT_{r}, \bbT'_{r})/\bbT_{r}(\F_{q})$.
\end{thm}

\begin{rem}\label{rem:scalar product}
\begin{enumerate}
\item
Note that it is assumed that $(\bfT,\theta)$ is ``split-generic'' in \cite[Theorem 6.2]{Cha24}, but this is automatically satisfied when $\bfT$ is elliptic in $\bfG$ (see \cite[Definition 6.1]{Cha24}).
\item
One of the remarkable points of Theorem \ref{thm:Cha24} is that the case where $p$ is small can be equally handled as long as we assume the Howe-factorizability.
However, when we later apply this theorem (proof of Theorem \ref{thm:reg comparison}), we assume that $p$ is large so that the algebraic construction and description of supercuspidal representations (Section \ref{sec:Yu}) work.
Thus any unramified pair $(\bfT,\theta)$ is automatically Howe-factorizable as mentioned before.
\end{enumerate}
\end{rem}

We note here that the fact that we chose to leave out $\bbU_r$ in the notation $X_{\bbT_r \subset \bbG_r}$ but include it in the notation $R_{\bbT_r, \bbU_r}^{\bbG_r}$ is quite misleading: while the geometry of $X_{\bbT_r \subset \bbG_r}$ genuinely depends on the choice of $\bbU_r$, a consequence of Theorem \ref{thm:Cha24} is that its $\theta$-Euler characteristic $R_{\bbT_r,\bbU_r}^{\bbG_r}(\theta)$ does not depend on this choice:


\begin{cor}\label{cor:Cha24}
Assume that $(\bfT,\theta)$ is an unramified elliptic Howe-factorizable pair of $\bfG$ of depth $\leq r$.
Then we have $R_{\bbT_r,\bbU_{r}}^{\bbG_r}(\theta)\cong R_{\bbT_r,\bbU'_{r}}^{\bbG_r}(\theta)$ for any triple $(\bfT,\bfU', \x)$.
\end{cor}


From now on, we write $R_{\bbT_r}^{\bbG_r}(\theta)$ for $R_{\bbT_r,\bbU_r}^{\bbG_r}(\theta)$ whenever $(\bfT,\theta)$ is unramified elliptic Howe-factorizable.

The second result in the $r=0$ setting is a special case of the Deligne--Lusztig character formula \cite[Theorem 4.2]{DL76}:

\begin{prop}[{\cite[Theorem 1.2]{CI21-RT}}]
\label{prop:geom vreg}
Let $(\bfT,\theta)$ be an unramified pair of depth $\leq r$.
Let $\gamma \in G_{\x,0}$ be an unramified very regular element. 
If $\gamma$ is not $G_{\x,0}$-conjugate to an element of $T$, then $\Theta_{R_{\bbT_{r},\bbU_{r}}^{\bbG_r}(\theta)}(\gamma)=0$. 
If $\gamma$ is $G_{\x,0}$-conjugate to an element of $T$ (in this case, we assume that $\gamma$ itself belongs to $T$ by conjugating), we have
\begin{equation*}
\Theta_{R_{\bbT_{r},\bbU_{r}}^{\bbG_r}(\theta)}(\gamma) 
= 
\sum_{w \in W_{G_{\x,0}}(\bfT)} \theta^{w}(\gamma).
\end{equation*}
\end{prop}

\subsection{Geometric Howe-unramified regular supercuspidal types}\label{subsec:reg comparison}

Let $(\bfT,\theta)$ be an unramified regular elliptic pair of $\bfG$.
We assume that $p \neq 2$ is not bad for $\bfG$ and $p \nmid |\pi_1(\bfG_{\der})| \cdot |\pi_1(\widehat \bfG_{\der})|$. Assume additionally that $q$ is large enough so that \eqref{eq:Henniart} holds.
Note that then $\theta$ is automatically Howe-factorizable; let $\Psi$ be a Yu datum corresponding to $(\bfT,\theta)$.
The main theorem of this sections is:

\begin{thm}\label{thm:reg comparison}
We have an isomorphism of $TG_{\x,0}$-representations
\[
\cc\tau^{\FKS}_{\Psi}
\cong
(-1)^{r(\bfG^0) - r(\bfT) + r(\bfT,\theta)} R_{\bbT_r}^{\bbG_r}(\theta),
\]
where $R_{\bbT_r}^{\bbG_r}(\theta)$ is extended to a representation of $TG_{\x,0}=Z_{\bfG}G_{\x,0}$ by demanding that $Z_{\bfG}$ act by $\theta|_{Z_{\bfG}}$. 
In particular, the compact induction of $(-1)^{r(\bfG^0) - r(\bfT) + r(\bfT,\theta)} R_{\bbT_r}^{\bbG_r}(\theta)$ to $G$ is irreducible supercuspidal and we have
\[
\pi_{\Psi}^{\FKS}
\cong
\cInd_{TG_{\x,0}}^{G}((-1)^{r(\bfG^0) - r(\bfT) + r(\bfT,\theta)} R_{\bbT_r}^{\bbG_r}(\theta)).   
\]
\end{thm}

\begin{proof}
    By Theorem \ref{thm:Cha24}, either $R_{\bbT_r}^{\bbG_r}(\theta)$ or $-R_{\bbT_r}^{\bbG_r}(\theta)$ is an irreducible representation of $G_{\x,0}$. 
    Write $|R_{\bbT_r}^{\bbG_r}(\theta)| = c \cdot R_{\bbT_r}^{\bbG_r}(\theta)$ ($c \in \{\pm 1\}$) for this representation. By Proposition \ref{prop:geom vreg}, for all unramified very regular $\gamma \in G_{\x,0}$,
    \begin{equation*}
        \Theta_{|R_{\bbT_r}^{\bbG_r}(\theta)|}(\gamma)
        =
        c\cdot\sum_{w \in W_{G_{\x,0}}(\bfT_\gamma,\bfT)} \theta^w(\gamma).
    \end{equation*}
    On the other hand, by Proposition \ref{prop:AS-vreg}, for all unramified very regular $\gamma \in G_{\x,0}$,
    \begin{equation*}
        \Theta_{\cc \tau_{\Psi}^{\FKS}}(\gamma)
        =
        (-1)^{r(\bfG^0)-r(\bfT)+r(\bfT,\theta)}\cdot \sum_{w \in W_{G_{\x,0}}(\bfT_\gamma,\bfT)} \theta^w(\gamma).
    \end{equation*}
    Note that both representations $|R_{\bbT_{r}}^{\bbG_{r}}(\theta)|$ and $\cc\tau_{\Psi}^{\FKS}$ have $Z_{\bfG}$-central character $\theta|_{Z_{\bfG}}$.
    Thus, by Theorem \ref{thm:vreg characterization}, we must have the following isomorphism of irreducible $TG_{\x,0}$-representations
    \begin{equation*}
        |R_{\bbT_r}^{\bbG_r}(\theta)| \cong \cc\tau_{\Psi}^{\FKS}.
    \end{equation*}
    In particular, we see that $c = (-1)^{r(\bfG^0) - r(\bfT) - r(\bfT,\theta)}$, and statement of the theorem follows by taking the compact induction of the above isomorphism.
\end{proof}

\begin{remark}\label{rem:toral comparison}
In the case that $\theta$ is $0$-toral, Theorem \ref{thm:reg comparison} is slightly weaker than \cite[Theorem 8.2]{CO25} in that the condition on $q$ required for Theorem \ref{thm:reg comparison} is stronger: in the \textit{op.\ cit.}, we need only assume that $q$ is large enough so that $T_{0,\vreg}$ generates $T_{0}$ as a group, and for this it is enough to assume 
\begin{equation*}
    \frac{|\bbT_0(\F_{q})|}{|\bbT_0(\F_{q})_{\nvreg}|} > 2.
\end{equation*}
On the other hand, the methodology in \cite{CO25} uses the $(\bfT,\bfG)$-genericity assumption on $\theta$ quite seriously. It is an interesting question to determine how the methods in \textit{op.\ cit.\ }can be refined to relax the required inequality \eqref{eq:Henniart}; this would likely yield a comparison result with a weaker requirement on $q$. The authors indeed investigated this for some time, but in the end opted to proceed using the simple characterization Theorem \ref{thm:vreg characterization}. 
\end{remark}

\subsection{Geometric Howe-unramified regular supercuspidal $L$-packets}\label{subsec:geom packets}

Let $(\bfT,\theta)$ be an unramified elliptic regular pair. 
We assume that $p \neq 2$ is not bad for $\bfG$ and $p \nmid |\pi_1(\bfG_{\der})| \cdot |\pi_1(\widehat \bfG_{\der})|$. Assume that $q$ is large enough so that \eqref{eq:Henniart} holds. 

In this subsection, we discuss the implications of Theorem \ref{thm:reg comparison} from the perspective of the Langlands program. Our discussion will take place in the setting of Kaletha's regular supercuspidal $L$-packets as constructed in \cite{Kal19} and reformulated in \cite{FKS23}. We very briefly review this. Following \textit{op.\ cit.}, we assume that $F$ has characteristic zero.

Let $\varphi \from W_F \to {}^L G$ be a regular supercuspidal parameter \cite[Definition 5.2.3]{Kal19}. Kaletha proves \cite[Proposition 5.2.7]{Kal19} that one can associate to $\varphi$ a \textit{regular supercuspidal $L$-packet datum}---a quadruple $(\bfT, \widehat j, \chi, \theta)$ where:
\begin{enumerate}[label=\textbullet]
    \item $\bfT$ is an $F$-rational torus whose dimension equals to the absolute rank of $\bfG$,
    \item $\widehat j \from \widehat \bfT \hookrightarrow \widehat \bfG$ is an embedding whose $\widehat \bfG$-conjugacy class is $\Gamma$-stable,
    \item $\chi$ is a minimally ramified $\chi$-data for the set of roots of $\bfT$ in $\bfG$ (see \cite[Definition 4.6.1]{Kal19}),
    \item $\theta \from T \to \C^\times$ is a character
\end{enumerate}
satisfying several conditions.
(See \cite[Section 5]{Kal19} and also \cite[Section 8]{CO25} for more details).
We assume further that $\varphi$ is \textit{Howe-unramified}: assume that $\bfT$ splits over an unramified extension of $F$. There is then a canonical choice for $\chi$. 

The $L$-packet corresponding to $\varphi$ is then indexed by $G$-conjugacy classes of $F$-rational embeddings $j \from \bfT \hookrightarrow \bfG$ which are \textit{admissible for $\widehat j$} in the sense of \cite[Section 5.1]{Kal19}. In \cite[Section 5.3]{Kal19}, Kaletha associates to each quintuple $(\bfT,\widehat j, \chi, \theta, j)$ a regular supercuspidal representation: the one arising from Yu's construction from a Yu datum determined by a(ny) Howe factorization of $(j\bfT, (j\theta)')$, where $j\bfT:=j(\bfT)\subset\bfG$ and $(j\theta)' \from jT \to \C^\times$ is the character $(j\theta)'(j(t)) = \theta(t) \cdot \varepsilon^{\ram}[j\theta](j(t))$, where $\varepsilon^{\ram}[j\theta]$ is the sign character discussed in Section \ref{subsec:ADS} ($j\theta$ is the pull back of the character $\theta$ along $j$). 
In summary, Kaletha's $L$-packet for $\varphi$ is the set
\begin{equation*}
    \{\pi_{(j\bfT, (j\theta)')}\}_j
\end{equation*}
where $j$ ranges over $G$-conjugacy classes of $F$-embeddings $\bfT\hookrightarrow\bfG$ admissible for $\widehat j$. This can be repackaged using the twisted Yu-construction introduced in \cite{FKS23} since $\pi_{(j\bfT, j\theta)}^{\FKS} \cong \pi_{(j\bfT, (j\theta)')}$ (see Section \ref{subsec:ADS}). 
Therefore the $L$-packet corresponding to $\varphi$ is 
\begin{equation*}
    \{\pi_{(j\bfT, j\theta)}^{\FKS}\}_j.
\end{equation*}

The key contribution of Theorem \ref{thm:reg comparison} in this context is that positive-depth Deligne--Lusztig induction realizes the \textit{correct} parametrization of regular supercuspidal representations. The first revelation of this type established the comparison between positive-depth Deligne--Lusztig induction and the twisted parametrization of regular supercuspidal representations in the $0$-toral (Definition \ref{defn:toral}) setting---this is one of the main contributions of \cite{CO25}. Theorem \ref{thm:reg comparison} allows the entire discussion of \cite[Section 8]{CO25} to hold for arbitrary Howe-unramified regular supercuspidal $L$-packets.

\begin{thm}\label{thm:geom packet}
    Let $\varphi$ be a Howe-unramified regular supercuspidal parameter and denote by $(\bfT, \widehat j, \chi, \theta)$ its associated regular supercuspidal $L$-packet datum of depth $r$. 
    Then the $L$-packet for $\varphi$ is 
    \begin{equation*}
        \{\cInd_{jT \cdot G_{\x_j,0}}^{G}(|R_{\bbT_{j,r}}^{\bbG_{j,r}}(j\theta)|)\}_j,
    \end{equation*}
    where
    \begin{enumerate}[label=\textbullet]
        \item $j$ varies over $G$-conjugacy classes of $F$-rational $\widehat j$-admissible embeddings $\bfT\hookrightarrow\bfG$, 
        \item $\x_j\in\cB(\bfG,F)$ is a point associated with the unramified elliptic maximal torus $j\bfT$,
        \item $\bbT_{j,r}$ and $\bbG_{j,r}$ are the algebraic groups defined in Subsection \ref{subsec:parahoricDL} associated to $\x_j$.
    \end{enumerate}
\end{thm}

Under the additional assumption that $p \geq (2+e)n$ where $e$ is the ramification degree of $F$ over $\Q_p$ and $n$ is the dimension of the smallest faithful rational representation of $\bfG$, Fintzen--Kaletha--Spice prove \cite[Theorem 4.4.4]{FKS23} that regular supercuspidal $L$-packets satisfy \textit{stability} and many endoscopic character identities.

\begin{cor}\label{cor:stability}
    Positive-depth Deligne--Lusztig induction induces the stability. 
    That is, 
    \begin{equation*}
        \text{$(j \from \bfT \hookrightarrow \bfG,\theta) \mapsto \cInd_{j\bfT(F) \cdot G_{\x_j,0}}^{\bfG(F)}(|R_{\bbT_{j,r}}^{\bbG_{j,r}}(\theta)|)$}
    \end{equation*}
    induces a map from the set of stable conjugacy classes of unramified elliptic regular pairs to the space of stable distributions.
\end{cor}

\subsection{Geometric Howe-unramified regular Kim--Yu types}\label{subsec:geom types}

Let $(\bfT, \theta)$ be an unramified regular pair of depth $r$ (thus note that $\bfT$ is not necessarily elliptic). 
We assume that $p \neq 2$ is not bad for $\bfG$ and $p \nmid |\pi_1(\bfG_{\der})| \cdot |\pi_1(\widehat \bfG_{\der})|$.
In Section \ref{subsec:types}, we discussed that Kaletha's Howe factorization gives rise to a sequence of twisted Levi subgroups $\vec \bfG = (\bfG^0 \subsetneq \cdots \subsetneq \bfG^d = \bfG)$ together with a sequence $\vec \bfM = (\bfM^0 \subseteq \cdots \subseteq \bfM^d = \bfM)$. 
As described in Section \ref{subsec:types}, one can consider the AFMO twisted Kim--Yu representation $(K_0,\rho)$. 
This representation restricts to a supercuspidal type $(K_{M,0}, \rho_{M})$ which induces to the $M_{\x,0}$-representation $\cc \tau^{\FKS}|_{M_{\x,0}}$. 

Recall that Proposition \ref{prop:types induction} established the relationship between $\Ind_{K_0}^{G_{\x,0}}(\rho)$ and the depth-$r'$ parabolic induction: for any $r' \geq r$,
\begin{equation*}
    \Ind_{K_0}^{G_{\x,0}}(\rho) \subseteq \Ind_{P_{x,0}G_{x,r'+}}^{G_{x,0}}(\Ind_{K_{M,0}}^{M_{\x,0}}(\rho)) = \Ind_{P_{x,0} G_{x,r'+}}^{G_{x,0}}(\cc \tau^{\FKS}).
\end{equation*}
Assume now that $q$ is large enough so that \eqref{eq:Henniart} holds for the pair $(\bfT,\theta)$ where $\bfT$ is viewed as a therefore \textit{elliptic} maximal torus in $\bfM$. Then Proposition \ref{prop:types induction} and  Theorem \ref{thm:reg comparison} together yield the following corollary:

\begin{thm}\label{thm:geom type}
    Retain the set-up above. 
    The restriction to $K_0$ of the $G_{\x,0}$-representation $(-1)^{r(\bfM^0) - r(\bfT) + r(\bfT,\theta)} R_{\bbT_{r'}}^{\bbG_{r'}}(\theta)$ contains the AFMO-twisted Kim--Yu type $\rho$. 
    Furthermore, if $\theta$ is split-generic, then $R_{\bbT_{r}}^{\bbG_{r}}(\theta)$ is a type and
    \begin{equation*}
        (-1)^{r(\bfM^0) - r(\bfT) + r(\bfT,\theta)} R_{\bbT_{r}}^{\bbG_{r}}(\theta) = \Ind_{K_0}^{G_{\x,0}}(\rho).
    \end{equation*}
\end{thm}

\begin{proof}
    By the transitivity of positive-depth Lusztig induction (see \cite[Proposition 3.3]{Cha24}, especially \cite[Lemma 3.4]{Cha24}), we have
    \begin{equation*}
        R_{\bbT_{r'}}^{\bbG_{r'}}(\theta) = \Ind_{P_{\x,0}G_{\x,{r'}+}}^{G_{\x,0}}(\Inf_{\bbM_{r'}(\F_q)}^{M_{\x,0}}(R_{\bbT_{r'}}^{\bbM_{r'}}(\theta))),
    \end{equation*}
    so the first assertion follows from Proposition \ref{prop:types induction} and Theorem \ref{thm:reg comparison}. For the second assertion, we specialize to the case $r' = r$. By Theorem \ref{thm:Cha24} and Remark \ref{rem:scalar product}, $R_{\bbT_r}^{\bbG_r}(\theta)$ is irreducible when $\theta$ is split-generic. Then of course the containment in Proposition \ref{prop:types induction} becomes an equality.
\end{proof}

\subsection{A conjecture}\label{subsec:conjecture}

It seems fitting to mention here our conjecture on realizing regular Kim--Yu types in general. Assume that $p \neq 2$ is not bad for $\bfG$. Assume that $(\bfT,\theta)$ is an unramified Howe-factorizable pair. In \cite{Cha24}, a variation of $R_{\bbT_r}^{\bbG_r}(\theta)$ was introduced: 
\begin{equation}\label{eq:modified DL induction}
    r_{\bbT_r}^{\bbG_r}(\theta) \coloneqq \Inf_{\bbG_{r_{d-1}}^d(\F_q)}^{\bbG_{r_d}^d(\F_q)} (R_{\bbG_{r_{d-1}}^{d-1}}^{\bbG_{r_{d-1}}^d}(r_{\bbT_{r_{d-2}}}^{\bbG_{r_{d-2}}^{d-1}}(\theta_{\leq d-1}))) \otimes \theta_d,
\end{equation}
where $\vec \phi = (\phi_{-1},\ldots,\phi_{d})$ is a Howe factorization for $\theta$, the functor $R_{\bbM_s}^{\bbG_s}$ denotes the depth-$s$ Lusztig induction from a twisted Levi $\bfM$ containing $\bfT$ (\`a la \cite[Section 3.1]{Cha24}), and $\theta_{\leq j} \coloneqq \prod_{i=-1}^{j} \phi_i|_{T_0}$.
A surprising feature of positive-depth Deligne--Lusztig induction $R_{\bbT_r}^{\bbG_r}$ is that when $\bfT$ is elliptic in $\bfG$, then $R_{\bbT_{r'}}^{\bbG_{r'}}(\theta) = R_{\bbT_r}^{\bbG_r}(\theta)$ for any $r' \geq r = \depth(\theta)$ \cite[Corollary 5.3]{Cha24}. From this, it is easy to conclude (see \cite[Proposition 6.4]{Cha24}) that $R_{\bbT_r}^{\bbG_r}(\theta)$ agrees with the variant $r_{\bbT_r}^{\bbG_r}(\theta)$ for $\theta$ split-generic.

\begin{conj}\label{conj:regular type}
    The representation $r_{\bbT_r}^{\bbG_r}(\theta)$ of $G_{\x,0}$ is a linear combination of Bushnell--Kutzko types and its restriction to $K_0$ contains the AFMO-twisted Kim--Yu type associated to $(\bfT,\theta)$.
\end{conj}

We conjecture the following structural result relating positive-depth Lusztig induction to the construction of Kim--Yu types. This is inspired by Kaletha's Howe factorization \cite{Kal19}, the recent construction of positive-depth character sheaves \cite{BC24}, and the Howe-factorizability of the structure of $R_{\bbT_r}^{\bbG_r}(\theta)$ in terms of successive positive-depth Lusztig inductions \cite{Cha24} as briefly recalled above in \eqref{eq:modified DL induction}.

\begin{conj}\label{conj:Lusztig induction}
    Let $\bfG'$ be a twisted Levi subgroup of $\bfG$ containing $\bfT$.
    Let $\alpha$ be any representation of $\bbG'_{r}(\F_q)$ of depth $<r$ and let $\phi \from \bbG_r'(\F_q) \to \overline \Q_\ell^\times$ be a $(\bfG',\bfG)$-generic character of depth $r$. Then
    \begin{equation*}
        |R_{\bbG_r'}^{\bbG_r}(\alpha \otimes \phi)| = \Ind_{K_0}^{G_{\x,0}}(\epsilon^{G/G'} \otimes \rho_{\alpha \otimes \phi}),
    \end{equation*}
    where $\epsilon^{G/G'}$ is the FKS-twist \cite[Definition 3.1]{FKS23} and $\rho_{\alpha \otimes \phi}$ is the representation constructed in \cite[Section 4]{Yu01} in the context of \cite[Section 7]{KY17}.
\end{conj}

In the above, it is implicit in the statement that either $R_{\bbG_r'}^{\bbG_r}(\alpha \otimes \phi)$ or $-R_{\bbG_r'}^{\bbG_r}(\alpha \otimes \phi)$ is a genuine representation (not just a virtual representation). Note that Conjecture \ref{conj:Lusztig induction} implies Conjecture \ref{conj:regular type}. In the particular, one can view Conjecture \ref{conj:Lusztig induction} as a refinement of the comparison results in this paper. For example, in the case that $\bfT$ is elliptic, one has $r_{\bbT_r}^{\bbG_r} = R_{\bbT_r}^{\bbG_r}$ \cite[Proposition 6.4]{Cha24}, so Theorem \ref{thm:reg comparison} (and also the later Theorem \ref{thm:gen comparison}) describes the \textit{composition} of Lusztig inductions along a Howe factorization in terms of Yu's construction.

\section{Green functions for positive-depth Deligne--Lusztig induction}\label{sec:DL}

Our aim of this section is to introduce a positive-depth version of Green functions. This will allow us to obtain a character formula for positive-depth Deligne--Lusztig induction analogous to the depth-zero setting \cite[Theorem 4.2]{DL76}. Looking ahead, after establishing the parallel picture in the setting of Yu's construction (Section \ref{sec:theta general}, especially Section \ref{sec:Yu green}), the structure afforded by these Green-function-theoretic character formulae will allow us to wield the results of Section \ref{subsec:reg comparison} to study $R_{\bbT_r}^{\bbG_r}(\theta)$ for arbitrary $\theta$ (Sections \ref{subsec:Yu exhaust}, \ref{subsec:KY exhaust}).

In this section, we assume $\bfT$ is an unramified maximal torus of $\bfG$ and that $\x \in \cB(\bfG,F)$ is associated to $\bfT$. We make no further assumptions: $p$, $q$, $\theta$ are all allowed to be arbitrary.

\subsection{Positive-depth Green functions}\label{sec:green}

In this section, we define a generalization of the Green functions of Deligne--Lusztig \cite[Definition 4.1]{DL76}. We follow the same strategy as in \cite[Section 4]{DL76} to give a character formula for $R_{\bbT_{r},\bbU_{r}}^{\bbG_r}(\theta)$ in terms of Green functions. We will see later (Section \ref{sec:character}) that in the special case that $\bfT$ is elliptic and $\theta$ is $0$-toral, this gives rise to a new proof of the character formula for regular supercuspidal representations.

Let $\bbG_{r}(\F_{q})_{\ss}$ (resp.\ $\bbG_{r}(\F_{q})_{\unip}$) denote the subset of semisimple (resp.\ unipotent) elements of $\bbG_{r}(\F_{q})$, respectively. 
We remark that $g\in \bbG_{r}(\F_{q})$ is semisimple (resp.\ unipotent) if and only if the order of $g$ is prime-to-$p$ (resp.\ $p$-power) since $\F_{q}$ is of characteristic $p>0$.

\begin{lem}\label{lem:ss-lift}
Let $s\in\bbG_{r}(\F_{q})_{\ss}$.
Then there exists a semisimple element $\tilde{s}\in G_{\x,0}$ such that $s$ and $\tilde{s}$ have the same order.
\end{lem}

\begin{proof}
Since we have $\bbG_{r}(\F_{q})=G_{\x,0:r+}$, we can find an element $\tilde{g}\in G_{\x,0}$ whose image in $G_{\x,0:r+}$ is $s$.
We take a topological Jordan decomposition $\tilde{g}=\tilde{s}\tilde{u}$, i.e., $\tilde{s}$ and $\tilde{u}$ are commuting elements such that
\begin{itemize}
\item
$\tilde{s}\in G_{\x,0}$ is an element of finite prime-to-$p$ order, and
\item
$\tilde{u}\in G_{\x,0}$ is an element satisfying $\lim_{n\rightarrow\infty}\tilde{u}^{p^{n}}=1$
\end{itemize}
(the existence of a topological Jordan decomposition is guaranteed by the compactness of $G_{\x,0}$; see \cite{Spi08}).
Here, note that $\tilde{s}\in G_{\x,0}$ is necessarily semisimple since its order is finite and prime-to-$p$.
Also note that the image of $\tilde{u}$ in $G_{\x,0:r+}$ is of finite $p$-power order.
In particular, the images of $\tilde{s}$ and $\tilde{u}$ give the Jordan decomposition of $s$.
As $s$ is semisimple, this implies that the image of $\tilde{s}$ is $s$.
\end{proof}

For any $s\in\bbG_{r}(\F_{q})_{\ss}$, we let $\bbG_{r}^s$ denote the centralizer of $s$ in $\bbG_{r}$, that is, the reduced subscheme of $\bbG_r$ whose $\overline \F_q$-points are given by the centralizer of $s$ in $\bbG_r(\ol{\F}_q)$; let $\bbG_{r,s}$ denote its identity component. 

On the other hand, by letting $\tilde{s}$ be a lift of $s$ to $G_{\x,0}$ as in Lemma \ref{lem:ss-lift}, we also have the centralizer group $\bfG^{\tilde{s}}$ of $\tilde{s}$ taken in $\bfG$.
By applying the construction of \cite[Section 2.6]{CI21-RT} to this closed subgroup $\bfG^{\tilde{s}}$ of $\bfG$, we obtain a smooth closed subgroup scheme $\bbG^{\tilde{s}}_{r}$ of $\bbG_{r}$ defined over $\F_{q}$.

\begin{lem}\label{lem:sr=rs}
    We have $\bbG_r^{\tilde s} = (\bbG_r)^s$ and hence also $\bbG_{\tilde s,r} = \bbG_{r,s}$.
\end{lem}

\begin{proof}
It suffices to check that $\bbG^{\tilde{s}}_{r}$ indeed satisfies that $\bbG^{\tilde s}_{r}(\ol{\F}_{q})=\bbG_{r}(\ol{\F}_{q})^{s}$ since this is the condition characterizing $\bbG_{r}^s$.
Let $\breve{F}$ denote the completion of a maximal unramified extension of $F$.
We write $\mathcal{G}_{\x}$ for the parahoric group scheme over $\cO$ associated to $\x$.
Similarly, we write $\mathcal{G}_{\x,r}$ for the $r$-th Moy--Prasad filtration group scheme over $\cO$ associated to $\x$.

By examining the construction of \cite[Section 2.6]{CI21-RT}, we can check that $\bbG^{\tilde{s}}_{r}(\ol{\F}_{q})$ is given by the quotient $\bfG^{\tilde{s}}(\breve{F})\cap\mathcal{G}_{\x}(\cO_{\breve{F}})/\bfG^{\tilde{s}}(\breve{F})\cap\mathcal{G}_{\x,r+}(\cO_{\breve{F}})$.
On the other hand, $\bbG_{r}(\ol{\F}_{q})^{s}$ is given by the set of $s$-fixed points (with respect to the conjugation action) of $\bbG_{r}(\ol{\F}_{q})=\mathcal{G}_{\x}(\cO_{\breve{F}})/\mathcal{G}_{\x,r+}(\cO_{\breve{F}})$.
The former is included in the latter (both are considered to be subgroups of $\mathcal{G}_{\x}(\cO_{\breve{F}})/\mathcal{G}_{\x,r+}(\cO_{\breve{F}})$).
Hence, it is enough to show its converse; that is, any coset $g\mathcal{G}_{\x,r+}(\cO_{\breve{F}})$ fixed by the $\tilde{s}$-conjugation is equal to a coset $g'\mathcal{G}_{\x,r+}(\cO_{\breve{F}})$ whose $g'$ is fixed by the $\tilde{s}$-conjugation.
This follows from a usual argument based on the vanishing of the first cohomology of $\mathcal{G}_{\x,r+}(\cO_{\breve{F}})$ with respect to the action of $\langle\tilde{s}\rangle$, which is valid since $\mathcal{G}_{\x,r+}(\cO_{\breve{F}})$ is a pro-$p$ group such that the subgroup $\mathcal{G}_{\x,r'}(\cO_{\breve{F}})$ for any $r'>r+$ is preserved by the action of $\tilde{s}$, which is of finite prime-to-$p$ order (see, e.g., \cite[Section 13.8]{KP23}).
\end{proof}

For any $(s,t) \in \bbG_{r}(\F_{q})_{\ss} \times \bbT_{r}(\F_{q})_{\ss}$, define a subvariety $X_{\bbT_r \subset \bbG_r}^{(s,t)}$ of $X_{\bbT_r \subset \bbG_r}$ by
\begin{equation*}
X_{\bbT_r \subset \bbG_r}^{(s,t)} \coloneqq \{x \in \bbG_r \mid \text{$x^{-1} \sigma(x) \in \bbU_r$ and $sxt = x$}\}.
\end{equation*}
We put $\bbT_{r}^{+}(\F_{q}):=T_{0+:r+}$.
Then the subgroup $\bbG_{r}^s(\F_{q}) \times \bbT_{r}^{+}(\F_{q})\subset \bbG_{r}(\F_{q}) \times \bbT_{r}(\F_{q})$ stabilizes the subvariety $X_{\bbT_r \subset \bbG_r}^{(s,t)}$.

Note that if $t$ is a semisimple element of $\bbT_{r}(\F_{q})$ with a lift $\tilde{t}\in T_{0}$, then $\bbT_{r}$ and $\bfT$ are contained in $\bbG_{r,t}$ and $\bfG_{\tilde{t}}$, respectively.
Here, since $t$ is semisimple, $\bfG_{\tilde{t}}$ is a connected reductive subgroup of $\bfG$.
Thus, $(\bbT_{r}\subset\bbG_{r,t})$ can be thought of as a pair obtained from $(\bfT\subset\bfG_{\tilde{t}})$ according to the manner of \cite{CI21-RT} by Lemma \ref{lem:sr=rs}.
In particular, it makes sense to consider the associated positive-depth Deligne--Lusztig variety $X_{\bbT_{r}\subset\bbG_{r,t}}$.

\begin{lem}\label{lem:(s,t) fixed}
Let $(s,t) \in \bbG_{r}(\F_{q})_{\ss} \times \bbT_{r}(\F_{q})_{\ss}$. Then
\begin{equation*}
X_{\bbT_r \subset \bbG_r}^{(s,t)} = \bigsqcup_{\substack{g \in \bbG_r(\F_q)/\bbG_{r,t}(\F_q) \\ {}^{g}t = s^{-1}}} X_{\bbT_r \subset \bbG_r}^{(s,t)}(g)
\end{equation*}
where $X_{\bbT_r \subset \bbG_r}^{(s,t)}(g) \coloneqq X_{\bbT_r \subset \bbG_r}^{(s,t)} \cap g \bbG_{r,t}$ for $g \in \bbG_{r}(\F_{q})$. 
Moreover, $x \mapsto g^{-1} x$ defines a $(\bbG_{r,s}(\F_{q}) \times \bbT_{r}^{+}(\F_{q}))$-equivariant isomorphism
\begin{equation*}
X_{\bbT_r \subset \bbG_r}^{(s,t)}(g) \cong X_{\bbT_r \subset \bbG_{r,t}},
\end{equation*}
where $\bbG_{r,s}(\F_{q})$ acts on the right-hand side via the identification $\bbG_{r,s}={}^{g}\bbG_{r,t}\cong\bbG_{r,t}\colon\gamma\mapsto\gamma^{g}$.
\end{lem}

\begin{proof}
If $x \in X_{\bbT_r \subset \bbG_r}^{(s,t)}$, then we have $sxt = x$ and $x^{-1}\sigma(x) = u$ for some $u \in \bbU_r$. Hence we have $xu = \sigma(x) = s\sigma(x)t = sxut = xt^{-1} u t$, which implies that $u \in \bbG_{r}^t$. 
By \cite[8.10]{Ste68}, this furthermore implies that $u \in \bbG_{r,t}$.
Since the Lang map is surjective, there exists a $z \in \bbG_{r,t}$ such that $z^{-1}\sigma(z) = u$. 
Let $g \coloneqq x z^{-1}$ so that we now have $x \in g \bbG_{r,t}$. 
We have $\sigma(g) = \sigma(xz^{-1}) = \sigma(x) \sigma(z)^{-1} = x u \cdot u^{-1} z^{-1} = xz^{-1} = g$ so that $g \in \bbG_r(\F_q)$. Furthermore, we have $gz = x = sxt = sgzt = sgtz$, which implies that ${}^{g}t=s^{-1}$.
This proves that $x \in X_{\bbT_r \subset \bbG_r}^{(s,t)}(g)$ for some $g \in \bbG_r(\F_q)$ such that ${}^{g}t = s^{-1}$.
In other words, we get 
\[
X_{\bbT_r \subset \bbG_r}^{(s,t)}
=
\bigcup_{\substack{g \in \bbG_r(\F_q)/\bbG_{r,t}(\F_q) \\ {}^{g}t=s^{-1}}} X_{\bbT_r \subset \bbG_r}^{(s,t)}(g)
\]
The fact that the $X_{\bbT_r \subset \bbG_r}^{(s,t)}(g)$ are disjoint for $g$ varying over representatives of $\bbG_r(\F_q)/\bbG_{r,t}(\F_q)$ is clear, so this proves the disjoint union assertion. 

We now prove the last assertion in the lemma. Fix a $g \in \bbG_r(\F_q)$ with ${}^{g}t=s^{-1}$. 
It is clear that we have a map 
\begin{equation*}
X_{\bbT_r \subset \bbG_r}^{(s,t)}(g) \to X_{\bbT_r \subset \bbG_{r,t}}\colon x \mapsto g^{-1} x.
\end{equation*}
The $(\bbG_{r,s}(\F_{q}) \times \bbT_{r}^{+}(\F_{q}))$-equivariance holds since $\gamma \in \bbG_{r,s}(\F_{q})$ implies $\gamma^{g} \in \bbG_{r,t}(\F_{q})$, and so for any $t_+ \in \bbT_{r}^{+}(\F_{q})$ we have $\gamma x t \mapsto g^{-1}(\gamma x t) = \gamma^{g}g^{-1} x t$. To see that we have an isomorphism, we just need to see that for any $z \in X_{\bbT_r \subset \bbG_{r,t}}$ and any $g \in \bbG_r(\F_q)$, we have $gz \in X_{\bbT_r \subset \bbG_r}^{(s,t)}(g)$.
 Indeed, we have $(gz)^{-1} \sigma(gz) = z^{-1}\sigma(z) \in \bbU_r\cap\bbG_{r,t}$ and $sgzt = sgtz = gt^{-1} \cdot tz = gz$.
 (Here, note that $\bbU_r\cap\bbG_{r,t}$ is a subgroup associated to the unipotent radical $\bfU\cap\bfG_{\tilde{t}}$ of a Borel subgroup of $\bfG_{\tilde{t}}$. The fact that $\bfU\cap\bfG_{\tilde{t}}$ is the unipotent radical of a Borel subgroup of $\bfG_{\tilde{t}}$ follows from that $\tilde{t}$ belongs to $\bfT$.)
\end{proof}

\subsection{Character formula for positive-depth Deligne--Lusztig induction}

Now we introduce the positive-depth Green function and establish a full character formula for $R_{\bbT_{r},\bbU_{r}}^{\bbG_{r}}(\theta)$.

\begin{defn}\label{def:pos depth Green}
For a character $\theta_{+} \from \bbT_{r}^{+}(\F_{q}) \to \overline \Q_\ell^\times$, we define the \textit{twisted positive-depth Green function} associated to $\theta_{+}$ by
\begin{equation*}
Q_{\bbT_r, \bbU_r}^{\bbG_r}(\theta_{+}) \from \bbG_{r}(\F_{q})_{\unip} \to \overline \Q_\ell; \quad u \mapsto \frac{1}{|\bbT_{r}(\F_{q})|} \sum_{t_+ \in \bbT_{r}^{+}(\F_{q})} \theta_{+}^{-1}(t_+) \Tr((u,t_+) ; H_c^*(X_{\bbT_r \subset \bbG_r}, \overline \Q_\ell)),
\end{equation*}
where $H_c^*(X_{\bbT_r \subset \bbG_r}, \overline \Q_\ell)$ denotes the alternating sum of $H_c^i(X_{\bbT_r \subset \bbG_r}, \overline \Q_\ell)$.
Observe that if $r=0$, then $\theta_{+} = \mathbbm{1}$ and $Q_{\bbT_0, \bbU_0}^{\bbG_0}(\mathbbm 1)$ is a classical Green function in the sense of \cite[Definition 4.1]{DL76}.
\end{defn}

\begin{thm}\label{thm:geom-char-formula}
Let $\theta \from \bbT_{r}(\F_{q}) \to \overline \Q_\ell^\times$ be any character and write $\theta_{+} = \theta|_{\bbT_{r}^{+}(\F_{q})}$.
\begin{enumerate}
\item
 For $\gamma \in \bbG_{r}(\F_{q})$ with Jordan decomposition $\gamma = su$,
\begin{align*}
\Theta_{R_{\bbT_r, \bbU_r}^{\bbG_r}(\theta)}(\gamma) &= \frac{1}{|\bbG_{r,s}(\F_q)|} \sum_{\substack{x \in \bbG_r(\F_q) \\ s^{x} \in \bbT_r(\F_q)}} {}^{x}\theta(s) \cdot Q_{{}^{x}\bbT_r, {}^{x}\bbU_r \cap \bbG_{r,s}}^{\bbG_{r,s}}({}^{x}\theta_{+})(u)\\
&= \frac{1}{|\bbG_{r,s}(\F_q)|} \sum_{\substack{x \in\bbG_r(\F_q) \\ {}^{x}s \in \bbT_r(\F_q)}} \theta^{x}(s) \cdot Q_{\bbT^{x}_r, \bbU^{x}_r \cap \bbG_{r,s}}^{\bbG_{r,s}}(\theta^{x}_{+})(u).
\end{align*}
\item
We have $Q_{\bbT_{r},\bbU_{r}}^{\bbG}(\theta_{+})=\Theta_{R_{\bbT_r, \bbU_r}^{\bbG_r}(\theta)}|_{\bbG_{r}(\F_{q})_{\unip}}$.
In particular, by Corollary \ref{cor:Cha24}, $Q_{\bbT_{r},\bbU_{r}}^{\bbG}(\theta_{+})$ is independent of the choice of $\bbU_{r}$ provided that $\bfT$ is elliptic in $\bfG$.
\end{enumerate}
\end{thm}

\begin{proof}
We first show (1).
The second equality is obtained just by inverting the index $x$ of the sum, so our task is to establish the first equality.
For any $t \in \bbT_{r}(\F_{q})$, write $t = t_{0} \cdot t_+$ for the Jordan decomposition of $t$, hence $t_{0} \in \bbT_{r}(\F_{q})_{\ss}$ and $t_+ \in \bbT_{r}(\F_{q})_{\unip}=\bbT_{r}^{+}(\F_{q})$.
By the Deligne--Lusztig fixed point formula \cite[Theorem 3.2]{DL76}, we have
\begin{equation*}
\Tr\bigl((\gamma,t); H_c^*(X_{\bbT_r \subset \bbG_r}, \overline \Q_\ell)\bigr) 
= 
\Tr\bigl((u, t_+) ; H_c^*(X_{\bbT_r \subset \bbG_r}^{(s,t_{0})}, \overline \Q_\ell)\bigr).
\end{equation*}
Using this together with Lemma \ref{lem:(s,t) fixed}, we have
\begin{align*}
\Theta_{R_{\bbT_r, \bbU_r}^{\bbG_r}(\theta)}(\gamma)
&= \frac{1}{|\bbT_{r}(\F_{q})|} \sum_{t \in \bbT_{r}(\F_{q})} \theta(t)^{-1} \Tr\bigl((\gamma,t); H_c^*(X_{\bbT_r \subset \bbG_r}, \overline \Q_\ell)\bigr) \\
&= \frac{1}{|\bbT_{r}(\F_{q})|} \sum_{t \in \bbT_{r}(\F_{q})} \theta(t)^{-1} \Tr\bigl((u,t_+) ; H_c^*(X_{\bbT_r \subset \bbG_r}^{(s,t_{0})}, \overline \Q_\ell)\bigr) \\
&= \frac{1}{|\bbT_{r}(\F_{q})|} \sum_{t \in \bbT_{r}(\F_{q})} \theta(t_{0})^{-1}\cdot\theta(t_{+})^{-1} \\
&\qquad\qquad\cdot\frac{1}{|\bbG_{r,t_{0}}(\F_q)|}\sum_{\substack{x \in \bbG_{r}(\F_{q}) \\ {}^{x}t_{0}= s^{-1}}} \Tr\bigl((u^{x}, t_+) ; H_c^*(X_{\bbT_r \subset \bbG_{r,t_{0}}}, \overline \Q_\ell)\bigr). 
\end{align*}
Since the internal sum is zero unless there exists $x\in\bbG_{r}(\F_{q})$ satisfying $t_{0} = (s^{x})^{-1}$, this becomes
\begin{align*}
\frac{1}{|\bbT_{r}(\F_{q})| \cdot |\bbG_{r,s}(\F_q)|} \sum_{\substack{x \in \bbG_{r}(\F_{q}) \\ s^{x} \in \bbT_{r}(\F_{q})}} \theta(s^{x}) \cdot\!\!\!\!
\sum_{t_+ \in \bbT_{r}^{+}(\F_{q})} \theta(t_+)^{-1} \Tr\bigl((u^{x}, t_+) ; H_c^*(X_{\bbT_r \subset \bbG_{r,t_{0}}}, \overline \Q_\ell)\bigr).
\end{align*}
As ${}^{x}\bbG_{r,t_{0}} = \bbG_{r,s}$, the map $g \mapsto {}^{x}g$ gives an isomorphism $X_{\bbT_r \subset \bbG_{r,t_{0}}} \cong X_{{}^{x}\bbT_r \subset \bbG_{r,s}}$ which is equivariant with respect to the actions of $(\bbG_{r,t_{0}}(\F_{q}) \times \bbT_{r}(\F_{q}))$ and $(\bbG_{r,s}(\F_{q}) \times \bbT_{r}(\F_{q}))$.
So we get
\begin{align}\label{e:u,t+}
\nonumber
&\frac{1}{|\bbT_{r}(\F_{q})| \cdot |\bbG_{r,s}(\F_q)|} \sum_{\substack{x \in \bbG_{r}(\F_{q}) \\ s^{x} \in \bbT_{r}(\F_{q})}} \theta(s^{x}) \cdot \sum_{{}^{x}t_+ \in {}^{x}\bbT_{r}^{+}(\F_{q})} {}^{x}\theta({}^{x}t_{+})^{-1} \Tr((u,{}^{x}t_+); H_c^*(X_{{}^{x}\bbT_r \subset \bbG_{r,s}}, \overline \Q_\ell))\\
&= \frac{1}{|\bbG_{r,s}(\F_q)|} \sum_{\substack{x \in \bbG_{r}(\F_{q}) \\ s^{x} \in \bbT_{r}(\F_{q})}} {}^{x}\theta(s) \cdot Q_{{}^{x}\bbT_r, {}^{x}\bbU_r \cap \bbG_{r,s}}^{\bbG_{r,s}}({}^{x} \theta_{+})(u).
\end{align}

We next show (2).
First note that
\begin{equation}\label{e:+ proj}
\frac{|\bbT_{r}(\F_{q})|}{|\bbT_{r}^{+}(\F_{q})|}\cdot
Q_{\bbT_r, \bbU_r}^{\bbG_{r}}(\theta_{+})(u)
=
\frac{1}{|\bbT_{r}^{+}(\F_{q})|} \sum_{t_+ \in \bbT_{r}^{+}(\F_{q})} \theta_{+}^{-1}(t_+) \Tr((u,t_+) ; H_c^*(X_{\bbT_r \subset \bbG_r}, \overline \Q_\ell))
\end{equation}
is the trace of the action of $u$ on the subspace of $H_c^*(X_{\bbT_r \subset \bbG_r}, \overline \Q_\ell)$ wherein $\bbT_{r}^{+}(\F_{q})$ acts by the character $\theta_{+}$.
Therefore this is equal to the trace of $u$ on $\bigoplus_{\theta'} R_{\bbT_r, \bbU_r}^{\bbG_r}(\theta')$ where the sum ranges over all characters $\theta'$ of $\bbT_{r}(\F_{q})$ such that $\theta'|_{\bbT_{r}^{+}(\F_{q})} = \theta_{+}$ (there are $|T_{0:0+}|=|\bbT_{r}(\F_{q})/\bbT_{r}^{+}(\F_{q})|$ such characters). 
On the other hand, by considering \eqref{e:u,t+} in the case that $s=1$, we see that the trace of $u$ on $R_{\bbT_r, \bbU_r}^{\bbG_r}(\theta')$ depends only on the restriction of $\theta'$ to $\bbT_{r}^{+}(\F_{q})$.
Hence we see that
\[
\frac{|\bbT_{r}(\F_{q})|}{|\bbT_{r}^{+}(\F_{q})|}\cdot
Q_{\bbT_r, \bbU_r}^{\bbG_{r}}(\theta_{+})(u)
=
\frac{|\bbT_{r}(\F_{q})|}{|\bbT_{r}^{+}(\F_{q})|}\cdot
\Theta_{R_{\bbT_r, \bbU_r}^{\bbG_r}(\theta)}(u),
\]
which implies the desired equality.
\end{proof}

From now on, we simply write $Q_{\bbT_{r}}^{\bbG}(\theta_{+})$ for $Q_{\bbT_{r},\bbU_{r}}^{\bbG}(\theta_{+})$ whenever $\bfT$ is elliptic in $\bfG$.

\begin{remark}
We explain how Theorem \ref{thm:geom-char-formula} implies Proposition \ref{prop:geom vreg}. 
If $\gamma \in G_{\x,0}$ is unramified very regular with topological Jordan decomposition $\gamma = su$, then $\bfG_{s} = \bfT_s$ is a torus so that $x \in G_{\x,0}$ satisfies ${}^{x}s \in T$ if and only if $x \in N_{G_{\x,0}}(\bfG_{s},\bfT)=N_{\bbG_{r}(\F_{q})}(\bbG_{r,s},\bbT_{r})$. 
If $N_{\bbG_{r}(\F_{q})}(\bbG_{r,s},\bbT_{r}) = \varnothing$, then $\Theta_{R_{\bbT_r}^{\bbG_r}(\theta)}(\gamma) = 0$, and otherwise 
\begin{align*}
\Theta_{R_{\bbT_r}^{\bbG_r}(\theta)}(\gamma) 
&= \frac{1}{|\bbG_{r,s}(\F_q)|} \sum_{x \in N_{\bbG_{r}(\F_{q})}(\bbG_{r,s},\bbT_{r})} \theta^{x}(s) \cdot \theta_{+}^{x}(u) \\
&= \sum_{w \in N_{\bbG_{r}(\F_{q})}(\bbG_{r,s},\bbT_{r})/\bbT_{r}(\F_{q})} \theta^{w}(\gamma).
\end{align*} 
\end{remark}

\begin{rem}\label{rem:geom-char-formula}
When $\bfT$ is elliptic, hence we know the independence of $R^{\bbG_r}_{\bbT,\bbU_r}(\theta)$ of $\bfU$, Theorem \ref{thm:geom-char-formula} (1) can be rewritten as follows:
\[
    \Theta_{R_{\bbT_r, \bbU_r}^{\bbG_r}(\theta)}(\gamma)
    = \sum_{\substack{x \in\bbG_r(\F_q)/\bbG_{r,s}(\F_q) \\ {}^{x}s \in \bbT_r(\F_q)}} \theta({}^{x}s) \cdot Q_{\bbT_r}^{\bbG_{r,{}^{x}s}}(\theta_{+})({}^{x}u).
\]
\end{rem}

%

\subsection{Orthogonality for positive-depth Green functions}

We next establish an inner product formula for positive-depth Green functions associated to elliptic $\bfT$ (Proposition \ref{prop:Q-inner-prod}) analogous to \cite[Theorem 6.9]{DL76}.

The following lemma can be easily checked:
\begin{lem}\label{lem:index-set-rewrite}
Let $\bfT$ and $\bfT'$ be elliptic maximal tori in $\bfG$.
Let $s\in\bbG_{r}(\F_{q})_{\ss}$.
The map $(g,g',n_1) \mapsto (g,g^{\prime-1}n_{1}g, n_{1})$ defines a bijection between the sets
\begin{align*}
\Biggl\{(g,g',n_1) \in \bbG_r(\F_{q}) \times \bbG_r(\F_{q}) \times \bbG_r(\F_{q}) \,\Bigg\vert\, 
\begin{array}{l}
s^{g} \in \bbT_{r}(\F_{q}), s^{g'} \in \bbT'_{r}(\F_{q})\\
n_1 \in N_{\bbG_{r,s}(\F_q)}({}^{g}\bbT_{r}, {}^{g'}{\bbT'}_{r})
\end{array}\Biggr\}, \\
\{(g,n,n_1) \in \bbG_r(\F_{q}) \times N_{\bbG_r(\F_{q})}(\bbT_r, \bbT_r') \times \bbG_{r,s}(\F_q) \mid {s}^{g} \in \bbT_{r}(\F_{q})\}
\end{align*}
and its inverse is given by $(g,n_{1}gn^{-1},n_{1})\mapsfrom(g,n,n_{1})$.
\end{lem}

\begin{prop}\label{prop:Q-inner-prod}
Assume that $\bfT$ and $\bfT'$ are elliptic maximal tori in $\bfG$.
For any characters $\theta_{+} \from \bbT_{r}^{+}(\F_{q}) \to \overline \Q_\ell^\times$ and $\theta_{+}' \from \bbT_{r}^{\prime+}(\F_{q}) \to \overline \Q_\ell^\times$,
\begin{equation*}
\sum_{u \in \bbG_r(\F_{q})_{\unip}} Q_{\bbT_r}^{\bbG_r}(\theta_{+})(u) \overline{Q_{\bbT'_r}^{\bbG_r}(\theta_{+}')(u)} 
= 
\frac{|\bbG_r(\F_{q})|}{|\bbT_{r}(\F_{q})| \cdot |\bbT'_{r}(\F_{q})|} \sum_{\substack{u \in \bbT_{r}^{+}(\F_{q}) \\ n \in N_{\bbG_r(\F_{q})}(\bbT_r, \bbT_r')}} \theta_{+}(u) \overline{\theta_{+}^{\prime n}(u)}.
\end{equation*}
\end{prop}

\begin{proof}
We first check that the conclusion holds if $\bfG$ is a torus, i.e., $\bfG=\bfT$. 
In this case, $\bfU$ is trivial, hence the variety $X_{r}$ is equal to $\bbG_{r}^{\sigma}\cong\bbT_{r}(\F_{q})$.
Thus $H_{c}^{\ast}(X_{r},\ol{\Q}_{\ell})$ is just the sum of all characters of $\bbT_{r}(\F_{q})$.
Hence we get
\begin{align*}
Q_{\bbT_r}^{\bbT_r}(\theta_{+})(u)
&=\frac{1}{|\bbT_{r}(\F_{q})|} \sum_{t_+ \in \bbT_{r}^{+}(\F_{q})} \theta_{+}^{-1}(t_+) \Tr((u,t_+) ; H_c^*(X_{\bbT_r \subset \bbG_r}, \overline \Q_\ell))\\
&=\frac{|\bbT_{r}^{+}(\F_{q})|}{|\bbT_{r}(\F_{q})|} \sum_{\begin{subarray}{c}\theta\colon\bbT_{r}(\F_{q})\rightarrow\ol{\Q}_{\ell}^{\times} \\ \theta|_{\bbT_{r}^{+}(\F_{q})}=\theta_{+}\end{subarray}}\theta(u)
=\theta_{+}(u)
\end{align*}
and so
\begin{equation*}
\sum_{u \in \bbT_{r}^{+}(\F_{q})} Q_{\bbT_r}^{\bbT_r}(\theta_{+})(u) \overline{Q_{\bbT_r}^{\bbT_r}(\theta_{+}')(u)} = \sum_{u \in \bbT_{r}^{+}(\F_{q})} \theta_{+}(u) \overline{\theta_{+}'(u)} = \frac{1}{|\bbT_{r}(\F_{q})|} \sum_{\substack{u \in \bbT_{r}^{+}(\F_{q}) \\ n \in \bbT_{r}(\F_{q})}} \theta_{+}(u) \overline{\theta_{+}'({}^{n}u)}.
\end{equation*}

Now we prove the assertion by the induction on the dimension of $\bfG$.
We let $Z(\bbG_{r})$ denote the center of $\bbG_{r}$.
We choose characters $\theta \from \bbT_{r}(\F_{q}) \to \overline \Q_\ell^\times$ and $\theta' \from \bbT'_{r}(\F_{q}) \to \overline \Q_\ell^\times$ such that $\theta|_{\bbT_{r}^{+}(\F_{q})} = \theta_{+}$, $\theta'|_{\bbT_{r}^{\prime+}(\F_{q})} = \theta_{+}'$, and $\theta|_{Z(\bbG_{r})(\F_{q})_{\ss}}=\theta'|_{Z(\bbG_{r})(\F_{q})_{\ss}}=\mathbbm{1}$.
Note that this is possible since $\bbT^{(\prime)}_{r}(\F_{q})$ is a finite abelian group; $\bbT^{(\prime)+}_{r}(\F_{q})$ is its $p$-part and $Z(\bbG_{r})(\F_{q})_{\ss}$ is contained in its prime-to-$p$ part.

For the inductive hypothesis, we assume that the formula holds when $\bfG$ is replaced by $\bfG_{\tilde{s}}$ for any lift $\tilde{s}$ of $s \in \bbG_r(\F_{q})_{\ss} \smallsetminus Z(\bbG_r)(\F_{q})_{\ss}$ as in Lemma \ref{lem:ss-lift}.
Now, by Theorem \ref{thm:geom-char-formula}, we have
\begin{multline*}
|\bbG_{r}(\F_{q})|\cdot\langle R_{\bbT_r}^{\bbG_r}(\theta), R_{\bbT_r'}^{\bbG_r}(\theta') \rangle_{\bbG_r(\F_{q})}\\
=
\sum_{s \in \bbG_r(\F_{q})_{\ss}} \frac{1}{|\bbG_{r,s}(\F_q)|^2} \sum_{\substack{g,g' \in \bbG_r(\F_{q}) \\ s^{g} \in \bbT_{r}(\F_{q}) \\ s^{g'} \in \bbT'_{r}(\F_{q})}} {}^{g}\theta(s) \overline{{}^{g'}\theta'(s)}
 \cdot \sum_{u \in \bbG_{r,s}(\F_q)_{\unip}} Q_{{}^{g}\bbT_{r}}^{\bbG_{r,s}}({}^{g}\theta_{+})(u) \cdot \overline{Q_{{}^{g'}\bbT'_{r}}^{\bbG_{r,s}}({}^{g'}\theta_{+})(u)}.
\end{multline*}
The summation over $s \in \bbG_r(\F_{q})_{\ss}$ can be separated into $s \in Z(\bbG_r)(\F_{q})_{\ss}$ and $s \in \bbG_r(\F_{q})_{\ss} \smallsetminus Z(\bbG_r)(\F_{q})_{\ss}$. 
We first note that the contribution of each $s \in Z(\bbG_r)(\F_{q})_{\ss}$ is given by 
\[
\sum_{u \in \bbG_r(\F_{q})_{\unip}} Q_{\bbT_r}^{\bbG_r}(\theta_{+})(u) \overline{Q_{\bbT'_r}^{\bbG_r}(\theta_{+}')(u)}.
\]
Here, we used that our $\theta$ and $\theta'$ are chosen so that $\theta|_{Z(\bbG_{r})(\F_{q})_{\ss}}=\theta'|_{Z(\bbG_{r})(\F_{q})_{\ss}}=\mathbbm{1}$.

We next compute the contribution of $s\in\bbG_r(\F_{q})_{\ss} \smallsetminus Z(\bbG_r)(\F_{q})_{\ss}$.
Using the inductive hypothesis, it is given by
\[
\frac{1}{|\bbG_{r,s}(\F_q)| \cdot |\bbT_{r}(\F_{q})| \cdot |\bbT'_{r}(\F_{q})|} 
\sum_{\substack{g,g' \in \bbG_r(\F_{q}) \\ s^{g} \in \bbT_{r}(\F_{q}) \\ s^{g'} \in \bbT'_{r}(\F_{q})}} {}^{g}\theta(s) \overline{{}^{g'}\theta'(s)} \cdot\sum_{\substack{u \in \bbT^{g,+}_{r}(\F_{q}) \\ n_1 \in N_{\bbG_{r,s}(\F_q)}(\bbT^{g}_{r}, \bbT^{\prime g'}_{r})}} {}^{g}\theta_{+}(u) \overline{{}^{n_{1}^{-1}g'}\theta_{+}^{\prime}(u)}.
\]
By using Lemma \ref{lem:index-set-rewrite}, this becomes
\begin{align*}
&\frac{1}{|\bbG_{r,s}(\F_q)| \cdot |\bbT_{r}(\F_{q})| \cdot |\bbT'_{r}(\F_{q})|} 
\sum_{\substack{g\in \bbG_r(\F_{q}) \\ n\in N_{\bbG_r(\F_{q})}(\bbT_r, \bbT_r') \\ n_{1}\in\bbG_{r,s}(\F_q) \\ s^{g} \in \bbT_{r}(\F_{q})}} {}^{g}\theta(s) \overline{{}^{n_{1}gn^{-1}}\theta'(s)} \cdot\sum_{u \in \bbT^{g,+}_{r}(\F_{q})} {}^{g}\theta_{+}(u) \overline{{}^{gn^{-1}}\theta_{+}^{\prime}(u)}\\
&=
\frac{1}{|\bbT_{r}(\F_{q})| \cdot |\bbT'_{r}(\F_{q})|} 
\sum_{\substack{g\in \bbG_r(\F_{q}) \\ n\in N_{\bbG_r(\F_{q})}(\bbT_r, \bbT_r') \\ s^{g} \in \bbT_{r}(\F_{q})}} {}^{g}\theta(s) \overline{{}^{gn^{-1}}\theta'(s)} \cdot\sum_{u \in \bbT^{g,+}_{r}(\F_{q})} {}^{g}\theta_{+}(u) \overline{{}^{gn^{-1}}\theta_{+}^{\prime}(u)}\\
&=
\frac{1}{|\bbT_{r}(\F_{q})| \cdot |\bbT'_{r}(\F_{q})|} 
\sum_{\substack{g\in \bbG_r(\F_{q}) \\ n\in N_{\bbG_r(\F_{q})}(\bbT_r, \bbT_r') \\ u \in \bbT^{g,+}_{r}(\F_{q})\\ s^{g} \in \bbT_{r}(\F_{q})}} {}^{g}\theta(su) \overline{{}^{gn^{-1}}\theta'(su)}.
\end{align*}
Thus, the sum over $\bbG_r(\F_{q})_{\ss} \smallsetminus Z(\bbG_r)(\F_{q})_{\ss}$ is 
\begin{align*}
\frac{1}{|\bbT_{r}(\F_{q})| \cdot |\bbT'_{r}(\F_{q})|} 
\sum_{s\in\bbG_r(\F_{q})_{\ss} \smallsetminus Z(\bbG_r)(\F_{q})_{\ss}}
\sum_{\substack{g\in \bbG_r(\F_{q}) \\ n\in N_{\bbG_r(\F_{q})}(\bbT_r, \bbT_r') \\ u \in \bbT^{g,+}_{r}(\F_{q})\\ s^{g} \in \bbT_{r}(\F_{q})}} {}^{g}\theta(su) \overline{{}^{gn^{-1}}\theta'(su)}.
\end{align*}
By furthermore noting that
\[
\{(s,g)\in\bbG_{r}(\F_{q})_{\ss}\times\bbG_{r}(\F_{q}) \mid s^{g}\in\bbT_{r}(\F_{q}) \}\rightarrow\bbT_{r}(\F_{q})_{\ss}\colon (s,g)\mapsto s^{g}
\]
is a surjection such that each fiber is of order $|\bbG_r(\F_{q})|$ and the subset $Z(\bbG_{r})(\F_{q})_{\ss}\times\bbG_{r}(\F_{q})$ is mapped to $Z(\bbG_{r})(\F_{q})_{\ss}$, we see that the above equals 
\[
\frac{|\bbG_{r}(\F_{q})|}{|\bbT_{r}(\F_{q})| \cdot |\bbT'_{r}(\F_{q})|} 
\sum_{s\in\bbT_r(\F_{q})_{\ss} \smallsetminus Z(\bbG_r)(\F_{q})_{\ss}}
\sum_{\substack{n\in N_{\bbG_r(\F_{q})}(\bbT_r, \bbT_r') \\ u \in \bbT^{+}_{r}(\F_{q})}} \theta(su) \overline{\theta^{\prime n}(su)}.
\]
Here, again using that $\theta|_{Z(\bbG_{r})(\F_{q})_{\ss}}=\theta'|_{Z(\bbG_{r})(\F_{q})_{\ss}}=\mathbbm{1}$, we may think of this as 
\[
\frac{|\bbG_{r}(\F_{q})|}{|\bbT_{r}(\F_{q})| \cdot |\bbT'_{r}(\F_{q})|} 
\Biggl(\sum_{\substack{t \in \bbT_{r}(\F_{q})\\n\in N_{\bbG_r(\F_{q})}(\bbT_r, \bbT_r')}} \theta(t) \overline{\theta^{\prime n}(t)}-|Z(\bbG_{r})(\F_{q})_{\ss}|\cdot\sum_{\substack{u \in \bbT^{+}_{r}(\F_{q})\\n\in N_{\bbG_r(\F_{q})}(\bbT_r, \bbT_r')}} \theta_{+}(u)\overline{\theta_{+}^{\prime n}(u)}\Biggr),
\]
which equals
\begin{multline*}
|\bbG_r(\F_{q})| \cdot |\{w \in W_{\bbG_{r}(\F_{q})}(\bbT_r, \bbT_r') \mid \theta=\theta^{\prime n}\}|\\
-\frac{|Z(\bbG_{r})(\F_{q})_{\ss}|\cdot|\bbG_{r}(\F_{q})|}{|\bbT_{r}(\F_{q})| \cdot |\bbT'_{r}(\F_{q})|}\cdot\sum_{\substack{u \in \bbT^{+}_{r}(\F_{q})\\n\in N_{\bbG_r(\F_{q})}(\bbT_r, \bbT_r')}} \theta_{+}(u)\overline{\theta_{+}^{\prime n}(u)}.
\end{multline*}

Therefore, $|\bbG_r(\F_{q})|\cdot\langle R_{\bbT_r}^{\bbG_r}(\theta), R_{\bbT_r'}^{\bbG_r}(\theta') \rangle$ is equal to
\begin{multline*}
|Z(\bbG_{r})(\F_{q})_{\ss}|\cdot\sum_{u \in \bbG_r(\F_{q})_{\unip}} Q_{\bbT_r}^{\bbG_r}(\theta_{+})(u) \overline{Q_{\bbT'_r}^{\bbG_r}(\theta_{+}')(u)}\\
+|\bbG_r(\F_{q})| \cdot |\{w \in W_{\bbG_{r}(\F_{q})}(\bbT_r, \bbT_r') \mid \theta=\theta^{\prime n}\}|\\
-\frac{|Z(\bbG_{r})(\F_{q})_{\ss}|\cdot|\bbG_{r}(\F_{q})|}{|\bbT_{r}(\F_{q})| \cdot |\bbT'_{r}(\F_{q})|}\cdot\sum_{\substack{u \in \bbT^{+}_{r}(\F_{q})\\n\in N_{\bbG_r(\F_{q})}(\bbT_r, \bbT_r')}} \theta_{+}(u)\overline{\theta_{+}^{\prime n}(u)}.
\end{multline*}
However, by Theorem \ref{thm:Cha24}, the second term $|\bbG_r(\F_{q})| \cdot |\{w \in W_{\bbG_{r}(\F_{q})}(\bbT_r, \bbT_r') \mid \theta=\theta^{\prime n}\}|$ is nothing but $|\bbG_r(\F_{q})|\cdot\langle R_{\bbT_r}^{\bbG_r}(\theta), R_{\bbT_r'}^{\bbG_r}(\theta') \rangle_{\bbG_r(\F_{q})}$.
This implies that we necessarily have
\[
\sum_{u \in \bbG_r(\F_{q})_{\unip}} Q_{\bbT_r}^{\bbG_r}(\theta_{+})(u) \overline{Q_{\bbT'_r}^{\bbG_r}(\theta_{+}')(u)}
=\frac{|\bbG_{r}(\F_{q})|}{|\bbT_{r}(\F_{q})| \cdot |\bbT'_{r}(\F_{q})|}\cdot\sum_{\substack{u \in \bbT^{+}_{r}(\F_{q})\\n\in N_{\bbG_r(\F_{q})}(\bbT_r, \bbT_r')}} \theta_{+}(u)\overline{\theta_{+}^{\prime n}(u)}.\qedhere
\]
\end{proof}

\section{Green functions for Yu's construction}\label{sec:Yu green}

We will first establish that in fact certain linear combinations of representations have a character formula in terms of a positive-depth Green function $\sfQ$ defined using the Fintzen--Kaletha--Spice twist of Yu's construction (Theorem \ref{thm:green Yu}). This may be of independent interest. We note that our aim here is not to give an explicit character formula for these representations, but just to illuminate the structural behavior that the character value at an element $g$ with topological Jordan decomposition $g = su$ is determined by behavior on $s$ (coming from the depth-zero part of $\theta$) and on $u$ (determined by the Green function $\sfQ$). 

The significance of Theorem \ref{thm:green Yu} in the present paper is based in the simple observation that for \textit{any} character $\theta$, there exists a \textit{regular} character $\theta'$ for which $\theta|_{T_{0+}} = \theta'|_{T_{0+}}$. This, together with the content of Section \ref{sec:green} (especially Theorem \ref{thm:geom-char-formula}), allows us to deduce an explicit description (Theorem \ref{thm:gen comparison}) of $R_{\bbT_r}^{\bbG_r}(\theta)$ using our comparison result for $R_{\bbT_r}^{\bbG_r}(\theta')$ (Theorem \ref{thm:reg comparison}).

Let $(\bfT,\theta)$ be an unramified elliptic pair (note that $\theta$ need not be regular). 
Assume that $p \neq 2$ is not bad for $\bfG$ and $p \nmid |\pi_1(\bfG_{\der})| \cdot |\pi_1(\widehat \bfG_{\der})|$.

\subsection{FKS--Yu virtual representations}\label{subsec:Yu green}

The assumption on $p$ guarantees the existence of a Howe factorization of $\theta$, which yields an associated \textit{clipped Yu datum} $\dashover{\Psi}:=(\vec{\bfG},\vec{\phi},\vec{r},\x)$ (i.e., a Yu datum without the depth-zero part) as well as a depth-zero character $\phi_{-1}$ of $T$. By construction, $\theta=\prod_{i=-1}^{d}\phi_{i}|_T$.
Note that $\dashover{\Psi}$ depends only on $\theta|_{T_{0+}}$.

Let us consider the Deligne--Lusztig representation $R_{\bbT_0}^{\bbG^0_0}(\phi_{-1})$ associated to $(\bbG^0_0,\bbT_0,\phi_{-1})$.
We write $R_{\bbT_0}^{\bbG^0_0}(\phi_{-1})=\sum_{j=1}^{n}m_j\cdot\rho_{0,j}$ for the irreducible decomposition of $R_{\bbT_0}^{\bbG^0_0}(\phi_{-1})$, hence $m_j\in\Z$ and $\rho_{0,j}$'s are pairwise inequivalent irreducible representations of $\bbG^0_0(\F_q)$.
For each $\rho_{0,j}$, we put $\Psi_j:=(\vec{\bfG},\vec{\phi},\vec{r},\x,\rho_{0,j})$.
Note that this quintuple may not be a Yu datum as $\rho_{0,j}$ may not be cuspidal.
However, we can still apply the (Fintzen--Kaletha--Spice twist of) Yu's construction to $\Psi_j$ to get a representation $\cc\tau^{\FKS}_{\Psi_j}$ of $G_{\x,0}$.
We define a virtual smooth representation $\cc\tau^\FKS_{(\bfT,\theta)}$ of $G_{\x,0}$ by
\[
    \cc\tau^\FKS_{(\bfT,\theta)}
    :=
    (-1)^{r(\bfG^0)-r(\bfT)}\cdot\sum_{j=1}^n m_j\cdot \cc\tau^{\FKS}_{\Psi_j}.
\]

In the following, we loosely identify elements of $G_{\x,0}$ with their images in $\bbG_r(\F_q)$.

\begin{definition}\label{def:green Yu}
    The \textit{positive-depth FKS--Yu Green function} associated to an unramified elliptic pair $(\bfT,\theta)$ is
    \begin{equation*}
        \sfQ_{\bbT_r}^{\bbG_r}(\theta_+) \from \bbG_r(\F_q)_{\unip} \to \C; \qquad u \mapsto \Theta_{\cc{\tau}^\FKS_{(\bfT,\theta)}}(u).
    \end{equation*}
\end{definition}

The next result shows that the character formula of the virtual representation $\cc{\tau}^\FKS_{(\bfT,\theta)}$ can be written in terms of the positive-depth FKS--Yu Green function; that is, it shares the same shape of character formula as the Deligne--Lusztig character formula for positive-depth Deligne--Lusztig induction (Theorem \ref{thm:geom-char-formula}). 

For a topologically semisimple element $s \in T_0$, the connected centralizer of $s$ in the $i$th twisted Levi $\bfG^i$ in the Howe factorization sequence of $\theta$ with respect to $\bfT \hookrightarrow \bfG$ also coincides with the $i$th twisted Levi in the Howe factorization sequence of $\theta$ with respect to $\bfT \hookrightarrow \bfG_s$. Hence the notation $\bfG_s^i$ for this group is unambiguous, and we get a sequence $\vec \bfG_s$. We define $r_s(\bfT,\theta)$ in the same way as $r(\bfT,\theta)$ (\cite[Proposition 4.9]{CO25}), but for $\bfG_s$ (not $\bfG$).

\begin{thm}\label{thm:green Yu} \mbox{}
    Let $(\bfT,\theta)$ any unramified elliptic pair.
    \begin{enumerate}
        \item     For $g \in \bbG_r(\F_q)$ with Jordan decomposition $g = su$, 
        \[
            \Theta_{\cc\tau^{\FKS}_{(\bfT,\theta)}}(g)
            =
            (-1)^{r(\bfG,s,\bfT,\theta)}
            \cdot\sum_{\begin{subarray}{c} x\in \bbG_{r}(\F_q)/\bbG_{s,r}(\F_q) \\{}^{x}s\in \bbT_r(\F_q)\end{subarray}}\theta({}^{x}s)
            \cdot
            \sfQ_{\bbT_r}^{\bbG_{{}^{x}s,r}}(\theta_+)({}^{x}u),
        \]    
        where we put $r(\bfG,s,\bfT,\theta):=r(\bfG^0)-r(\bfG^{0}_s)+r(\bfT,\theta)-r_s(\bfT,\theta)$.
        \item $\sfQ_{\bbT_r}^{\bbG_r}(\theta_+)$ only depends on $\theta_+=\theta|_{T_{0+}}$.
    \end{enumerate}
\end{thm}

We discuss the proof of Theorem \ref{thm:green Yu}, and necessary background to state the proof, in the next subsection.

\subsection{Deligne--Lusztig character formula for FKS--Yu virtual representations}\label{subsec:green Yu proof}

We first briefly review the construction of $\cc\tau_{(\bfT,\theta)}^\FKS$ (there are various versions of possible explanations; here, we follow \cite{HM08}).

Associated to $\vec{\bfG}=(\bfG^0\subsetneq\cdots\subsetneq\bfG^d)$ and $\vec{r}=(0\leq r_0<\cdots<r_{d-1}\leq r_d)$, we get ``Yu's subgroup''
\[
\cc{K}^{i}:= G^0_{\x,0}G^1_{\x,s_0}\cdots G^i_{\x,s_{i-1}},
\]
where $s_i:=r_i/2$ for each $i=1,\ldots,d$.
We put $\cc{K}^{0}:=G^0_{\x,0}$.
Thus, by putting $J^{i+1}:=(G^{i},G^{i+1})_{\x,(r_{i},s_{i})}$, we have $\cc{K}^{i+1}=\cc{K}^{i}J^{i+1}$ for $0\leq i \leq d-1$.
The point of this construction is that by appealing to the commutator product on $J^{i+1}$ and considering  the character $\phi_i \from G^i \to \C^\times$ of the sequence $\vec \phi = (\phi_0\ldots,\phi_d)$, the quotient group $V^{i+1}:=J^{i+1}/J^{i+1}_+$ can be equipped with a symplectic structure. Its associated Heisenberg group $H(V_{i+1})$ can be regarded as a quotient of $J^{i+1}/(G^{i},G^{i+1})_{\x,(r_{i}+,s_{i}+)}$. The genericity conditions on the character $\phi_i$ imply that it uniquely determines an irreducible representation $\omega_{i}$ of $\Sp(V_{i+1})\ltimes H(V_{i+1})$ by the Stone--von Neumann theorem.
Since the conjugation action of $\cc{K}^{i}$ on $\cc{K}^{i+1}$ induces a symplectic action on $\cc{K}^{i+1}/\cc{K}^{i}J^{i+1}_+\cong V_{i+1}$, we have a homomorphism $\cc{K}^{i}\ltimes J^{i+1}\rightarrow \Sp(V_{i+1})\ltimes H(V_{i+1})$; we again write $\omega_{i}$ for the pullback of $\omega_{i}$ along this map. 
By inflating $\phi_i|_{\cc{K}^i}$ to $\cc{K}^{i}\ltimes J^{i+1}$, we take the tensor product $\omega_{i}\otimes(\phi_i|_{\cc{K}^i})$ as representations of $\cc{K}^{i}\ltimes J^{i+1}$.
Then, in fact, this representation factors through the multiplication map $\cc{K}^{i}\ltimes J^{i+1}\twoheadrightarrow \cc{K}^{i}J^{i+1}=\cc{K}^{i+1}$. 
This $\cc K^{i+1}$-representation factors through $\cc K^{i+1}/G_{\x,r_{i}+}^{i+1} = \cc K^{i+1} G_{\x,r_i+}/G_{\x,r_i+}$ and therefore can be extended to a representation of $\cc K^{i+1} (G^i, G)_{\x,(r_i+,s_i+)} = \cc K^{i+1} G_{\x,s_i+}  \supseteq \cc K^d$; we set $\cc K \coloneqq \cc K^d$ and define $\kappa_i$ to be the resulting $\cc K$-representation. 
Set $\kappa_d:=\phi_d|_{\cc{K}_d}$ and define the $\cc{K}$-representation
\begin{equation*}
    \kappa_+ \coloneqq \kappa_0 \otimes \cdots \otimes \kappa_d.
\end{equation*}

For each $1 \leq j \leq n$, we define the $\cc{K}$-representation 
\[
    \cc\rho_{\Psi_j}^\FKS\coloneqq \rho_{0,j} \otimes \kappa_+ \otimes \epsilon =  \rho_{0,j}\otimes\kappa_0\otimes\cdots\otimes\kappa_d\otimes\epsilon_{\Psi_j},
\]
where $\rho_{0,j}$ is regarded as a representation of $\cc{K}$ by the inflation and $\epsilon_{\Psi_j}$ is the sign character of Fintzen--Kaletha--Spice (see Section \ref{subsec:ADS}).
Note that $\epsilon$ is determined by the clipped datum $\dashover{\Psi}=(\vec{\bfG},\vec{\phi},\vec{r},\x)$, hence independent of the choice of $\rho_{0,j}$.
Also note that $\epsilon_{\Psi_j} = \varepsilon^\ram[\theta]|_{T_0}$.

We let $\cc{K}_{0+}:=\cc{K}\cap G_{\x,0+}$; note that then $\cc{K}=G^0_{\x,0}\cc{K}_{0+}$.

\begin{lem}\label{lem:separation}
    Let $g\in \cc{K}$ be any element and let $g=su$ be its topological Jordan decomposition.
    Then, after replacing $g$ with its $\cc{K}_{0+}$-conjugate if necessary, we have $s\in G^0_{\x,0}$.
\end{lem}

\begin{proof}
    This follows from the topological semisimplicity of $s$ (cf.\ \cite[Lemma 5.7]{Oi23-TECR}).
\end{proof}

Suppose that $s\in T_0$ is a topologically semisimple element. 
Then, by replacing the clipped datum $(\vec{\bfG},\vec{\phi},\vec{r},\x)$ with $(\vec{\bfG}_s,\vec{\phi},\vec{r},\x)$, we can perform the previous construction to get a representation $\kappa_{s,+}:=\kappa_{s,0}\otimes\cdots\otimes\kappa_{s,d}$ of the group $K_s:=G^{0}_{s,\x,0}G^{1}_{s,\x,s_0}\cdots G^{d}_{s,\x,s_{d-1}}$.

Note that the conjugate action of $T$ preserves each $J^{i+1}$ and $J^{i+1}_+$, hence induces an action on $V_{i+1}$, which is in fact symplectic.
We recall a description of this action due to Adler--Spice (\cite[Proof of Proposition 3.8]{AS09}; see also \cite[Section 6.1]{Oi23-TECR}).
By fixing a finite unramified extension $E$ of $F$ splitting $\bfT$, we put
\[
\bfV_{i+1}\coloneqq \Lie(\bfG^{i},\bfG^{i+1})(E)_{\x,(r_i,s_i):(r_i,s_{i}+)}.
\]
Since 
\[
J^{i+1}/J^{i+1}_{+}
=
(G^{i},G^{i+1})_{\x,(r_i,s_i):(r_i,s_{i}+)}
=
(\bfG^{i},\bfG^{i+1})(E)_{\x,(r_i,s_i):(r_i,s_{i}+)}^{\Gamma},
\]
the exponential map $\Lie(\bfG^{i},\bfG^{i+1})(E)_{\x,(r_i,s_i):(r_i,s_{i}+)}\xrightarrow{\sim}(\bfG^{i},\bfG^{i+1})(E)_{\x,(r_i,s_i):(r_i,s_{i}+)}$ induces an identification
\[
\bfV_{i+1}^{\Gamma}\xrightarrow{\sim}V_{i+1}:=J^{i+1}/J^{i+1}_{+}.
\]
Let $\Phi(\bfG^i,\bfT)$ be the set of roots of $\bfT$ in $\bfG^i$.
For each $\alpha\in \Phi(\bfG^{i+1},\bfT)\smallsetminus \Phi(\bfG^i,\bfT)$, we let $\bfV_{\alpha}$ be the image of $\bmfg_{\alpha}(E)\cap\Lie(\bfG^{i},\bfG^{i+1})(E)_{\x,(r_i,s_i)}$ in $\bfV_{i+1}$, where $\bmfg_\alpha$ denotes the root space of $\alpha$ in $\bmfg$.
Then the root space decomposition $\bmfg\cong\bmft\oplus\bigoplus_{\alpha\in\Phi(\bfG,\bfT)}\bmfg_{\alpha}$ naturally induces a decomposition
\[
\bfV_{i+1}
=
\bigoplus_{\alpha\in \Phi(\bfG^{i+1},\bfT)\smallsetminus \Phi(\bfG^i,\bfT)}\bfV_{\alpha}.
\]
For each $\alpha\in \Phi(\bfG^{i+1},\bfT)\smallsetminus \Phi(\bfG^i,\bfT)$, we put $V_{\alpha}\coloneqq \bfV_{\alpha}^{\Gamma_{\alpha}}$, where $\Gamma_{\alpha}$ denotes the stabilizer of $\alpha$ in $\Gamma$ (note that $\bfV_{\alpha}$ and $V_{\alpha}$ might be zero depending on $\alpha$).
We define a subset $\Xi_{i+1}$ of $\Phi(\bfG^{i+1},\bfT)\smallsetminus \Phi(\bfG^i,\bfT)$ by
\[
\Xi_{i+1}\coloneqq \{\alpha\in\Phi(\bfG^{i+1},\bfT)\smallsetminus \Phi(\bfG^i,\bfT)\mid V_{\alpha}\neq0\}.
\]
We remark that, for any $\alpha\in\Xi_{i+1}$, the space $V_{\alpha}$ is (noncanonically) isomorphic to the residue field $k_{\alpha}$ of $F_{\alpha}:=(F^\sep)^{\Gamma_\alpha}$.
Also note that $\Xi_{i+1}$ is preserved by the action of $\Sigma:=\Gamma\times\{\pm1\}$ on $\Phi(\bfG,\bfT)$.
We write $\dot{\Xi}_{i+1}$ and $\ddot{\Xi}_{i+1}$ for $\Xi_{i+1}/\Gamma$ and $\Xi_{i+1}/\Sigma$, respectively.

For $\Gamma\alpha\in\dot{\Xi}_{i+1}$, we put
\[
\bfV_{\Gamma\alpha}\coloneqq \bigoplus_{\beta\in\Gamma\alpha}\bfV_{\beta}
\quad\text{and}\quad
V_{\Gamma\alpha}\coloneqq \bfV_{\Gamma\alpha}^{\Gamma}.
\]
Then, for any $\Gamma\alpha\in\dot{\Xi}_{i+1}$, we have
\[
V_{\alpha}
\xrightarrow{\sim}
V_{\Gamma\alpha}=\Bigl(\bigoplus_{\beta\in\Gamma\alpha}\bfV_{\beta}\Bigr)^{\Gamma}
\colon X_{\alpha} \mapsto \sum_{\sigma\in\Gamma/\Gamma_{\alpha}}\sigma(X_{\alpha}).
\]
Therefore we get
\begin{align}\label{eq:decomp}
V_{i+1}
\cong
\bigoplus_{\Gamma\alpha\in\dot{\Xi}_{i+1}}V_{\Gamma\alpha}
=
\bigoplus_{\Sigma\alpha\in\ddot{\Xi}_{i+1}}V_{\Sigma\alpha},
\end{align}
where we put $V_{\Sigma\alpha}\coloneqq V_{\Gamma\alpha}\oplus V_{-\Gamma\alpha}$ for asymmetric $\alpha$ (i.e., $-\alpha\notin\Gamma\alpha$) and $V_{\Sigma\alpha}\coloneqq V_{\Gamma\alpha}$ for symmetric $\alpha$ (i.e., $-\alpha\in\Gamma\alpha$).
This decomposition $V_{i+1}\cong \bigoplus_{\Sigma\alpha\in\ddot{\Xi}_{i+1}}V_{\Sigma\alpha}$ is in fact an orthogonal decomposition into symplectic subspaces.

We have a similar description for the groups associated to $(\vec{\bfG}_s,\vec{\phi},\vec{r},\x)$.
We use the subscript ``$s$'' to denote the relevant groups.
For example, we put $V_{s,i+1}:=J_s^{i+1}/J_{s,+}^{i+1}=(G_s^{i},G_s^{i+1})_{\x,(r_i,s_i):(r_i,s_{i}+)}$, $\Xi_{s,i+1}:=\{\alpha\in\Phi(\bfG_s^{i+1},\bfT)\smallsetminus \Phi(\bfG_s^{i},\bfT)\mid V_{\alpha}\neq0\}$, and so on.
Note that $V_{s,i+1}$ is identified with the symplectic subspace of $V_{i+1}$ stabilized by the conjugate action of $s$, or equivalently, the sum of $V_{\Sigma\alpha}$'s for $\Sigma\alpha\in\ddot{\Xi}_{i+1}$ satisfying $\alpha(s)=1$. 
Hence, by putting $V_{i+1}^{s\neq1}$ to be the sum of $V_{\Sigma\alpha}$'s for $\Sigma\alpha\in\ddot{\Xi}_{i+1}$ satisfying $\alpha(s)\neq1$, we get an orthogonal decomposition
\[
    V_{i+1}
    =V_{s,i+1}\oplus V_{i+1}^{s\neq1}
    =\Bigl(\bigoplus_{\begin{subarray}{c}\Sigma\alpha\in\ddot{\Xi}_{i+1}\\ \alpha(s)=1\end{subarray}} V_{\Sigma\alpha}\Bigr)\oplus\Bigl(\bigoplus_{\begin{subarray}{c}\Sigma\alpha\in\ddot{\Xi}_{i+1}\\ \alpha(s)\neq1\end{subarray}} V_{\Sigma\alpha}\Bigr).
\]

We write
\begin{equation*}
    \theta_{\geq 0} \coloneqq \prod_{i=0}^d \phi_i|_{T_0}, \qquad 
    \theta_+:=\theta|_{T_{0+}} = \prod_{i=-1}^d \phi_i|_{T_{0+}} = \prod_{i=0}^d \phi_i|_{T_{0+}} = \theta_{\geq 0}|_{T_{0+}}. 
\end{equation*}

\begin{prop}\label{prop:descent}
    For $g\in \cc{K}$ with topological Jordan decomposition $g=su$ such that $s \in T_0$,
    \[
        \Theta_{\kappa_{+}}(su)
        =
        (-1)^{r(\bfT,\theta)-r_s(\bfT,\theta)}
        \cdot
        \varepsilon^\ram[\theta](s)
        \cdot 
        \theta_{\geq0}(s)
        \cdot
        \Theta_{\kappa_{s,+}}(u).
    \]
\end{prop}

\begin{proof}
    By recalling that $\cc{K}=G^0_{\x,0}J^1\cdots J^d$, we choose elements $u_0\in G^0_{\x,0}$ and $u_i\in J^i$ for each $1\leq i \leq d$ such that $u=u_0u_1\cdots u_d$.
    Since $\kappa_+=\kappa_0\otimes\cdots\otimes\kappa_d$, we have $\Theta_{\kappa_+}(g)=\prod_{i=0}^{d}\Theta_{\kappa_i}(g)$.
    We compute each $\Theta_{\kappa_i}(g)$.

    Recall that, for each $0\leq i \leq d-1$, $\kappa_i$ is the inflation from $\cc{K}^{i+1}$ to $\cc{K}$ of the descent of $\omega_i\otimes(\phi_i)|_{\cc{K}^i}$ along the map $\cc{K}^{i}\ltimes J^{i+1}\twoheadrightarrow \cc{K}^{i}J^{i+1}=\cc{K}^{i+1}$.
    Hence we have
    \begin{align*}
        \Theta_{\kappa_i}(su)
        =\Theta_{\kappa_i}(su_{0}\cdots u_{i+1})
        &=\Theta_{\omega_i\otimes(\phi_i)|_{\cc{K}^i}}(su_{0}\cdots u_{i}\ltimes u_{i+1})\\
        &=\Theta_{\omega_i}(su_{0}\cdots u_{i}\ltimes u_{i+1})\cdot\phi_i(su_{0}\cdots u_{i}).
    \end{align*}
    As the conjugation action of $\cc{K}_{0+}\cap\cc{K}^{i+1}$ on $V_{i+1}$ is trivial, 
    \[
        \Theta_{\omega_i}(su_{0}\cdots u_{i}\ltimes u_{i+1})
        =
        \Theta_{\omega_i}(su_{0}\ltimes u_{i+1}).
    \]

    Let us write the image of $u_{i+1}\in J^{i+1}$ in $V_{i+1}\cong\bigoplus_{\Sigma\alpha\in\ddot{\Xi}_{i+1}}V_{\Sigma\alpha}$ as $(u_{\alpha})_{\alpha}$.
    Since $su=us$, $(u_{\alpha})_{\alpha}$ is invariant under the $s$-conjugation.
    However, the $s$-conjugation is given by $\alpha(s)$-multiplication on each $V_{\Sigma\alpha}$.
    Hence we necessarily have $u_\alpha=0$ for any $\Sigma\alpha\in\ddot{\Xi}_{i+1}$ such that $s(\alpha)\neq1$.
    In other words, the image of $u_{i+1}\in J^{i+1}$ in $V_{i+1}$ lies in $V_{s,i+1}$.
    Thus the image of $u_{i+1}\in J^{i+1}$ in the Heisenberg quotient $H(V_{i+1})$ lies in $H(V_{s,i+1})$.
    
    On the other hand, the image of $s$ in $\Sp(V_{i+1})$ is contained in the subgroup $\prod_{\alpha\in\ddot{\Xi}_{i+1}}\Sp(V_{\Sigma\alpha})$; the image is given by $(\alpha(s))_{\alpha\in\ddot{\Xi}_{i+1}}$.
    As $u_0$ commutes with $s$, the image of $u_0$ in $\Sp(V_{i+1})$ is contained in $\Sp(V_{s,i+1})\times\Sp(V_{i+1}^{s\neq1})$.
    Let us write $(u'_0,u''_0)$ for this image.

    Now recall the following general fact (see \cite[2.5]{Ger77}): 
    \begin{quote}
        We fix a nontrivial additive character $\psi$ of $\F_p$.
        If $W$ and $W'$ are Heisenberg groups over $\F_p$, then the restriction of the Heisenberg--Weil representation $\omega_{W\oplus W',\psi}$ of $\Sp(W\oplus W')\ltimes H(W\oplus W')$ with central character $\psi$ to $(\Sp(W)\ltimes H(W))\times (\Sp(W')\ltimes H(W'))$ is isomorphic to the tensor product $\omega_{W,\psi}\otimes\omega_{W',\psi}$ of the Heisenberg--Weil representations with central character $\psi$.
    \end{quote}
    Therefore, if we let $\omega_{s,i}$ and $\omega_{i}^{s\neq1}$ denote the Heisenberg--Weil representations associated to the symplectic subspaces $V_{s,i+1}$ and $V_{i+1}^{s\neq1}$, then we have 
    \[
        \Theta_{\omega_i}(su_{0}\ltimes u_{i+1})
        =
        \Theta_{\omega_{s,i}}(u'_{0}\ltimes u_{i+1})
        \cdot
        \Theta_{\omega_{i}^{s\neq1}}(su''_{0}\ltimes 1).
    \]
    Since $g=su$ is the topological Jordan decomposition, the images of $s$ and $u$ in $G^0_{\x,0:0+}$ gives the Jordan decomposition.
    In particular, the product $s\cdot u''_0\in \Sp(V_{i+1}^{s\neq1})$ is also the Jordan decomposition.
    As the $s$-action on $V_{i+1}^{s\neq1}$ does not have any non-zero fixed vector, the descent formula of the character of the Heisenberg--Weil representation \cite[Lemma 8.2.1]{Spi21} implies that $\Theta_{\omega_{i}^{s\neq1}}(su''_{0}\ltimes 1)=\Theta_{\omega_{i}^{s\neq1}}(s\ltimes 1)$.

    In summary, we obtain
    \begin{align*}
        \Theta_{\kappa_i}(su)
        &=
        \Theta_{\omega_{s,i}}(u'_{0}\ltimes u_{i+1})
        \cdot
        \Theta_{\omega_{i}^{s\neq1}}(s\ltimes 1)
        \cdot\phi_i(su_{0}\cdots u_{i})\\
        &=
        \Theta_{\kappa_{s,i}}(u)
        \cdot
        \Theta_{\omega_{i}^{s\neq1}}(s\ltimes 1)
        \cdot\phi_i(s)
    \end{align*}
    for each $0\leq i \leq d-1$.
    As $\kappa_d=\phi_d$, we have $\kappa_d(su)=\phi_d(s)\phi_{d}(u)=\phi_d(s)\kappa_{s,d}(u)$.
    Therefore, we finally get
    \begin{align*}
        \Theta_{\kappa_{+}}(su)
        &=\prod_{i=0}^{d-1}\Bigl(\Theta_{\kappa_{s,i}}(u)
        \cdot
        \Theta_{\omega_{i}^{s\neq1}}(s\ltimes 1)
        \cdot\phi_i(s)\Bigr)\cdot \phi_d(s)\cdot\kappa_{s,d}(u)\\
        &=
        \prod_{i=0}^{d}\Theta_{\kappa_{s,i}}(u)
        \cdot
        \prod_{i=0}^{d-1}\Theta_{\omega_{i}^{s\neq1}}(s\ltimes 1)
        \cdot \prod_{i=0}^{d}\phi_i(s)\\
        &=\Theta_{\kappa_{s,+}}(u)
        \cdot
        \prod_{i=0}^{d-1}\Theta_{\omega_{i}^{s\neq1}}(s\ltimes 1)
        \cdot \theta_{\geq0}(s).
    \end{align*}
        
    Note that $\Theta_{\omega_{i}^{s\neq1}}(s\ltimes 1)$ is exactly the source where the Adler--DeBacker--Spice sign character (for unramified roots) comes from (see \cite[Proposition 3.8]{AS09}, \cite[Proposition 4.21]{DS18}, and also \cite[Section 4.2]{CO25}).
    To be more precise, we have 
    \[
        \prod_{i=0}^{d-1}\Theta_{\omega_{i}^{s\neq1}}(s\ltimes 1)
        =
        (-1)^{r(\bfT,\theta)-r_s(\bfT,\theta)}
        \cdot
        \varepsilon^\ram[\theta](s).
    \]
    Hence we get the asserted identity.
\end{proof}

\begin{proof}[Proof of Theorem \ref{thm:green Yu}]
    We fix $g\in\bbG_r(\F_q)$ with Jordan decomposition $g=su$, or equivalently, $g\in G_{\x,0}$ with a topological Jordan decomposition $g=su$.

    Recall that each $\cc\tau^{\FKS}_{\Psi_j}$ is obtained by inducing $\cc\rho^{\FKS}_{\Psi_j}$ from $\cc K$ to $TG_{\x,0}$.
    Hence, by the Frobenius character formula, we have
    \[
        \Theta_{\cc\tau^{\FKS}_{\Psi_j}}(g)
        =
        \sum_{\begin{subarray}{c} x\in \cc{K}\backslash G_{\x,0} \\{}^{x}g\in \cc{K}\end{subarray}} \Theta_{\cc\rho^{\FKS}_{\Psi_j}}({}^{x}g),
    \]
    which implies that 
    \[
        \Theta_{\cc\tau^{\FKS}_{(\bfT,\theta)}}(g)
        =
        \sum_{\begin{subarray}{c} x\in \cc{K}\backslash G_{\x,0} \\{}^{x}g\in \cc{K}\end{subarray}} \Theta_{\cc\rho^{\FKS}_{(\bfT,\theta)}}({}^{x}g),
    \]
    where $\cc\rho^{\FKS}_{(\bfT,\theta)}$ is defined in a similar manner to $\cc\tau^{\FKS}_{(\bfT,\theta)}$.
    As reviewed above, each $\cc\rho^{\FKS}_{\Psi_j}$ is defined to be $\rho_{0,j}\otimes \kappa_{+}\otimes\epsilon$.
    Thus $\Theta_{\cc\rho^{\FKS}_{(\bfT,\theta)}}$ is expressed as the product of $\Theta_{\rho_{0,j}}$, $\Theta_{\kappa_{+}}$, and $\Theta_{\epsilon}$.

    Suppose that $x\in G_{\x,0}$ satisfies ${}^{x}g\in \cc{K}$.
    Then, by Lemma \ref{lem:separation}, we may suppose (by replacing $x$ with an element of $\cc{K}_{0+}x$ if necessary) that ${}^{x}s\in G^0_{\x,0}$ and ${}^{x}u\in\cc{K}$.
    If we write ${}^{x}u=({}^{x}u)_0\cdot({}^{x}u)_+$ according to the product expression $\cc{K}=G^0_{\x,0}\cc{K}_{0+}$, then we have
    \[
        \Theta_{\cc\rho^{\FKS}_{\Psi_j}}({}^{x}g)
        =
        \Theta_{\rho_{0,j}}({}^{x}s\cdot({}^{x}u)_0)
        \cdot\Theta_{\kappa_{+}}({}^{x}s\cdot{}^{x}u)\cdot\epsilon_{\Psi_j}({}^{x}s\cdot {}^{x}u).
    \]
    This implies that 
    \[
        \Theta_{\cc\rho^{\FKS}_{(\bfT,\theta)}}({}^{x}g)
        =
        (-1)^{r(\bfG^0)-r(\bfT)}\cdot
        \Theta_{R_{\bbT_0}^{\bbG_0}(\phi_{-1})}({}^{x}s\cdot({}^{x}u)_0)
        \cdot\Theta_{\kappa_{+}}({}^{x}s\cdot{}^{x}u)\cdot\epsilon_{\Psi_j}({}^{x}s).
    \]
    Here, we used that $\epsilon_{\Psi_j}$ is quadratic, hence trivial at any topologically unipotent element.
    Note that, the images of ${}^{x}s$ and $({}^{x}u)_0$ under the map $\cc{K}\twoheadrightarrow\bbG^0_0(\F_q)$ gives the Jordan decomposition of the image of ${}^{x}s\cdot ({}^{x}u)_0$.
    Indeed, the orders of ${}^{x}s$ and $({}^{x}u)_0$ are prime-to-$p$ and $p$-power, respectively.
    Moreover, since we have ${}^{x}s\cdot {}^{x}u={}^{x}u\cdot {}^{x}s$ in $\cc{K}$, we get ${}^{x}s\cdot ({}^{x}u)_0=({}^{x}u)_0\cdot {}^{x}s$ in $\cc{K}/\cc{K}_{0+}\cong\bbG^0_{0}(\F_q)$.
    Therefore, by the Deligne--Lusztig character formula, we get
    \[
        \Theta_{R_{\bbT_0}^{\bbG_0}(\phi_{-1})}({}^{x}s\cdot({}^{x}u)_0)
        =
        \sum_{\begin{subarray}{c}y\in \bbG^0_0(\F_q) \\ {}^{y}({}^{x}s)\in \bbT_0(\F_q) \end{subarray}} \phi_{-1}({}^{y}({}^{x}s))\cdot Q^{\bbG_{{}^x s,0}^0}_{\bbT_0^{y}}(({}^{x}u)_0).
    \]
    Hence
    \[
        \Theta_{\cc\tau^{\FKS}_{(\bfT,\theta)}}(g)
        =
        (-1)^{r(\bfG^0)-r(\bfT)}
        \sum_{\begin{subarray}{c} x\in \cc{K}\backslash G_{\x,0} \\{}^{x}g\in \cc{K}\end{subarray}}
        \sum_{\begin{subarray}{c}y\in \bbG^0_0(\F_q) \\ {}^{y}({}^{x}s)\in \bbT_0(\F_q) \end{subarray}} \phi_{-1}({}^{y}({}^{x}s))\cdot Q^{\bbG_{{}^x s,0}^0}_{\bbT_0^{y}}(({}^{x}u)_0)
        \cdot\Theta_{\kappa_{+}}({}^{x}s\cdot{}^{x}u)\cdot\epsilon_{\Psi_j}({}^{x}s).
    \]

    As $\cc{K}=G^0_{\x,0}\cc{K}_{0+}$ and $\cc{K}/\cc{K}_{0+}\cong\bbG^0_0(\F_q)$, the right-hand side equals
    \[
        (-1)^{r(\bfG^0)-r(\bfT)}
        \sum_{\begin{subarray}{c} x\in \cc{K}_{0+}\backslash G_{\x,0} \\{}^{x}s\in \bbT_0(\F_q)\end{subarray}}
         \phi_{-1}({}^{x}s)\cdot Q^{\bbG_{{}^x s,0}^0}_{\bbT_0}(({}^{x}u)_0)
        \cdot\dot{\Theta}_{\kappa_{+}}({}^{x}s\cdot{}^{x}u)
        \cdot\epsilon_{\Psi_j}({}^{x}s),
    \]
    where we regard $s$ as an element of $\bbG_0(\F_q)$ (which naturally contains $\bbG^0_0(\F_q)$ and $\bbT_0(\F_q)$) in the index set and $\dot{\Theta}_{\kappa_{\theta_+}}$ denotes the zero extension of $\Theta_{\kappa_{\theta_+}}$ from $\cc{K}$ to $G_{\x,0}$.
    Thus, in other words, the index set is $\{x\in \cc{K}_{0+}\backslash G_{\x,0} \mid {}^{x}s\in T_{0}G_{\x,0+}\}$.
    To rewrite this set, we consider a natural map
    \[
        \{x\in G_{\x,0} \mid {}^{x}s\in T_{0}\}
        \rightarrow
        \{x\in \cc{K}_{0+}\backslash G_{\x,0} \mid {}^{x}s\in T_{0}G_{\x,0+}\}.
    \]
    By the same argument as in the proof of Lemma \ref{lem:separation}, we can check that this map is surjective.
    To investigate the fibers, let us suppose that $x\in G_{\x,0}$ and $k\in \cc{K}_{0+}$ satisfies ${}^{x}s\in T_0$ and ${}^{kx}s\in T_0$.
    Since ${}^{x}s$ and ${}^{kx}s$ are elements of $T_0$ which are conjugate, there exists an element $w\in W_{G}(\bfT)$ of the Weyl group satisfying ${}^{kx}s={}^{w}({}^{x}s)$, i.e., the $k$-conjugation on ${}^{x}s$ is given by the $w$-conjugation.
    However, since $k$ is pro-unipotent element and we assume $p \neq 2$ is not bad for $\bfG$, $w$ must be trivial.
    Hence $k$ belongs to $\cc{K}_{0+}\cap G_{{}^{x}s}=\cc{K}_{{}^{x}s,0+}$.
    Thus the each fiber is given by $\cc{K}_{{}^{x}s,0+}x=x\cc{K}_{s,0+}$.
    Similarly, we can also check that a natural map
    \[
        \{x\in G_{\x,0} \mid {}^{x}s\in T_{0}\}
        \rightarrow
        \{x\in G_{\x,r+}\backslash G_{\x,0} \mid {}^{x}s\in T_{0}G_{\x,r+}\}
    \]
    is surjective and that each fiber is given by $G_{{}^{x}s,\x,r+}x=xG_{s,\x,r+}$.
    Thus, combining these, we see that a natural map 
    \[
        \{x\in G_{\x,r+}\backslash G_{\x,0} \mid {}^{x}s\in T_{0}G_{\x,r+}\}
        \rightarrow
        \{x\in \cc{K}_{0+}\backslash G_{\x,0} \mid {}^{x}s\in T_{0}G_{\x,0+}\}
    \]
    is surjective and the order of each fiber is $|\cc{K}_{s,0+}/G_{s,\x,r+}|$.
    Hence we obtain 
    \[
        \Theta_{\cc\tau^{\FKS}_{(\bfT,\theta)}}(g)
        =
        \frac{(-1)^{r(\bfG^0)-r(\bfT)}}{|\cc{K}_{s,0+}/G_{s,\x,r+}|}\cdot
        \sum_{\begin{subarray}{c} x\in \bbG_{r}(\F_q) \\{}^{x}s\in \bbT_r(\F_q)\end{subarray}}
         \phi_{-1}({}^{x}s)\cdot Q^{\bbG_{{}^x s,0}^0}_{\bbT_0}(({}^{x}u)_0)
        \cdot\dot{\Theta}_{\kappa_{+}}({}^{x}s\cdot{}^{x}u)
        \cdot\epsilon_{\Psi_j}({}^{x}s).
    \]

    Note that, when $s$ is trivial, the right-hand side is determined only by $Q^{\bbG^0_0}_{\bbT_0}$, $\kappa_+$, and $\epsilon_{\Psi_j}$, all of which are independent of $\phi_{-1}$.
    Thus the positive-depth FKS--Yu Green function $\sfQ_{\bbT_r}^{\bbG_r}(\theta_+)=\Theta_{\cc{\tau}^\FKS_{(\bfT,\theta)}}|_{\bbG_r(\F_q)_\unip}$ depends only on $\theta_+=\theta|_{T_{0+}}$.

    Our remaining task is to rewrite the right-hand side in terms of the positive-depth FKS--Yu Green function $\sfQ_{\bbT_r}^{\bbG_{s,r}}(\theta_+)$ for $(\bfG_s,\bfT,\theta)$.
    Using Proposition \ref{prop:descent}, we obtain
    \begin{align*}
        \Theta_{\cc\tau^{\FKS}_{(\bfT,\theta)}}(g)
        &=
        \frac{(-1)^{r(\bfG^0)-r(\bfT)+r(\bfT,\theta)-r_s(\bfT,\theta)}}{|\cc{K}_{s,0+}/G_{s,\x,r+}|}
        \cdot\sum_{\begin{subarray}{c} x\in \bbG_{r}(\F_q) \\{}^{x}s\in \bbT_r(\F_q)\end{subarray}}
         \theta({}^{x}s)\cdot Q^{\bbG_{{}^x s,0}^0}_{\bbT_0}(({}^{x}u)_0)
        \cdot\dot{\Theta}_{\kappa_{{}^{x}s,+}}({}^{x}u)\\
        &=
        \frac{(-1)^{r(\bfG^0)-r(\bfT)+r(\bfT,\theta)-r_s(\bfT,\theta)}}{|\cc{K}_{s,0+}/G_{s,\x,r+}|}
        \\
        &\qquad\qquad\qquad
        \cdot\sum_{\begin{subarray}{c} x\in \bbG_{r}(\F_q)/\bbG_{s,r}(\F_q) \\{}^{x}s\in \bbT_r(\F_q)\end{subarray}}\theta({}^{x}s)
        \cdot\sum_{y\in\bbG_{{}^{x}s,r}(\F_q)}
        Q^{\bbG_{{}^x s,0}^0}_{\bbT_0}({}^{y}({}^{x}u)_0)
        \cdot\dot{\Theta}_{\kappa_{{}^{x}s,+}}({}^{y}({}^{x}u)).
    \end{align*}
    On the other hand, by applying the computation so far to $(\bfG_{{}^{x}s},\bfT,\theta)$ and the unipotent element ${}^{x}u\in\bbG_{{}^{x}s,r}(\F_q)$, we also obtain
    \begin{align*}
        \sfQ_{\bbT_r}^{\bbG_{{}^{x}s,r}}(\theta_+)({}^{x}u)
        &=
        \frac{(-1)^{r(\bfG_{{}^{x}s}^{0})-r(\bfT)}}{|\cc{K}_{{}^{x}s,0+}/G_{{}^{x}s,\x,r+}|}
        \cdot\sum_{y\in\bbG_{{}^{x}s,r}(\F_q)}
        Q^{\bbG_{{}^x s}^0,0}_{\bbT_0}({}^{y}({}^{x}u)_0)
        \cdot\dot{\Theta}_{\kappa_{{}^{x}s,+}}({}^{y}({}^{x}u))\\
        &=
        \frac{(-1)^{r(\bfG_{s}^{0})-r(\bfT)}}{|\cc{K}_{s,0+}/G_{s,\x,r+}|}
        \cdot\sum_{y\in\bbG_{{}^{x}s,r}(\F_q)}
        Q^{\bbG_{{}^x s, 0}^0}_{\bbT_0}({}^{y}({}^{x}u)_0)
        \cdot\dot{\Theta}_{\kappa_{{}^{x}s,+}}({}^{y}({}^{x}u)).
    \end{align*}
    Therefore, we get
    \[
        \Theta_{\cc\tau^{\FKS}_{(\bfT,\theta)}}(g)
        =
        (-1)^{r(\bfG,s,\bfT,\theta)}
        \cdot\sum_{\begin{subarray}{c} x\in \bbG_{r}(\F_q)/\bbG_{s,r}(\F_q) \\{}^{x}s\in \bbT_r(\F_q)\end{subarray}}\theta({}^{x}s)
        \cdot
        \sfQ_{\bbT_r}^{\bbG_{{}^{x}s,r}}(\theta_+)({}^{x}u).\qedhere
    \]
\end{proof}

\section{Exhaustion for Howe-unramified types}\label{sec:theta general}

We have now defined and analyzed Green functions for positive-depth Deligne--Lusztig induction (Section \ref{sec:DL}) and Green functions for Yu's construction (Section \ref{sec:Yu green}), we are ready to deploy an important application: to extend regular-$\theta$ results to all $\theta$.
The strategy we employ is inspired by Lusztig's work on Green functions \cite{Lus90}. 

Let $(\bfT,\theta)$ be an unramified  pair. 
We assume that $p \neq 2$ is not bad for $\bfG$ and that $(\bfT,\theta)$ is Howe-factorizable (see the comment in Section \ref{subsec:assumptions}). 
As we will be invoking Theorem \ref{thm:vreg characterization} by way of Theorem \ref{thm:reg comparison}, we also assume that $q$ is large enough so that \eqref{eq:Henniart} holds for $\bfT \hookrightarrow \bfM$ (notation as in Section \ref{subsec:types}).
Furthermore, we assume that $q$ is large enough so that there exists a regular (in $\bfG$) depth zero character of $T$.

\begin{rem}\label{rem:reg-depth-zero}
    For the existence of a regular depth zero character of $T$, it is enough to find a character of $\bbT_0$ which is not stabilized by any nontrivial element of $W_{\bbG_0}(\bbT_0)$.
    Any character of $\bbT_0(\F_q)$ determines an $\F_q$-rational semisimple element of the dual torus $\hat{\bbT}_0$ contained in the Langlands dual group $\hat{\bbG}_0$, and vice versa.
    This identification is equivariant with respect to the actions of the Weyl groups $W_{\bbG_0}(\bbT_0)\cong W_{\hat{\bbG}_0}(\hat{\bbT}_0)$.
    Hence, finding a regular character of $\bbT_0(\F_q)$ is equivalent to finding a regular semisimple element of $\hat{\bbT}_0$.
\end{rem}

\subsection{Positive-depth elliptic Deligne--Lusztig induction for arbitrary $\theta$}\label{subsec:Yu exhaust}

Assume that $\bfT \hookrightarrow \bfG$ is elliptic. The comparison between the geometrically constructed (virtual) representations $R_{\bbT_r}^{\bbG_r}(\theta)$ and the algebraically constructed (virtual) representations $\cc \tau_{(\bfT,\theta)}^{\FKS}$ easily follow by combining the results of Sections \ref{subsec:reg comparison}, \ref{sec:green}, and \ref{sec:Yu green}.

\begin{thm}\label{thm:Q comparison}
    For any unramified elliptic pair $(\bfT,\theta)$,
    \begin{equation*}
        \sfQ_{\bbT_r}^{\bbG_r}(\theta_+) = (-1)^{r(\bfG^0) - r(\bfT) + r(\bfT,\theta)} \cdot Q_{\bbT_r}^{\bbG_r}(\theta_+).
    \end{equation*}
\end{thm}

\begin{proof}
    By the assumption on $q$, we can find a regular depth zero character $\phi_{-1}' \from T \to \C^\times$. 
    Then $\theta' \coloneqq \theta \cdot \phi_{-1}^{-1} \cdot \phi_{-1}'$ defines a regular character of $T$ and $(\phi_{-1}', \phi_0, \ldots, \phi_d)$ is a Howe factorization for $T$. Then for $c = (-1)^{r(\bfG^0) - r(\bfT) + r(\bfT,\theta)}$,
    \begin{align*}
        \sfQ_{\bbT_r}^{\bbG_r}(\theta_+) 
        = \sfQ_{\bbT_r}^{\bbG_r}(\theta_+') 
        &= \Theta_{\cc \tau_{(\bfT, \theta')}^{\FKS}}|_{\bbG_r(\F_q)_{\unip}} \\
        &= c \cdot \Theta_{R_{\bbT_r}^{\bbG_r}(\theta')}|_{\bbG_r(\F_q)_{\unip}}
        = c \cdot Q_{\bbT_r}^{\bbG_r}(\theta_+') = c \cdot Q_{\bbT_r}^{\bbG_r}(\theta_+),
    \end{align*}
    where the equalities hold by Theorem \ref{thm:green Yu} (2), Definition \ref{def:green Yu}, Theorem \ref{thm:reg comparison} (this is the only place the assumption on $q$ is needed), Definition \ref{def:pos depth Green}, and Theorem \ref{thm:geom-char-formula}. 
\end{proof}

\begin{thm}\label{thm:gen comparison}
    For any unramified elliptic pair $(\bfT,\theta)$,
    \[
    \cc\tau^{\FKS}_{(\bfT,\theta)}
    \cong
    (-1)^{r(\bfG^0) - r(\bfT) + r(\bfT,\theta)}\cdot R_{\bbT_r}^{\bbG_r}(\theta).
    \]
\end{thm}

\begin{proof}
    To show the claim, it is enough to check that the both-hand sides have the same character at any element $g\in G_{\x,0}$.
    We fix $g\in G_{\x,0}$ and take its topological Jordan decomposition $g=su$.
    By Theorem \ref{thm:green Yu}, 
    \[
        \Theta_{\cc\tau^{\FKS}_{(\bfT,\theta)}}(g)
        =
        (-1)^{r(\bfG,s,\bfT,\theta)}
        \cdot\sum_{\begin{subarray}{c} x\in \bbG_{r}(\F_q)/\bbG_{s,r}(\F_q) \\{}^{x}s\in \bbT_r(\F_q)\end{subarray}}\theta({}^{x}s)
        \cdot
        \sfQ_{\bbT_r}^{\bbG_{{}^{x}s,r}}(\theta_+)({}^{x}u).
    \]
    (Recall that $r(\bfG,s,\bfT,\theta):=r(\bfG^0)-r(\bfG^0_s)+r(\bfT,\theta)-r_s(\bfT,\theta)$.)
    On the other hand, by Theorem \ref{thm:geom-char-formula} and Remark \ref{rem:geom-char-formula}, 
    \[
    \Theta_{R_{\bbT_r}^{\bbG_r}(\theta)}(g)
    = \sum_{\substack{x \in\bbG_r(\F_q)/\bbG_{r,s}(\F_q) \\ {}^{x}s \in \bbT_r(\F_q)}} \theta({}^{x}s) \cdot Q_{\bbT_r}^{\bbG_{{}^{x}s,r}}(\theta_{+})({}^{x}u).
    \]
    The desired identity $\Theta_{\cc\tau^{\FKS}_{(\bfT,\theta)}}(g)=\Theta_{(-1)^{r(\bfG^0) - r(\bfT) + r(\bfT,\theta)}\cdot R_{\bbT_r}^{\bbG_r}(\theta)}(g)$ then follows by applying Theorem \ref{thm:Q comparison} to $(\bfT,\theta)$, viewing $\bfT$ as an unramified elliptic maximal torus in $\bfG_s$. We note that applying Theorem \ref{thm:Q comparison} here requires \eqref{eq:Henniart} to hold for $\bfT \subseteq \bfG_s$, but this is implied by the assumption that \eqref{eq:Henniart} holds for $\bfT \subseteq \bfG$: the right-hand side of the inequality shrinks when passing from $\bfG$ to $\bfG_s$ and the denominator of the left-hand side shrinks when passing from $\bfG$ and $\bfG_s$ (since the roots of $\bfT$ in $\bfG_s$ are also roots in $\bfG$).
\end{proof}

\subsection{Exhaustion for Howe-unramified Yu types}

An immediate corollary of Theorem \ref{thm:gen comparison} and the exhaustiveness of Deligne--Lusztig induction \cite[Corollary 7.7]{DL76} is as follows. 
Let $\Psi=(\vec{\bfG},\vec{\phi},\vec{r},\x,\rho'_0)$ be any Howe-unramified Yu datum.
Then there exists an unramified elliptic maximal torus $\bfT$ of $\bfG^0$ with associated point $\x$ and a representation $\rho_0$ of $TG^0_{\x,0}$ satisying $\Ind_{TG^0_{\x,0}}^{G^0_{\bar{\x}}}\rho_0=\rho'_0$. On the geometric side, for any character $\theta$, recall (e.g.\ Section \ref{subsec:reg comparison}) that the virtual $G_{\x,0}$-representation $R_{\bbT_r}^{\bbG_r}(\theta)$ can be extended to a virtual representation of $ T G_{\x,0} = Z_G G_{\x,0}$ by demanding that $Z_G$ acts by $\theta|_{Z_G}$. 

\begin{cor}\label{cor:inner product depth zero}
    For any depth-zero character $\phi_{-1}\colon T\rightarrow\C^\times$, we define an unramified elliptic pair $(\bfT,\theta)$ by $\theta = \phi_{-1} \cdot \prod_{i=0}^d \phi_i|_{T}$.
    Then 
    \[
        \langle \cc{\tau}_\Psi^{\FKS}, R_{\bbT_r}^{\bbG_r}(\theta) \rangle_{T G_{\x,0}}
        =
        (-1)^{r(\bfT,\theta)}\cdot\langle \rho_0, R_{\bbT_0}^{\bbG_0^0}(\phi_{-1}) \rangle_{T G_{\x,0}^0}.
    \]
    In particular, every Howe-unramified supercuspidal type appears in $R_{\bbT_r}^{\bbG_r}(\theta)$ for some $(\bfT,\theta)$.
\end{cor}

\subsection{Exhaustion for Howe-unramified Kim--Yu types} \label{subsec:KY exhaust}

We now relax the ellipticity assumption on $\bfT$ imposed in Section \ref{subsec:Yu exhaust}. 
By Proposition \ref{prop:types induction}, we know that every Howe-unramified type appears as the positive-depth parabolic induction of a Howe-unramified supercuspidal type. By Corollary \ref{cor:inner product depth zero}, we know that any Howe-unramified supercuspidal type appears in $R_{\bbT_r}^{\bbM_r}(\theta)$ for some $(\bfT,\theta)$. This immediately implies:

\begin{thm}\label{eq:Hi exhaustion}
    Let $\Psi$ be a Howe-unramified Kim--Yu datum.
    For any depth-zero character $\phi_{-1}\colon T\rightarrow\C^\times$, we define an unramified elliptic pair $(\bfT,\theta)$ by $\theta = \phi_{-1} \cdot \prod_{i=0}^d \phi_i|_{T}$, where $\phi_i$'s are generic characters contained in $\Psi$.
    Then for each prime $\ell \neq p$, there exists an $i \geq 0$ such that the associated Kim--Yu type $\Ind_{K_0}^{G_{\x,0}}(\rho_\Psi^{\FKS})$ is a summand of $H_c^i(X_{\bbT_r \subset \bbG_r})_\theta$.
\end{thm}

We remark that a stronger form of exhaustion would be to insist a nonvanishing inner product with the alternating sum of the cohomology groups; i.e., $\langle \Ind_{K_0}^{G_{\x,0}}(\rho_\Psi^{\FKS}), R_{\bbT_r}^{\bbG_r}(\theta) \rangle \neq 0$. However, because $R_{\bbT_r}^{\bbG_r}(\theta)$ is a virtual representation, to prove this nonvanishing would involve additionally guaranteeing that $\Ind_{K_0}^{G_{\x,0}}(\rho_\Psi^{\FKS})$ will not get ``cancelled out'' when taking a signed formal linear combination of representations of the form $I_{P_{\x,0}}^{G_{\x,0}}(\rho')$ for the constituents $\rho'$ of the virtual representation $R_{\bbT_r}^{\bbM_r}(\theta)$. We do not prove this, but we believe this stronger assertion is true. In fact, we think that the better exhaustion result is the nonvanishing inner product with $r_{\bbT_r}^{\bbG_r}(\theta)$ (see \eqref{eq:modified DL induction}). This should essentially follow from Conjecture \ref{conj:Lusztig induction}.

\section{Springer hypothesis for Howe-unramified toral supercuspidal representations}\label{sec:character}\label{sec:springer}

In \cite{DR09}, DeBacker--Reeder established a character formula for regular depth-zero supercuspidal representations (see, especially, \cite[Lemmas 10.0.4 and 12.4.3]{DR09}).
The appearance of orbital integrals in their character formula is explained geometrically by the \textit{Springer hypothesis}: that classical Green functions \`a la Deligne--Lusztig can be re-interpreted as the Fourier transform of delta functions supported on co-adjoint orbits of regular semisimple elements of the dual Lie algebra. For regular supercuspidal representations of positive-depth, character formulas \textit{also} can be written in terms of orbital integrals (e.g.\ see \cite[Theorem 4.28]{DS18}, \cite[Theorem 4.3.5]{FKS23}). In this section, we demonstrate that the comparison theorem (Theorem \ref{thm:reg comparison}) can be utilized to give a \textit{geometric} explanation for the appearance of orbital integrals in such character formulae. For this, we use two main tools: 
\begin{enumerate}
    \item the construction of positive-depth character sheaves and its relationship to positive-depth Deligne--Lusztig induction \cite{BC24}, and 
    \item the positive-depth analogue of DeBacker--Reeder's $p$-adic harmonic-analytic methods passing from parahoric statements to its compact induction.
\end{enumerate}
We carry out (1) in Section \ref{subsec:pos-Springer} and (2) in Section \ref{subsec:DR}, both under the assumption that the regular supercuspidal representation is Howe-unramified and $0$-toral (Definition \ref{defn:toral}). We note that our perspective in this section is to establish proof-of-concept: understanding the relationship between the inherently geometric $R_{\bbT_r}^{\bbG_r}(\theta)$ and the inherently algebraic construction of Adler \cite{Adl98} and Yu \cite{Yu01} yields newfound understanding of character formulae.

\begin{remark}
    Establishing the positive-depth Springer hypothesis reveals yet another method to establish the relationship between $R_{\bbT_r}^{\bbG_r}(\theta)$ and Yu's construction. We outline this here. 
    Following the computations \cite{DS18} for the representation of $G_{\x,0}$ associated to a  Howe-unramified 0-toral pair $(\bfT,\theta)$, we can obtain a character formula identical to the form of Corollary \ref{cor:formula with orbits}. One can then conclude the comparison theorem \cite[Theorem 7.8]{CO25} for $0$-toral characters, weakening the assumption on $q$ to the condition that $\bbT_0(\F_q) \cap \bbG_0(\F_q)_{\reg}$ is nonempty. (Note, however, that an assumption on $p$ required in \cite{DR09}, see Section \ref{subsec:log}.)
\end{remark}

\subsection{Positive-depth Springer hypothesis}\label{subsec:pos-Springer}

The main result of this section is a re-interpretation of the character of the representation $R_{\bbT_r}^{\bbG_r}(\theta)$ in terms of the Fourier transform of an orbital integral (Theorem \ref{thm:pos Springer}). 
The $r=0$ statement is known as the \textit{Springer hypothesis}; it was proved by Kazhdan first in \cite{Kaz77}. 
Below, we adapt Kazhdan--Varshavsky's sheaf-theoretic proof of the Springer hypothesis \cite[Appendix A]{KV06} to the positive-depth setting using:
\begin{enumerate}
    \item Bezrukavnikov's work with the first author on character sheaves on parahoric subgroups \cite{BC24},
    \item positive-depth adaptations of DeBacker--Reeder's analysis of logarithm maps on $p$-adic groups \cite[Appendix B]{DR09}.
\end{enumerate}
In the $r=0$ setting, Kazhdan--Varshavsky construct a family of perverse sheaves interpolating between the two sides (a Deligne--Lusztig induction and the Fourier transform of an orbital integral). We offer an alternative proof here which works equally well for $r=0$ and $r > 0$.

\subsubsection{Logarithm and exponential maps}\label{subsec:log}

Let $n$ be the dimension of the smallest faithful $F$-rational representation of $\bfG$. 
In the following, we assume that $p \geq (2+e)n$, where $e$ is the ramification index of $F/\Q_p$.
Then, by \cite[Lemma B.0.9, Appendix B]{DR09}, there exists a Frobenius-equivariant and $G$-equivariant bijective map
\[
\log\colon \bfG(F^{\ur})_{0+}\rightarrow\bmfg(F^{\ur})_{0+},
\]
where $\bfG(F^{\ur})_{0+}$ (resp.\ $\bmfg(F^{\ur})_{0+}$) denotes the set of topologically unipotent elements of $\bigcup_{\y\in\cB(\bfG,F)}\bfG(F^{\ur})_{\y,0}$ (resp.\ topologically nilpotent elements of $\bigcup_{\y\in\cB(\bfG,F)}\bmfg(F^{\ur})_{\y,0}$).
Moreover, the restriction of $\log$ on $\bfG(F^{\ur})_{0+}\cap \bfG(F^{\ur})_{\x,0}$ induces a $\sigma$-equivariant and $\bbG_{0}$-equivariant isomorphism of $\F_{q}$-varieties
\[
\log\colon (\bbG_{0})_{\unip}\rightarrow(\mfg_{0})_{\nilp},
\]
where $\mfg_{0}$ denotes the Lie algebra of $\bbG_{0}$ over $\F_{q}$, as proved in \cite[Corollary B.6.19]{DR09}.
We can check that the proof of \cite[Corollary B.6.19]{DR09} then holds for $r>0$ with no complications:

\begin{lem}\label{lem:log unipotent}
For any $r\in\R_{>0}$, the above bijective map $\log$ also induces a $\sigma$-equivariant and $\bbG_{r}$-equivariant isomorphism of $\F_{q}$-varieties
\[
\log\colon (\bbG_{r})_{\unip}\rightarrow(\mfg_{r})_{\nilp},
\]
where $\mfg_{r}$ denotes the Lie algebra of $\bbG_{r}$ over $\F_{q}$.
\end{lem}

We let $\exp \from \bmfg(F^{\ur})_{0+} \to \bfG(F^{\ur})_{0+}$ denote the inverse of $\log$.

\subsubsection{Filtration on the dual Lie algebra}

Let $\bfT$ be an unramified elliptic maximal torus of $\bfG$ and $\x\in\cB(\bfG,F)$ a point associated to $\bfT$.
Let $\bmfg^{\ast}$ be the dual to the Lie algebra $\bmfg$ over $F$.
We define a filtration $\{\mfg^{\ast}_{\x,r}\}_{r\in\widetilde{\R}}$ on $\mfg^{\ast}=\bmfg^{\ast}(F)$ by
\[
\mfg^{\ast}_{\x,r}
:=
\{Y^{\ast}\in \mfg^{\ast} \mid \langle Y^{\ast},\mfg(F)_{\x,(-r)+}\rangle \subset \mfp_{F}\}.
\]

Let $\bbW_{r}$ denote the $r$-th truncated Witt ring scheme associated to $F$ when $F$ has characteristic $0$ and the $r$-th jet scheme of $\Ga$ when $F$ has characteristic $p>0$.
Note that, in both cases, we have $\bbW_{r}(\F_{q})=\cO_{F}/\mfp_{F}^{r+}$.
We fix a nontrivial additive character $\psi\colon \cO_{F}/\mfp_{F}^{r+}\rightarrow\C^{\times}$ and consider the associated $\ell$-adic multiplicative sheaf $\mathcal{L}_{\psi}$ on $\bbW_{r}$.
(We also fix an isomorphism $\C\cong\ol{\Q}_{\ell}$, but $\cL_{\psi}$ is independent of this choice.)

Note that the Lie algebra $\mfg_{r}$ of $\bbG_{r}$ is a free $\bbW_{r}$-module.\footnote{We caution that our usage of the symbol $\mfg_{r}$ is different from the standard one, that is, $\mfg_{r}=\bigcup_{\y}\mfg_{\y,r}$.}
We have $\mfg_{r}(\F_{q})\cong\mfg_{\x,0}/\mfg_{\x,r+}$ (this follows from the exactness of the Lie algebra functor in this setting; cf.\ \cite[2.4.3 (a) and 13.3.2]{KP23}).
Let $\mfg^{\ast}_{r}$ be the dual of $\mfg_{r}$, i.e., an affine group scheme over $\F_{q}$ representing the functor $R\mapsto \Hom_{\bbW_{r}(R)}(\mfg_{r}(R), \bbW_{r}(R))$ for $\F_q$-algebras $R$.
Note that hence
\[
\mfg^{\ast}_{r}(\F_{q})
=\Hom_{\cO_{F}/\mfp_{F}^{r+}}(\mfg_{\x,0}/\mfg_{\x,r+}, \cO_{F}/\mfp_{F}^{r+}).
\]
As
\[
\mfg^{\ast}_{\x,0}
=\Hom_{\cO_{F}}(\mfg^{\ast}_{\x,0+},\mfp_{F})
\cong \Hom_{\cO_{F}}(\mfg^{\ast}_{\x,0},\cO_{F})
\]
(note that jumps occur only at integers, so $\mfg^{\ast}_{\x,0+}=\mfg^{\ast}_{\x,1}=\mfp_{F}\cdot \mfg^{\ast}_{\x,0}$), we have a natural reduction map from $\mfg^{\ast}_{\x,0}$ to $\mfg^{\ast}_{r}(\F_{q})$.
Since $\mfg_{\x,0}$ is a free $\cO_{F}$-module, this map is surjective.
Moreover, it can be easily seen that the kernel is $\mfg^\ast_{\x,r+}$.
In other words, we have a natural identification $\mfg^{\ast}_{r}(\F_{q})\cong \mfg^{\ast}_{\x,0}/\mfg^{\ast}_{\x,r+}$.

\subsubsection{Geometric Fourier transform}\label{subsec:geom FT}

We define the Fourier transform between the derived categories $D_{\bbG_r}(\mfg^{\ast}_r)$ and $D_{\bbG_r}(\mfg_r)$ of $\bbG_{r}$-equivariant constructible $\ell$-adic sheaves on $\mfg^{\ast}_{r}$ and $\mfg_{r}$ by
\[
\cFT\colon D_{\bbG_{r}}(\mfg^{\ast}_{r})\rightarrow D_{\bbG_{r}}(\mfg_{r});\quad
\mathcal{F}\mapsto \pr_{2,!}((\pr_{1}^{\ast}\mathcal{F})\otimes(\langle-,-\rangle^{\ast}\mathcal{L}_{\psi}))
\]
\[
\begin{tikzcd}
& \mfg^{\ast}_{r}\times\mfg_{r} \ar{rr}{\langle-,-\rangle} \arrow[dl, "\pr_1" swap] \arrow[dr, "\pr_2"] &&\bbW_{r} \\
\mfg^{\ast}_{r} && \mfg_{r}
\end{tikzcd}
\]
(cf., e.g., \cite[Section III.13]{KW01}).
When $\mathcal{F}\in D_{\bbG_{r}}(\mfg^{\ast}_{r})$ has a Weil structure, i.e., $\mathcal{F}$ is equipped with an isomorphism $\sigma^{\ast}\mathcal{F}\cong\mathcal{F}$, its Fourier transform $\cFT(\mathcal{F})$ also has a natural Weil structure.
Note that, if we let $f_{\cF}$ and $f_{\cFT(\cF)}$ denote the functions on $\mfg^{\ast}_{r}(\F_{q})$ and $\mfg_{r}(\F_{q})$ associated to the Weil sheaves $\cF$ and $\cFT(\cF)$ under the function-sheaf dictionary, then $f_{\cFT(\cF)}$ exactly realizes the Fourier transform $\FT \from C(\mfg_{r}^{\ast}(\F_{q})) \to C(\mfg_{r}(\F_{q}))$ of $f_{\cF}$ in the classical sense; that is,
\[
f_{\cFT(\cF)}(Y)
=\sum_{Y^{\ast}\in\mfg^{\ast}_{r}(\F_{q})} f_{\cF}(Y^{\ast})\cdot\psi(\langle Y^{\ast},Y\rangle) \eqqcolon \FT(f_{\cF})(Y).
\]

For any closed subset $V$ of $\mfg_{r}$ or $\mfg^{\ast}_{r}$, we let $\delta_{V}$ denote the extension-by-zero of the constant sheaf $\Qlb$ on $V$.
If $V$ is $\F_{q}$-rational, then $\delta_{V}$ has an obvious Weil structure.
Note that then the characteristic function $\mathbbm{1}_{V(\F_{q})}$ of $V(\F_{q})$ is associated to $\delta_{V}$ under the function-sheaf dictionary.
We later use this notation for the $\F_{q}$-rational closed subsets $\{X^{\ast}_{r}\}$ and $\bbG_{r}\cdot X^{\ast}_{r}$ of $\mfg^{\ast}_{r}$.

\subsubsection{Positive-depth Springer hypothesis}

We prove the positive-depth Springer hypothesis for $0$-toral characters $\theta \from T \to \C^\times$; let $r$ be the depth of $\theta$. This gives an interpretation of the associated positive-depth Green function $Q_{\bbT_r}^{\bbG_r}(\theta_+)$ (Definition \ref{def:pos depth Green}) as the Fourier transform of certain coadjoint orbits.

Fix a level-$1$ additive character $\psi_F$ of $F$ (i.e., $\psi_{F}$ is trivial on $\mfp_{F}$ but not on $\cO_{F}$). We can then choose an element $X^{\ast}\in\mft^{\ast}_{-r}\subset\mfg^{\ast}_{\x,-r}$ satisfying
\[
\theta(\exp(Y))=\psi_{F}(\langle X^{\ast},Y\rangle)
\]
for any $Y\in\mft_{0+}$.
(Here, we are identifying $\mft^{\ast}$ with the eigenspace of $\mfg^{\ast}$ with eigenvalue $1$ with respect to the co-adjoint action of $\bfT$ on $\mfg^{\ast}$; then $\mft^{\ast}_{-r}$ is contained in $\mfg^{\ast}_{\x,-r}$).
We remark that such an element $X^{\ast}$ always exists and is well-defined modulo $\mft^{\ast}_{0}$ since $\theta\circ\exp|_{\mft_{0+}}$ is an additive character of $\mft_{0+}$ trivial on $\mft_{r+}$.
Also, by the $0$-toral regularity assumption on $(\bfT,\theta)$, the image of $X^{\ast}$ under $\mft^{\ast}_{-r}/\mft^{\ast}_{0}\twoheadrightarrow\mft^{\ast}_{-r}/\mft^{\ast}_{(-r)+}$ is a $\bfG$-generic element of depth $r$.
But then it implies that any element of $X^{\ast}+\mft^{\ast}_{0}$ is $\bfG$-generic of depth $r$.
In particular, $X^{\ast}$ itself is $\bfG$-generic of depth $r$, which implies that $X^{\ast}$ is necessarily elliptic regular semisimple (in the sense that the connected centralizer of $X_{r}^{\ast}$ in $\bfG$ is $\bfT$, which is an elliptic maximal torus).

By fixing a uniformizer $\varpi$ of $F$, we define a character $\psi_{r}\colon F\rightarrow\C^{\times}$ by 
\[
\psi_{r}(x):=\psi_{F}(\varpi^{-r}x).
\]
Then $\psi_{r}$ is an additive character of level $r+1$, hence induces an additive character of $\cO_{F}/\mfp_{F}^{r+1}$ which is nontrivial on $\mfp_{F}^{r}/\mfp_{F}^{r+1}$; assume that this induced character is the additive character $\psi$ used in the definition of the geometric Fourier transform in Section \ref{subsec:geom FT}. 
Note that then 
\[
\theta(\exp(Y))
=\psi_{F}(\langle X^{\ast},Y\rangle)
=\psi(\langle \varpi^{r}X^{\ast},Y\rangle).
\]
We put $X_{r}^{\ast}:=\varpi^{r}X^{\ast}\in\mfg^{\ast}_{\x,0}$, which is an elliptic regular semisimple element.
As we have $\mfg^{\ast}_{r}(\F_{q})\cong \mfg^{\ast}_{\x,0}/\mfg^{\ast}_{\x,r+}$, it makes sense to regard $X_{r}^{\ast}$ as an element of $\mfg^{\ast}_{r}(\F_{q})$.
Note that then the coadjoint orbit $\bbG_{r}\cdot X_{r}^{\ast}$ of $X_{r}^{\ast}$ is closed in $\mfg^{\ast}_{r}$ (this follows from that all the root values of $X_{r}^{\ast}$ are distinct even modulo $\mfp_{F}^{r+}$, which is guaranteed by the $\bfG$-genericity of $X^{\ast}$).

Now we state and prove the positive-depth Springer hypothesis.

\begin{thm}[positive-depth Springer hypothesis]\label{thm:pos Springer}
Let $(\bfT,\theta)$ be an unramified elliptic $0$-toral regular pair of depth $r$ and assume $q$ is large enough so that $\bbT_0$ has a regular semisimple element. 
For any topologically unipotent $u \in G_{\x,0}$,
    \begin{equation*}
       q^{\frac{1}{2}\dim(\bbG_r/\bbT_r)} \cdot Q_{\bbT_r}^{\bbG_r}(\theta_{+})(u) = \FT(\mathbbm{1}_{\bbG_{r}(\F_{q})\cdot X_{r}^{\ast}})(\log(u)).
    \end{equation*}
\end{thm}

\begin{proof}
    By \cite{BC24}, there is a parabolic induction functor $\pInd_{\bbT_r}^{\bbG_r} \from D_{\bbT_r}(\bbT_r) \to D_{\bbG_r}(\bbG_r)$ which induces a $t$-exact equivalence of categories between so-called $(\bfT,\bfG)$-generic subcategories $D_{\bbT_r}^\psi(\bbT_r)$ and $D_{\bbG_r}^\psi(\bbG_r)$ (see \cite[Definition 4.5]{BC24}).
    This functor is defined by (a shift of)
    \begin{equation*}
        \pInd_{\bbT_r}^{\bbG_r} \coloneqq \pi_! f^*.
    \end{equation*}
    where $\widetilde \bbG_r \coloneqq \{(g,h\bbB_r) \in \bbG_r \times \bbG_r/\bbB_r \mid h^{-1} g h \in \bbB_r\}$ and
    \begin{equation*}
        \begin{tikzcd}
            & \widetilde \bbG_r \arrow[dl, "f" swap] \arrow[dr, "\pi"] \\
            \bbT_r && \bbG_r
        \end{tikzcd}
        \qquad
        \begin{aligned}
            f(g, h\bbB_r) &= \pr_{\bbT_r}(h^{-1}gh), \\
            \pi(g,h\bbB_r) &= g
        \end{aligned}
    \end{equation*} 
    ($\pr_{\bbT_{r}}$ denotes the natural projection $\bbB_{r}\twoheadrightarrow\bbT_{r}$).
    
    Let $\cL_\theta$ be the multiplicative local system on $\bbT_r$ associated to $\theta$ in the following sense: $\cL_{\theta}$ is a Weil sheaf whose associated function $f_{\cL_{\theta}}$ is equal to the character of $\bbT_{r}(\F_{q})\cong T_{0:r+}$ induced by $\theta$. 
    By our choice of $X_{r}^{\ast}$, then we have the equality of functions 
    \[
    f_{\cL_{\theta}}|_{(\bbT_{r})(\F_{q})_\unip}=\log^{\ast} \FT_{\bbT_{r}}(\mathbbm{1}_{X_{r}^{\ast}})|_{\bbT_{r}(\F_{q})_{\unip}}.
    \]
    Since both $\cL_{\theta}$ and $\log^{\ast}\cFT(\delta_{X_{r}^{\ast}})$ are multiplicative local systems on $(\bbT_{r})_{\unip}$, this implies that 
    \[
    \cL_{\theta}|_{(\bbT_r)_\unip} = \log^{\ast} \cFT(\delta_{X_{r}^{\ast}})|_{(\bbT_r)_{\unip}}.
    \]
    A direct application of the base-change theorem shows that we then have
    \begin{equation*}
        \pInd_{\bbT_r}^{\bbG_r}(\cL_\theta)|_{(\bbG_r)_\unip} \cong \pInd_{\bbT_r}^{\bbG_r}(\log^{\ast}\cFT(\delta_{X_{r}^{\ast}}))|_{(\bbG_r)_\unip}.
    \end{equation*}

    Analogously to the above, we may define parabolic induction from the Lie algebra $\mft_r$ to $\mfg_r$ as a push-pull along $f(X,h\bbB_r) = \pr_{\mft_r}(\Ad(h)(X))$ and $\pi(X,h\bbB_r)= X$ (similarly, from $\mft^{\ast}_r$ to $\mfg^{\ast}_r$).
    By Lemma \ref{lem:log unipotent}, the logarithm $\log \from (\bbG_r)_{\unip} \to (\mfg_r)_{\unip}$ is bijective and $\bbG_r$-equivariant, and therefore lifts to a map $\log \from \pi^{-1}((\bbG_r)_{\unip}) \to \pi^{-1}((\mfg_r)_{\nilp})$ such that $f \circ \log = \log \circ f$ and $\pi \circ \log = \log \circ \pi$. By base change again, we obtain
    \begin{equation}
        \pInd_{\bbT_r}^{\bbG_r}(\log^{\ast} \cFT(\delta_{X_{r}^{\ast}}))|_{(\bbG_r)_{\unip}}
        \cong \log^{\ast}(\pInd_{\mft_r}^{\mfg_r}(\cFT(\delta_{X_{r}^{\ast}}))|_{(\mfg_r)_{\nilp}}).
    \end{equation}
    It now follows from the same argument as in \cite[Theorem 1.4 (i)]{Mir04} that
    \begin{equation*}
        \pInd_{\mft_r}^{\mfg_r}(\cFT(\delta_{X_{r}^{\ast}})) = \cFT(\pInd_{\mft^{\ast}_r}^{\mfg^{\ast}_r}(\delta_{X_{r}^{\ast}})).
    \end{equation*}
    We now compute $\pInd_{\mft^{\ast}_r}^{\mfg^{\ast}_r}(\delta_{X_{r}^{\ast}})=\pi_{!}f^{\ast}(\delta_{X^{\ast}_{r}})$.
    By the base change theorem, the pull-back $f^{\ast}(\delta_{X^{\ast}_{r}})$ is nothing but the zero extension to $\tilde{\mfg}_{r}$ of the constant $\ell$-adic sheaf on $\{(Y^{\ast},h\bbB_r) \mid \Ad(h)^{-1}(Y^{\ast}) \in X_{r}^{\ast} + \mathfrak{u}^{\ast}_r\}$.
    The image of $\{(Y^{\ast},h\bbB_r) \mid \Ad(h)^{-1}(Y^{\ast}) \in X_{r}^{\ast} + \mathfrak{u}^{\ast}_r\}$ under the map $\pi$ is the closed subset $\bbG_{r}\cdot X_{r}^{\ast}$ of $\mfg_{r}$; each fiber of $\{(Y^{\ast},h\bbB_r) \mid \Ad(h)(Y^{\ast}) \in X^{\ast}_{r} + \mathfrak{u}^{\ast}_r\} \to \bbG_{r}\cdot X_{r}^{\ast}$ is isomorphic to $\bbU_r$.
    Since $\bbU_r$ is isomorphic to an affine space of dimension $d$ over $\F_{q}$, we conclude that
    \begin{equation*}
        \pInd_{\mft^{\ast}_r}^{\mfg^{\ast}_r}(\delta_{X_{r}^{\ast}}) \cong \delta_{\bbG_{r}\cdot X_{r}^{\ast}}[2d](d),
    \end{equation*}
    where $d=\dim\bbU_{r}$.
    We have now shown
    \begin{equation*}
        \pInd_{\bbT_r}^{\bbG_r}(\cL_\theta)|_{(\bbG_r)_{\unip}} \cong \log^{\ast}(\cFT(\delta_{\bbG_{r}\cdot X_{r}^{\ast}})|_{(\mathfrak g_r)_{\nilp}})[2d](d).
    \end{equation*}

    By \cite[Theorem 10.9]{BC24} and Theorem \ref{thm:geom-char-formula} (2), the function associated to $\pInd_{\bbT_r}^{\bbG_r}(\cL_\theta)$ is given by the character of $R_{\bbT_r}^{\bbG_r}(\theta)$. In particular, we see that the function associated to $\pInd_{\bbT_r}^{\bbG_r}(\cL_\theta)|_{\bbG_r(\F_q)_\unip}$ is exactly $R_{\bbT_r}^{\bbG_r}(\theta)|_{\bbG_r(\F_q)_\unip}$, which is the Green function $Q_{\bbT_r}^{\bbG_r}(\theta_+)$ by Theorem \ref{thm:geom-char-formula} (2).
     (Note that there is a constant $(-1)^{\dim \bbG_r}$ in \textit{op.\ cit.} This is due to fact that there, we consider the function associated to the perverse sheaf $\pi_! f^* \cL_\theta[\dim \bbG_r]$.) On the other hand, the function associated to $\log^{\ast}(\cFT(\delta_{\bbG_{r}\cdot X_{r}^{\ast}})|_{(\mathfrak g_r)_{\nilp}})[2d](d)$ is 
    \[
    (-1)^{2d}\cdot q^{d}\cdot
    \log^{\ast}(\FT(\mathbbm{1}_{\bbG_{r}\cdot X_{r}^{\ast}})|_{(\mfg_{r}(\F_{q}))_{\nilp}}).
    \]
    Noting that $\frac{1}{2}\dim(\bbG_r/\bbT_r)=\dim\bbU_{r}=d$, we get the assertion.
\end{proof}

Combining the positive-depth Springer hypothesis (Theorem \ref{thm:pos Springer}) with the character formula of $R_{\bbT_r}^{\bbG_r}(\theta)$ in terms of Green functions (Theorem \ref{thm:geom-char-formula}) we obtain:

\begin{cor}\label{cor:formula with orbits}
    Let $(\bfT,\theta)$ be an unramified elliptic $0$-toral regular pair of depth $r$ and assume $q$ is large enough so that $\bbT_0(\F_q)$ has a regular semisimple element.
    For $\gamma \in \bbG_r(\F_q)$ with  Jordan decomposition $\gamma = su$,
    \begin{align*}
        \Theta_{R_{\bbT_r}^{\bbG_r}(\theta)}(\gamma) 
        &= \frac{1}{|\bbG_{r,s}(\F_q)|} \sum_{\substack{x \in\bbG_r(\F_q) \\ {}^{x}s \in \bbT_r(\F_q)}} \theta^{x}(s) \cdot q^{-\frac{1}{2} \dim(\bbG_{r,s}/\bbT_r)} \cdot \FT(\mathbbm{1}_{\bbG_{r,s}(\F_q) \cdot X_r^*})(\log(u)).
    \end{align*}
\end{cor}

\subsection{DeBacker--Reeder's character formula}\label{subsec:DR}

The aim of the remaining of this section is to establish a character formula of the regular supercuspidal representation $\pi_{(\bfT,\theta)}^{\FKS}$ as a consequence of our comparison result (Theorem \ref{thm:reg comparison}) and the positive-depth Springer hypothesis (Theorem \ref{thm:pos Springer}).
For this, we reproduce the discussion of \cite[Sections 10 and 12]{DR09} in the positive-depth setting.

Let $(\bfT,\theta)$ be a Howe-unramified elliptic $0$-toral regular pair of depth $r\in\R_{>0}$.
By Theorem \ref{thm:reg comparison}, 
\[
\pi_{(\bfT,\theta)}^{\FKS}
\cong
\cInd_{TG_{\x,0}}^{G}\bigl((-1)^{r(\bfT, \theta)} R_{\bbT_r}^{\bbG_r}(\theta)\bigr).
\]
In particular, $\cInd_{TG_{\x,0}}^{G}((-1)^{r(\bfT, \theta)} R_{\bbT_r}^{\bbG_r}(\theta))$ is an irreducible supercuspidal representation.

We let $\dot{\Theta}_{R_{\bbT_r}^{\bbG_r}(\theta)}$ denote the zero extension of the character of the representation $R_{\bbT_r}^{\bbG_r}(\theta)$ of $TG_{\x,0}$ to $G$.
We fix Haar measures $dg$ and $dz$ of $G$ and $Z_{\bfG}$, respectively.
We also write $dg$ for their quotient measure on $G/Z_{\bfG}$ by abuse of notation.

\begin{lem}\label{lem:HC-int}
For any regular semisimple element $\gamma\in G$, 
\[
\Theta_{\pi_{(\bfT,\theta)}^{\FKS}}(\gamma)
=
(-1)^{r(\bfT,\theta)}\cdot \frac{dz(Z_{\bfG,0})}{dg(G_{\x,0})}
\cdot\int_{G/Z_{\bfG}}\int_{K}\dot{\Theta}_{R_{\bbT_r}^{\bbG_r}(\theta)}({}^{gk}\gamma)\,dk\,dg.
\]
Here, $\deg(\pi_{(\bfT,\theta)}^{\FKS})$ is the formal degree of $\pi_{(\bfT,\theta)}^{\FKS}$ with respect to $dg$ and $K$ is any open compact subgroup of $G$ with Haar measure $dk$ normalized so that $dk(K)=1$.
\end{lem}

\begin{proof}
By Harish-Chandra's integration formula (\cite[94 page]{HC70}), 
\[
\Theta_{\pi_{(\bfT,\theta)}^{\FKS}}(\gamma)
=
\frac{\deg(\pi_{(\bfT,\theta)}^{\FKS})}{\dim(R_{\bbT_r}^{\bbG_r}(\theta))}\cdot\int_{G/Z_{\bfG}}\int_{K}\dot{\Theta}_{R_{\bbT_r}^{\bbG_r}(\theta)}({}^{gk}\gamma)\,dk\,dg.
\]
By, e.g., \cite[Theorem A.14]{BH96}, we have $\deg(\pi_{(\bfT,\theta)}^{\FKS})=(-1)^{r(\bfT,\theta)} \dim(R_{\bbT_r}^{\bbG_r}(\theta))\cdot dg(TG_{\x,0}/Z_{\bfG})^{-1}$.
Since $TG_{\x,0}/Z_{\bfG}=Z_{\bfG}G_{\x,0}/Z_{\bfG}\cong G_{\x,0}/Z_{\bfG,0}$, we get the desired identity.
\end{proof}

In the following, we write
\[
\cR_{\bfT}^{\bfG}(\theta)(\gamma)
:=
\frac{dz(Z_{\bfG,0})}{dg(G_{\x,0})}
\cdot\int_{G/Z_{\bfG}}\int_{K}\dot{\Theta}_{R_{\bbT_r}^{\bbG_r}(\theta)}({}^{gk}\gamma)\,dk\,dg,
\]
so that $\Theta_{\pi_{(\bfT,\theta)}^{\FKS}}(\gamma)=(-1)^{r(\bfT,\theta)}\cdot\cR_{\bfT}^{\bfG}(\theta)(\gamma)$.

The following is the analogue of \cite[Lemma 10.0.7]{DR09}.
\begin{lem}\label{lem:localization}
Let $\gamma\in G_{\x,0}$ be a regular semisimple element with topological Jordan decomposition $\gamma=\gamma_{0}\gamma_{+}$.
We put $K_{\gamma_{0}}:=K\cap G_{\gamma_{0}}$ and choose the Haar measure $dl$ on $K_{\gamma_{0}}$ so that $dl(K_{\gamma_{0}})=1$.
Then we have
\[
\int_{G/Z_{\bfG}}\int_{K}\dot{\Theta}_{R_{\bbT_r}^{\bbG_r}(\theta)}({}^{gk}\gamma)\,dk\,dg
=
\int_{G/Z_{\bfG}}\int_{K_{\gamma_{0}}}\dot{\Theta}_{R_{\bbT_r}^{\bbG_r}(\theta)}({}^{gl}\gamma)\,dl\,dg.
\]
\end{lem}

\begin{proof}
The only ingredient of the proof of \cite[Lemma 10.0.7]{DR09} to be modified in the current setting is \cite[Lemma 10.0.6]{DR09}, which asserts that the function 
\[
G/Z_{\bfG}\rightarrow\C\colon g \mapsto
\int_{K_{\gamma_{0}}}\dot{\Theta}_{R_{\bbT_r}^{\bbG_r}(\theta)}({}^{gl}\gamma)\,dl
\]
is compactly supported.
By Theorem \ref{thm:reg comparison}, $(-1)^{r(\bfT,\theta)}\cdot R_{\bbT_{r}}^{\bbG_{r}}(\theta)\cong \cc\tau_{(\bfT,\theta)}^{\FKS}$.
By Yu's construction (modified by Fintzen--Kaletha--Spice), the representation $\cc\tau_{(\bfT,\theta)}^{\FKS}$ is given to be the induction of a representation $\sigma_{(\bfT,\theta)}^{\FKS}$ of $TG_{\x,0+}$ to $TG_{\x,0}$ (see Section \ref{subsec:ADS}).
Hence, by the Frobenius character formula, we have
\[
\Theta_{R_{\bbT_{r}}^{\bbG_{r}}(\theta)}(\gamma')
=
(-1)^{r(\bfT,\theta)}\cdot
\sum_{x\in TG_{\x,0}/TG_{\x,0+}} \dot{\Theta}_{\sigma_{(\bfT,\theta)}^{\FKS}}({}^{x}\gamma')
\]
for any $\gamma'\in TG_{\x,0}$, where $\dot{\Theta}_{\sigma_{(\bfT,\theta)}^{\FKS}}$ denotes the zero extension of $\Theta_{\sigma_{(\bfT,\theta)}^{\FKS}}$ from $TG_{\x,0+}$ to $TG_{\x,0}$.
By \cite[Lemma 6.3]{AS09} (note that the $0$-torality assumption is needed for this), the function 
\[
G/Z_{\bfG}\rightarrow\C\colon g \mapsto
\int_{K_{\gamma_{0}}}\dot{\Theta}_{\sigma_{(\bfT,\theta)}^{\FKS}}({}^{gl}\gamma)\,dl
\]
is compactly supported, where $\dot{\Theta}_{\sigma_{(\bfT,\theta)}^{\FKS}}$ denotes the zero extension of the character of $\sigma_{(\bfT,\theta)}^{\FKS}$ to $G$.
Hence the function 
\[
G/Z_{\bfG}\rightarrow\C\colon g \mapsto
\sum_{x\in TG_{\x,0}/TG_{\x,0+}}\int_{K_{\gamma_{0}}}\dot{\Theta}_{\sigma_{(\bfT,\theta)}^{\FKS}}({}^{xgl}\gamma)\,dl
\]
is also compactly supported since each summand is just a translation of the compactly supported function $g \mapsto
\int_{K_{\gamma_{0}}}\dot{\Theta}_{\sigma_{(\bfT,\theta)}^{\FKS}}({}^{gl}\gamma)\,dl$ and the sum is finite.
This completes the proof.
\end{proof}

Now we establish the following, which is an analogue of \cite[Lemma 10.0.4]{DR09}:

\begin{prop}\label{prop:CF-1st}
For any regular semisimple element $\gamma\in G$, we have 
\[
\Theta_{\pi_{(\bfT,\theta)}^{\FKS}}(\gamma)
=
(-1)^{r(\bfT,\theta)}\sum_{\begin{subarray}{c} x\in S\backslash G/G_{\gamma_{0}} \\ {}^{x}\gamma_{0}\in T \end{subarray}}
\theta({}^{x}\gamma_{0})\cdot\cR_{\bfT^{x}}^{\bfG_{\gamma_{0}}}(\theta^x)(\gamma_{+}).
\]
\end{prop}

\begin{proof}
Recall that $\Theta_{\pi_{(\bfT,\theta)}^{\FKS}}(\gamma)=(-1)^{r(\bfT,\theta)}\cdot\cR_{\bfT}^{\bfG}(\theta)(\gamma)$.
By Lemma \ref{lem:localization}, we have
\begin{align*}
\cR_{\bfT}^{\bfG}(\theta)(\gamma)
&=
\frac{dz(Z_{\bfG,0})}{dg(G_{\x,0})}
\cdot\int_{G/Z_{\bfG}}\int_{K}\dot{\Theta}_{R_{\bbT_r}^{\bbG_r}(\theta)}({}^{gk}\gamma)\,dk\,dg\\
&=
\frac{dz(Z_{\bfG,0})}{dg(G_{\x,0})}
\cdot\int_{G/Z_{\bfG}}\int_{K_{\gamma_{0}}}\dot{\Theta}_{R_{\bbT_r}^{\bbG_r}(\theta)}({}^{gl}\gamma)\,dl\,dg.
\end{align*}

We utilize the positive-depth Deligne--Lusztig character formula (Theorem \ref{thm:geom-char-formula}):
\[
R_{\bbT_r}^{\bbG_r}(\theta)(\gamma')
= 
\frac{1}{|\bbG_{\gamma'_{0},r}(\F_q)|} \sum_{\substack{x \in \bbG_r(\F_q) \\ {}^{x}\gamma'_{0} \in \bbT_r(\F_q)}} \theta({}^{x}\gamma'_{0}) \cdot Q_{\bbT_r}^{\bbG_{{}^{x}\gamma'_{0},r}}(\theta_{+})({}^{x}\gamma^{\prime}_{+}),
\]
where $\gamma'$ is any element of $G_{\x,0}$ with topological Jordan decomposition $\gamma'=\gamma'_{0}\gamma'_{+}$.
With the obvious usage of the symbol $\dot{(-)}$ denoting the zero extension, we have
\[
\dot{\Theta}_{R_{\bbT_r}^{\bbG_r}(\theta)}(\gamma')
= 
\frac{1}{|\bbG_{\gamma'_{0},r}(\F_q)|} \sum_{x \in \bbG_r(\F_q)} \dot{\theta}({}^{x}\gamma'_{0}) \cdot \dot{Q}_{\bbT_r}^{\bbG_{{}^{x}\gamma'_{0},r}}(\theta_{+})({}^{x}\gamma^{\prime}_{+}).
\]
Hence $\cR_{\bfT}^{\bfG}(\theta)(\gamma)$ equals
\begin{multline*}
\frac{dz(Z_{\bfG,0})}{dg(G_{\x,0})}
\cdot\int_{G/Z_{\bfG}}\int_{K_{\gamma_{0}}}\frac{1}{|\bbG_{{}^{gl}\gamma_{0},r}(\F_q)|}\sum_{x \in \bbG_r(\F_q)} \dot{\theta}({}^{xgl}\gamma_{0}) \cdot \dot{Q}_{\bbT_r}^{\bbG_{{}^{xgl}\gamma_{0},r}}(\theta_{+})({}^{xgl}\gamma_{+})\,dl\,dg\\
=
\frac{dz(Z_{\bfG,0})}{dg(G_{\x,0})}\cdot\frac{|\bbG_r(\F_q)|}{|\bbG_{\gamma_{0},r}(\F_q)|}
\cdot\int_{G/Z_{\bfG}}\dot{\theta}({}^{g}\gamma_{0})\int_{K_{\gamma_{0}}}\dot{Q}_{\bbT_r}^{\bbG_{{}^{g}\gamma_{0},r}}(\theta_{+})({}^{gl}\gamma_{+})\,dl\,dg
\end{multline*}
by letting $x$ be absorbed into the integral over $g$ and also noting that any $l$ commutes with $\gamma_{0}$.

By dividing $G/Z_{\bfG}$ into double cosets $TG_{\x,r+}\backslash G/G_{\gamma_{0}}$, we have
\begin{multline*}
\int_{G/Z_{\bfG}}\dot{\theta}({}^{g}\gamma_{0})\int_{K_{\gamma_{0}}}\dot{Q}_{\bbT_r}^{\bbG_{{}^{g}\gamma_{0},r}}(\theta_{+})({}^{gl}\gamma_{+})\,dl\,dg\\
=
\sum_{\begin{subarray}{c} x\in TG_{\x,r+}\backslash G/G_{\gamma_{0}} \\ {}^{x}\gamma_{0}\in {}^{G_{\x,r+}}T \end{subarray}}
\int_{TG_{\x,r+}xG_{\gamma_{0}}/Z_{\bfG}}\dot{\theta}({}^{g}\gamma_{0}) \int_{K_{\gamma_{0}}}\dot{Q}_{\bbT_r}^{\bbG_{{}^{g}\gamma_{0},r}}(\theta_{+})({}^{gl}\gamma_{+})\,dl\,dg.
\end{multline*}
We fix a Haar measure $dh$ on $G_{\gamma_{0}}$ and consider the quotient measure $d\bar{g}$ on $G/G_{\gamma_{0}}$ of $dg$ by $dh$.
Then each summand of the above sum equals
\[
\int_{TG_{\x,r+}xG_{\gamma_{0}}/G_{\gamma_{0}}}
\int_{G_{\gamma_{0}}/Z_{\bfG}}
\dot{\theta}({}^{\bar{g}h}\gamma_{0}) \int_{K_{\gamma_{0}}}\dot{Q}_{\bbT_r}^{\bbG_{{}^{\bar{g}h}\gamma_{0},r}}(\theta_{+})({}^{\bar{g}hl}\gamma_{+})\,dl\,dh\,d\bar{g}.
\]
Here, in the middle integral, we consider the quotient measure of $dh$ by $dz$ and use the same symbol $dh$ for it.
Since $h$ commutes with $\gamma_{0}$ and the integrand of the outer integral is invariant under the left $TG_{\x,r+}$-translation, this equals
\begin{align*}
&d\bar{g}(TG_{\x,r+}xG_{\gamma_{0}}/G_{\gamma_{0}})
\int_{G_{\gamma_{0}}/Z_{\bfG}}
\dot{\theta}({}^{x}\gamma_{0}) \int_{K_{\gamma_{0}}}\dot{Q}_{\bbT_r}^{\bbG_{{}^{x}\gamma_{0},r}}(\theta_{+})({}^{xhl}\gamma_{+})\,dl\,dh\\
&=d\bar{g}(TG_{\x,r+}xG_{\gamma_{0}}/G_{\gamma_{0}})
\cdot\dot{\theta}({}^{x}\gamma_{0}) 
\int_{G_{\gamma_{0}}/Z_{\bfG}}
\int_{K_{\gamma_{0}}}\dot{Q}_{\bbT^{x}_r}^{\bbG_{\gamma_{0},r}}(\theta^x_{+})({}^{hl}\gamma_{+})\,dl\,dh.
\end{align*}

We note that the following natural surjective map is in fact bijective (this is a simple application of \cite[Lemma 9.10]{AS08}; cf.\ \cite[Lemma 5.13]{Oi23-TECR}):
\[
T\backslash\{x\in G \mid {}^{x}\gamma_{0}\in T\}/G_{\gamma_{0}}
\twoheadrightarrow
TG_{\x,r+}\backslash\{x\in G \mid {}^{x}\gamma_{0}\in {}^{G_{\x,r+}}T\}/G_{\gamma_{0}}.
\]
Hence the above sum can be rewritten as
\[
\sum_{\begin{subarray}{c} x\in T\backslash G/G_{\gamma_{0}} \\ {}^{x}\gamma_{0}\in T \end{subarray}}
d\bar{g}(TG_{\x,r+}xG_{\gamma_{0}}/G_{\gamma_{0}})
\cdot\theta({}^{x}\gamma_{0}) 
\int_{G_{\gamma_{0}}/Z_{\bfG}}
\int_{K_{\gamma_{0}}}\dot{Q}_{\bbT^{x}_r}^{\bbG_{\gamma_{0},r}}(\theta^x_{+})({}^{hl}\gamma_{+})\,dl\,dh.
\]

We also note that
\[
TG_{\x,r+}xG_{\gamma_{0}}
=G_{\x,r+}TxG_{\gamma_{0}}
=G_{\x,r+}xT^{x}G_{\gamma_{0}}
=G_{\x,r+}xG_{\gamma_{0}}
=G_{\x,r+}G_{{}^{x}\gamma_{0}}x
\]
for any $x$ satisfying ${}^{x}\gamma_{0}\in T$.
Hence, the map $g\mapsto gx^{-1}$ induces a bijection 
\[
TG_{\x,r+}xG_{\gamma_{0}}/G_{\gamma_{0}}
\xrightarrow{1:1}
G_{\x,r+}G_{{}^{x}\gamma_{0}}/G_{{}^{x}\gamma_{0}}.
\]
As the latter is bijective to $G_{\x,r+}/G_{{}^{x}\gamma_{0},\x,r+}$ (this follows from \cite[Lemmas 5.29 and 5.33]{AS08}; cf.\ the proof of \cite[Lemma 2.4]{AS09}), we get
\[
d\bar{g}(TG_{\x,r+}xG_{\gamma_{0}}/G_{\gamma_{0}})
=\frac{dg(G_{\x,r+})}{dh(G_{{}^{x}\gamma_{0},\x,r+})}
=\frac{dg(G_{\x,r+})}{dh(G_{\gamma_{0},\x,r+})}
\]
(here, in the middle, we introduce a Haar measure on $G_{{}^{x}\gamma_{0}}$ naturally induced from that on $G_{\gamma_{0}}$ and use the same symbol $dh$ for denoting it).

In summary, so far we have obtained
\begin{multline*}
\cR_{\bfT}^{\bfG}(\theta)(\gamma)
=
\frac{dz(Z_{\bfG,0})}{dg(G_{\x,0})}\cdot\frac{|\bbG_r(\F_q)|}{|\bbG_{\gamma_{0},r}(\F_q)|}
\cdot\frac{dg(G_{\x,r+})}{dh(G_{\gamma_{0},\x,r+})}\\
\cdot\sum_{\begin{subarray}{c} x\in T\backslash G/G_{\gamma_{0}} \\ {}^{x}\gamma_{0}\in T \end{subarray}}
\theta({}^{x}\gamma_{0}) 
\int_{G_{\gamma_{0}}/Z_{\bfG}}
\int_{K_{\gamma_{0}}}\dot{Q}_{\bbT^{x}_r}^{\bbG_{\gamma_{0},r}}(\theta^x_{+})({}^{hl}\gamma_{+})\,dl\,dh.
\end{multline*}
Here recall that $\bbG_{r}(\F_{q})\cong G_{\x,0}/G_{\x,r+}$, hence we have $|\bbG_r(\F_q)|\cdot dg(G_{\x,r+})=dg(G_{\x,0})$.
Similarly, we have $|\bbG_{\gamma_{0},r}(\F_q)|\cdot dh(G_{\gamma_{0}\x,r+})=dh(G_{\gamma_{0},\x,0})$.
Thus, we obtain
\begin{align*}
\cR_{\bfT}^{\bfG}(\theta)(\gamma)
=
\frac{dz(Z_{\bfG,0})}{dh(G_{\gamma_{0},\x,0})}\sum_{\begin{subarray}{c} x\in T\backslash G/G_{\gamma_{0}} \\ {}^{x}\gamma_{0}\in T \end{subarray}}
\theta({}^{x}\gamma_{0}) 
\int_{G_{\gamma_{0}}/Z_{\bfG}}
\int_{K_{\gamma_{0}}}\dot{Q}_{\bbT^{x}_r}^{\bbG_{\gamma_{0},r}}(\theta^x_{+})({}^{hl}\gamma_{+})\,dl\,dh.
\end{align*}

If we apply this formula to $(\bfG_{\gamma_{0}},\bfT^{x},\theta^x,\gamma_{+})$ instead of $(\bfG,\bfT,\theta,\gamma)$, we also obtain
\begin{align*}
\cR_{\bfT^{x}}^{\bfG_{\gamma_{0}}}(\theta^x)(\gamma_{+})
=
\frac{dz'(Z_{\bfG_{\gamma_{0}},0})}{dh(G_{\gamma_{0},\x,0})}
\int_{G_{\gamma_{0}}/Z_{\bfG_{\gamma_{0}}}}
\int_{K_{\gamma_{0}}}\dot{Q}_{\bbT^{x}_r}^{\bbG_{\gamma_{0},r}}(\theta^x_{+})({}^{hl}\gamma_{+})\,dl\,dh',
\end{align*}
where we fix a Haar measure $dz'$ on $Z_{\bfG_{\gamma_{0}}}$ and consider the quotient measure $dh'$ of $dh$ by $dz'$.
Here, note that we have a chain 
\[
\bfZ_{\bfG}
\subset \bfZ_{\bfG_{\gamma_{0}}}
\subset \bfT^{x}
\subset \bfG_{\gamma_{0}}
\subset \bfG.
\]
Since $\bfT^{x}$ is elliptic in $\bfG$, which means that $\bfT^{x}$ is anisotropic modulo $\bfZ_{\bfG}$, we see that $Z_{\bfG_{\gamma_{0}}}$ is compact modulo $Z_{\bfG}$.
Thus it makes sense to choose $dz'$ so that its restriction to $Z_{\bfG}$ coincides with $dz$.
With these choices of Haar measures, we obtain
\begin{align*}
\cR_{\bfT}^{\bfG}(\theta)(\gamma)
&=
\frac{dz(Z_{\bfG,0})}{dh(G_{\gamma_{0},\x,0})}\sum_{\begin{subarray}{c} x\in T\backslash G/G_{\gamma_{0}} \\ {}^{x}\gamma_{0}\in T \end{subarray}}
\theta({}^{x}\gamma_{0}) 
\int_{G_{\gamma_{0}}/Z_{\bfG}}
\int_{K_{\gamma_{0}}}\dot{Q}_{\bbT^{x}_r}^{\bbG_{\gamma_{0},r}}(\theta^x_{+})({}^{hl}\gamma_{+})\,dl\,dh\\
&=
\frac{dz'(Z_{\bfG_{\gamma_{0}},0})}{dh(G_{\gamma_{0},\x,0})}\sum_{\begin{subarray}{c} x\in T\backslash G/G_{\gamma_{0}} \\ {}^{x}\gamma_{0}\in T \end{subarray}}
\theta({}^{x}\gamma_{0}) 
\int_{G_{\gamma_{0}}/Z_{\bfG_{\gamma_{0}}}}
\int_{K_{\gamma_{0}}}\dot{Q}_{\bbT^{x}_r}^{\bbG_{\gamma_{0},r}}(\theta^x_{+})({}^{hl}\gamma_{+})\,dl\,dh'\\
&=
\sum_{\begin{subarray}{c} x\in T\backslash G/G_{\gamma_{0}} \\ {}^{x}\gamma_{0}\in T \end{subarray}}
\theta({}^{x}\gamma_{0})\cdot\cR_{\bfT^{x}}^{\bfG_{\gamma_{0}}}(\theta^x)(\gamma_{+}).
\end{align*}
This completes the proof.
\end{proof}

We finally rewrite the topologically unipotent contribution $\cR_{\bfT^{x}}^{\bfG_{\gamma_{0}}}(\theta^x)(\gamma_{+})$ following the argument of \cite[Lemma 12.4.3]{DR09}.
Recall that we have fixed an element $X^{\ast}\in\mft^{\ast}(F)_{-r}\subset\mfg^{\ast}(F)$ satisfying $\theta(\exp(Y))=\psi_{F}(\langle X^{\ast},Y\rangle)$ for any $Y\in\mft(F)_{0+}$.
As $X_{r}^{\ast}$ is elliptic regular semisimple, its centralizer group $G_{X^{\ast}}$ in $G$ is a subgroup of $N_{G}(\bfT)$ containing $T$.

We let $\widehat{\mu}_{X^{\ast}}^{G}(-)\colon \mfg(F)\rightarrow\C$ denote the locally constant function on $\mfg(F)$ representing the Fourier transform of the orbital integral with respect to $X^{\ast}$.
To be more precise, we consider the distribution
\[
\widehat{\mu}_{X^{\ast}}^{G}\colon C_{c}^{\infty}(\mfg(F))\rightarrow\C;\quad
f\mapsto \int_{G/G_{X^{\ast}}}\widehat{f}({}^{g}X^{\ast})\,dg.
\]
Here,
\begin{itemize}
\item
for $f\in C_{c}^{\infty}(\mfg(F))$, we let $\widehat{f}\in C_{c}^{\infty}(\mfg(F))$ denotes the Fourier transform of $f$, i.e., 
\[
\widehat{f}(Y^{\ast}):=\int_{\mfg(F)} f(X)\cdot\psi_{F}(\langle Y^{\ast},X\rangle)\,dX,
\]
where $dX$ is any Haar measure on $\mfg(F)$,
\item
we fix the unique Haar measure $dt$ on $G_{X^{\ast}}$ which extends $dz$ (note that $Z_{\bfG}$ is of finite index in $T$, hence also in $G_{X^{\ast}}$) and use the induced (quotient) measure on $G/G_{X^{\ast}}$ in the integration over $G/G_{X^{\ast}}$.
\end{itemize}
Then there exists a unique locally constant function on $\mfg(F)$, which is written by the same symbol $\widehat{\mu}_{X^{\ast}}^{G}(-)$, satisfying
\[
\widehat{\mu}_{X^{\ast}}^{G}(f)
=
\int_{\mfg(F)} f(X)\cdot\widehat{\mu}_{X^{\ast}}^{G}(X)\,dX
\]
for any $f\in C_{c}^{\infty}(\mfg(F))$.
Note that the function $\widehat{\mu}_{X^{\ast}}^{G}(-)$ does not depend on the choice of $dX$ but depends on the choice of a Haar measure on $G_{X^{\ast}}$.

Recall that we have fixed a Haar measure $dg$ without any prescription so far.
In the following, we normalize $dg$ such that $dg(G_{\x,0})=|\bbG_{r}(\F_{q})|\cdot q^{-\frac{1}{2}\dim\bbG_{r}}$.
Also, we normalize $dz$ (and $dt$ simultaneously) so that $dt(T_{0})=|\bbT_{r}(\F_{q})|\cdot q^{-\frac{1}{2}\dim\bbT_{r}}$.

\begin{prop}\label{prop:p-adic Springer}
For any topologically unipotent element $\gamma_{+}$ of $G_{\x,0}$, we have
\[
\cR_{\bfT}^{\bfG}(\theta)(\gamma_{+})
=
\widehat{\mu}_{X^{\ast}}^{G}(\log(\gamma_{+})).
\]
\end{prop}

\begin{proof}
As observed in the proof of Proposition \ref{prop:CF-1st}, we have
\[
\cR_{\bfT}^{\bfG}(\theta)(\gamma_{+})
=
\frac{dz(Z_{\bfG,0})}{dg(G_{\x,0})}
\int_{G/Z_{\bfG}}
\int_{K}\dot{Q}_{\bbT_r}^{\bbG_{r}}(\theta_{+})({}^{gk}\gamma_{+})\,dk\,dg.
\]
On the other hand, by \cite[Proposition 3.3.1 (a)]{AD04}, we have
\[
\widehat{\mu}_{X^{\ast}}^{G}(Y)
=
\frac{|\bbT_{r}(\F_{q})|}{|\bbG_{r}(\F_{q})_{X^{\ast}}|}\cdot\frac{dz(Z_{\bfG,0})}{dt(T_{0})}
\int_{G/Z_{\bfG}}\int_{K}\int_{G_{\x,0}} \psi_{F}(\langle X^{\ast},{}^{k'gk}Y\rangle)\,dk'\,dk\,dg
\]
for any regular semisimple element $Y\in\mfg(F)$, where we choose a Haar measure $dk'$ on $G_{\x,0}$ so that $dk'(G_{\x,0})=1$.
We give a remark about the measures used here.
In \cite[Section 3.2]{AD04}, the orbital integral with respect to $X^{\ast}$ is defined to be the integral over $G/A_{\bfG}$ (not $G/G_{X^{\ast}}$), where $\mathbf{A}_{\bfG}$ denotes the maximal split central torus of $\bfG$.
Also, the right-hand side of the formula of \cite[Proposition 3.3.1 (a)]{AD04} is given by the integration over $G/A_{\bfG}$ and does not have the factor $dz(Z_{\bfG,0})/dt(T_{0})$.
With our choices of Haar measures, we have $\int_{G/A_{\bfG}}=[Z_{\bfG}:A_{\bfG}]\int_{G/Z_{\bfG}}$ and $\int_{G/A_{\bfG}}=[G_{X^{\ast}}:A_{\bfG}]\int_{G/G_{X^{\ast}}}$.
By also noting that
\[
\frac{[G_{X^{\ast}}:A_{\bfG}]}{[Z_{\bfG}:A_{\bfG}]}
=|\bbG_{r}(\F_{q})_{X^{\ast}}/\bbT_{r}(\F_{q})|\cdot\frac{dt(T_{0})}{dz(Z_{\bfG,0})},
\]
we see that the formula \cite[Proposition 3.3.1 (a)]{AD04} is equivalent to the above one.

Thus, to show the asserted identity, it is enough to show that
\[
\dot{Q}_{\bbT_r}^{\bbG_{r}}(\theta_{+})(\gamma_{+})
=
\frac{|\bbT_{r}(\F_{q})|}{|\bbG_{r}(\F_{q})_{X^{\ast}}|}\cdot
\frac{dg(G_{\x,0})}{dt(T_{0})}\cdot\int_{G_{\x,0}} \psi_{F}(\langle X^{\ast},{}^{k'}Y\rangle)\,dk'
\]
for any topologically unipotent element $\gamma_{+}=\exp(Y)$.
(In the following, we write $k$ for $k'$ to make the notation lighter.)

We first show that we may assume that $Y\in \mfg(F)_{\x,0}$.
In other words, let us check that 
\[
\int_{G_{\x,0}} \psi_{F}(\langle X^{\ast},{}^{k}Y\rangle)\,dk=0
\]
whenever $Y\notin\mfg(F)_{\x,0}$.
For this, we utilize \cite[Lemma 6.1.1]{AD04}; we choose the data $(\bfH,\bfM,y,r,A,t)$ in \textit{loc.\ cit.}\ to be $(\bfT,\bfG,\x,r,X^{\ast},t)$, where $t$ is any real number such that $2t<0<\frac{r}{2}$.
Then the consequence of \textit{loc.\ cit.}\ is that, for any $Y'\in\mfg(F)_{\x,2t}\smallsetminus\mfg(F)_{\x,2t+}$ satisfying $Y'\notin \mft(F)+\mfg(F)_{\x,\frac{r}{2}+t}$, we have
\[
\int_{G_{\x,(\frac{r}{2}-t)+}}\psi_{F}(\langle X^{\ast},{}^{k}Y'\rangle)\,dk=0.
\]
Here, note that $Y'\in \mft(F)+\mfg(F)_{\x,\frac{r}{2}+t}\iff Y'\in\mfg(F)_{\x,\frac{r}{2}+t}$ if $Y'$ is topologically nilpotent.
Thus, since $2t+<\frac{r}{2}+t$, any topologically nilpotent $Y'$ satisfying $Y'\in\mfg(F)_{\x,2t}\smallsetminus\mfg(F)_{\x,2t+}$ necessarily satisfies the condition that $Y'\notin \mft(F)+\mfg(F)_{\x,\frac{r}{2}+t}$.
Consequently, for any topologically nilpotent element $Y'\in\mfg(F)_{\x,2t}\smallsetminus\mfg(F)_{\x,2t+}$, we have 
\[
\int_{G_{\x,(\frac{r}{2}-t)+}}\psi_{F}(\langle X^{\ast},{}^{k}Y'\rangle)\,dk=0.
\]
Now let us suppose that $Y$ is a topologically nilpotent element not belonging to $\mfg(F)_{\x,0}$; let $t\in\R_{<0}$ be the unique negative number such that $Y\in\mfg(F)_{\x,2t}\smallsetminus\mfg(F)_{\x,2t+}$.
Then we have 
\[
\int_{G_{\x,0}}\psi_{F}(\langle X^{\ast},{}^{k}Y\rangle)\,dk
=
\sum_{g\in G_{\x,0}/G_{\x,(\frac{r}{2}-t)+}} \int_{G_{\x,(\frac{r}{2}-t)+}}\psi_{F}(\langle X^{\ast},{}^{kg}Y\rangle)\,dk.
\]
Since ${}^{g}Y$ is a topologically nilpotent element of $Y\in\mfg(F)_{\x,2t}\smallsetminus\mfg(F)_{\x,2t+}$ for any $g\in G_{\x,0}$, this integral is $0$ (apply the previous discussion to $Y'={}^{g}Y$).

We suppose that $Y$ is a topologically nilpotent element of $\mfg(F)_{\x,0}$ in the following.
By our normalizations of $dg$ and $dt$, we have
\begin{align*}
\frac{|\bbT_{r}(\F_{q})|}{|\bbG_{r}(\F_{q})_{X^{\ast}}|}\cdot
\frac{dg(G_{\x,0})}{dt(T_{0})}
=\frac{|\bbT_{r}(\F_{q})|}{|\bbG_{r}(\F_{q})_{X^{\ast}}|}\cdot\frac{|\bbG_{r}(\F_{q})|\cdot q^{-\frac{1}{2}\dim\bbG_{r}}}{|\bbT_{r}(\F_{q})|\cdot q^{-\frac{1}{2}\dim\bbT_{r}}}
=\frac{|\bbG_{r}(\F_{q})|}{|\bbG_{r}(\F_{q})_{X^{\ast}}|}\cdot q^{-\frac{1}{2}\dim(\bbG_{r}/\bbT_{r})}.
\end{align*}
On the other hand, we have 
\begin{align*}
\int_{G_{\x,0}} \psi_{F}(\langle X^{\ast},{}^{k'}Y\rangle)\,dk'
&=\sum_{k'\in G_{\x,0}/G_{\x,r+}} \psi_{F}(\langle {}^{k'}X^{\ast},Y\rangle)\cdot dk'(G_{\x,r+})\\
&=\sum_{k'\in G_{\x,0}/(G_{\x,0})_{X^{\ast}}G_{\x,r+}} \psi_{F}(\langle {}^{k'}X^{\ast},Y\rangle)\cdot |\bbG_{r}(\F_{q})|^{-1}\cdot |\bbG_{r}(\F_{q})_{X^{\ast}}|\\
&=\sum_{Y^{\ast}\in \bbG_{r}(\F_{q})\cdot X^{\ast}} \psi_{F}(\langle Y^{\ast},Y\rangle)\cdot |\bbG_{r}(\F_{q})|^{-1}\cdot |\bbG_{r}(\F_{q})_{X^{\ast}}|.
\end{align*}
With the notation in Section \ref{subsec:pos-Springer},
\[
\sum_{Y^{\ast}\in \bbG_{r}(\F_{q})\cdot X^{\ast}} \psi_{F}(\langle Y^{\ast},Y\rangle)
=\sum_{Y^{\ast}\in \bbG_{r}(\F_{q})\cdot X_{r}^{\ast}} \psi_{r}(\langle Y^{\ast},Y\rangle)
=\FT(\mathbbm{1}_{X_{r}^{\ast}})(Y).
\]
Since $\FT(\mathbbm{1}_{X_{r}^{\ast}})(Y)$ equals $Q_{\bbT_r}^{\bbG_{r}}(\theta_{+})(\log(\gamma_{+}))\cdot q^{\frac{1}{2}\dim(\bbG_{r}/\bbT_{r})}$ by the positive-depth Springer hypothesis (Theorem \ref{thm:pos Springer}), this completes the proof.
\end{proof}

\begin{cor}[cf.\ {\cite[Theorem 4.3.5]{FKS23}}]
For any regular semisimple element $\gamma\in G$, we have 
\[
\Theta_{\pi_{(\bfT,\theta)}^{\FKS}}(\gamma)
=
(-1)^{r(\bfT,\theta)}\sum_{\begin{subarray}{c} x\in T\backslash G/G_{\gamma_{0}} \\ {}^{x}\gamma_{0}\in T \end{subarray}}
\theta({}^{x}\gamma_{0})\cdot
\widehat{\mu}_{X^{\ast}}^{G_{{}^{x}\gamma_{0}}}(\log({}^{x}\gamma_{+})).
\]
\end{cor}

\begin{proof}
By Proposition \ref{prop:CF-1st},
\[
\Theta_{\pi_{(\bfT,\theta)}^{\FKS}}(\gamma)
=
(-1)^{r(\bfT,\theta)}\sum_{\begin{subarray}{c} x\in T\backslash G/G_{\gamma_{0}} \\ {}^{x}\gamma_{0}\in T \end{subarray}}
\theta({}^{x}\gamma_{0})\cdot\cR_{\bfT^{x}}^{\bfG_{\gamma_{0}}}(\theta^x)(\gamma_{+}).
\]
The summand equals
\[
\theta({}^{x}\gamma_{0})\cdot\cR_{\bfT}^{\bfG_{{}^{x}\gamma_{0}}}(\theta)({}^{x}\gamma_{+})
=\theta({}^{x}\gamma_{0})\cdot\widehat{\mu}^{G_{{}^{x}\gamma_{0}}}_{X^{\ast}}(\log({}^{x}\gamma_{+}))
\]
by Proposition \ref{prop:p-adic Springer}.
\end{proof}

\section{Discussion of small $q$}\label{sec:small q}

Let $(\bfT,\theta)$ be a Howe-unramified elliptic regular pair.
In our preceding work \cite{CO25}, we established a comparison between the regular supercuspidal representation and positive-depth Deligne--Lusztig representations in the $0$-toral setting by assuming the inequality
\[
\frac{|\bbT_{0}(\F_{q})|}{|\bbT_{0}(\F_{q})_{\nvreg}|}>2.
\]
On the other hand, in this paper, we established a comparison result for $(\bfT,\theta)$ not necessarily toral (Theorem \ref{thm:reg comparison}) by assuming the following stronger inequality
\[
\frac{|\bbT_{0}(\F_{q})|}{|\bbT_{0}(\F_{q})_{\nvreg}|}>2\cdot|W_{\bbG_{0}}(\bbT_{0})(\F_{q})|.
\]
We refer to the former and the latter inequalities as the \textit{weak Henniart inequality} and the \textit{strong Henniart inequality}, respectively.

In principle, it is not very difficult to explicate these inequalities as long as the input data $(\bfG,\bfT)$ is given explicitly.
For example, if $\bfG$ is an exceptional simple group of adjoint type and $\bfT$ is an unramified elliptic maximal torus of $\bfG$ whose $\bbT_{0}$ is an elliptic maximal torus of $\bbG_{0}$ of ``Coxeter type'', the Henniart inequalities are described as in Table \ref{table:Henniart}.
(The very regularity of any semisimple element of $\bbT_{0}(\F_{q})$ coincides with the usual regularity in this case).
See \cite[Appendix A.1]{CO23-sc} for details.

\begin{table}[hbtp]
\caption{Henniart inequalities for Coxeter tori of exceptional groups}\label{table:Henniart}
\begin{tabular}{|c|c|c|c|c|} \hline
$G$ & $|\bbT_{0}(\F_{q})|$ & $|\bbT_{0}(\F_{q})_{\nreg}|$ & weak & strong \\ \hline
$E_{6}$ & $(q^{4}-q^{2}+1)(q^{2}+q+1)$ & $q^{2}+q+1$ & $q$: any & $q>2$ \\\hline
$E_{7}$ & $(q^{6}-q^{3}+1)(q+1)$ & $\begin{cases}3(q+1)&\text{$q\equiv-1\mod3$}\\q+1&\text{$q\not\equiv-1\mod3$}\end{cases}$ &  $q$: any & $\begin{cases}q>2\\ \text{$q$: any}\end{cases}$\\ \hline
$E_{8}$ & $q^{8}+q^{7}-q^{5}-q^{4}-q^{3}+q+1$ & 1  & $q$: any & $q$: any \\ \hline
$F_{4}$ & $q^{4}-q^{2}+1$ & $1$ & $q$: any & $q>2$ \\ \hline
$G_{2}$ & $q^{2}-q+1$ & $\begin{cases}3&\text{$q\equiv-1\mod3$}\\1&\text{$q\not\equiv-1\mod3$}\end{cases}$ & $\begin{cases}q>2\\ \text{$q$: any}\end{cases}$ & $\begin{cases}q>6\\ q>3 \end{cases}$ \\ \hline
\end{tabular}
\end{table}

Therefore, only the cases which cannot be covered by our results are 
\begin{itemize}
\item
$\bfG$ is of type $E_{6}$, $q=2$;
\item
$\bfG$ is of type $F_{4}$, $q=2$;
\item
$\bfG$ is of type $G_{2}$, $q=2, 3, 5$.
\end{itemize}

The aim of this section is to examine if the characterization result (Theorem \ref{thm:vreg characterization}) is still valid in the setting that $\theta$ has depth zero and $q$ is very small.
More precisely, we investigate the following 
\begin{quest}\label{Q:recover}
Suppose that $\rho$ is an irreducible representation of $\bbG_{0}(\F_{q})$ equipped with a sign $c$ such that, for any regular semisimple element $g\in\bbG_{0}(\F_{q})$, 
\[
\Theta_{\rho}(g)
=
c \cdot \Theta_{R_{\bbT_{0}}^{\bbG_{0}}(\theta)} (g).
\]
Then do we necessarily have $\rho\cong c R_{\bbT_{0}}^{\bbG_{0}}(\theta)$?
\end{quest}

\subsection{Example: $G_{2}$}\label{subsec:G2}

In this subsection, we assume that $\bfG=G_{2}$ and $\bfT$ is its unramified elliptic maximal torus whose $\bbT_{0}$ is of Coxeter type.
Because we only consider the depth-zero case, we drop the subscript ``$0$'' from the symbols $\bbG_{0}$ and $\bbT_{0}$.
Let $\theta$ be a regular character of $T$ of depth zero and we again write $\theta$ for the induced regular character of $\bbT(\F_{q})$.

In fact, as we will see, the answer to the above question in this case is as follows:
\begin{itemize}
\item
When $q=2$, there is no regular character of $\bbT(\F_{q})$ (thus the assumption of this problem is empty).
\item
When $q=3$, there is a counterexample, i.e., there exists a $\rho$ satisfying the assumption of Question \ref{Q:recover} but not isomorphic to $c \cdot \Theta_{R_{\bbT}^{\bbG}(\theta)}$.
\item
When $q=5$, this problem depends on $\theta$;
\begin{itemize}
\item
if $\theta|_{\bbT(\F_{5})_{\nreg}}\not\equiv\mathbbm{1}$, then the answer is affirmative, but
\item
if $\theta|_{\bbT(\F_{5})_{\nreg}}\equiv\mathbbm{1}$, then there is a counterexample.
\end{itemize}
\end{itemize}

\subsubsection{Structures of maximal tori of $G_{2}(\F_{q})$}
For any connected reductive group $\bbG$ over $\F_{q}$, the rational conjugacy classes of $\F_{q}$-rational maximal tori are bijectively parametrized by the Frobenius-twisted conjugacy classes in the Weyl group of $\bbG$.
The latter is classified by Carter \cite{Car72} in the case where $\bbG$ is split.

The group $G_{2}$ has $6$ conjugacy classes; they are named ``$\varnothing$'', ``$A_{1}$'', ``$\tilde{A}_{1}$'', ``$A_{1}\times\tilde{A}_{1}$'', ``$A_{2}$'', and ``$G_{2}$'' (see \cite[Table 7]{Car72}).
For any such conjugacy class $\Gamma$, let us write $\bbT_{\Gamma}$ for an $\F_{q}$-rational maximal torus corresponding to $\Gamma$.
Then the orders of $\bbT_{\Gamma}(\F_{q})$ and $W_{\bbT_{\Gamma}}:=W_{\bbG}(\bbT_{\Gamma})(\F_{q})$ are given as follows (see also \cite[Table 3 and Lemma 26]{Car72}):
\begin{table}[hbtp]
\caption{Maximal tori of $G_{2}(\F_{q})$}\label{table:tori}
\begin{tabular}{|c|c|c|c|c|} \hline
$\Gamma$ & $|\bbT_{\Gamma}(\F_{q})|$ & $|\bbT_{\Gamma}(\F_{q})|$ ($q=3$) & $|W_{\bbT_{\Gamma}}|$ & split rank \\ \hline
$\varnothing$ & $(q-1)^{2}$ & $4$ & $12$ & $2$ (split) \\ \hline
$A_{1}$ & $(q-1)(q+1)$ & $8$  & $4$ & $1$ \\ \hline
$\tilde{A}_{1}$ & $(q-1)(q+1)$ & $8$  & $4$ & $1$  \\ \hline
$A_{1}\times\tilde{A}_{1}$ & $(q+1)^{2}$ & $16$  & $12$ & $0$ (elliptic) \\ \hline
$A_{2}$ & $q^{2}+q+1$ & $13$  & 6 & $0$ (elliptic) \\ \hline
$G_{2}$ & $q^{2}-q+1$ & $7$  & 6 & $0$ (elliptic, Coxeter) \\ \hline
\end{tabular}
\end{table}

Note that ``$G_{2}$'' is the conjugacy class of Coxeter elements.
Hence $\bbT_{G_{2}}$ is our maximal torus $\bbT$.
We can check that $\bbT(\F_{q})$ has a non-regular semisimple element other than unit if and only if $q\equiv-1\pmod{3}$; in this case, the number of non-regular semisimple elements is $3$.
Also note that the rational Weyl group $W_{\bbT}$ is cyclic of order $6$.

\subsubsection{The case of $G_{2}(\F_{2})$}
We first consider the case where $q=2$.
In this case, we have $\bbT(\F_{2})=\bbT(\F_{2})_{\nreg}\cong\Z/3\Z$ while $W_{\bbT}\cong\Z/6\Z$.
In particular, the action of $W_{\bbT}$ on $\bbT(\F_{2})$ necessarily factors through a nontrivial quotient of $W_{\bbT}$.
This means that there cannot exist a regular character of $\bbT(\F_{2})$.

\subsubsection{The case of $G_{2}(\F_{3})$}
We next consider the case where $q=3$.

The group $G_{2}(\F_{3})$ has $23$ conjugacy classes, hence has $23$ irreducible representations.
Table \ref{table:conjugacy} is the list of $23$ conjugacy classes; if a conjugacy class has name ``$n$x'', then it means that the order of any representative of the class is given by $n$.
The last column of Table \ref{table:conjugacy} expresses which tori contain the (semisimple) conjugacy classes.

\begin{table}[hbtp]
\caption{Conjugacy classes of $G_{2}(\F_{3})$}\label{table:conjugacy}
\begin{tabular}{|c|c|c|c|c|} \hline
conjugacy class & order & order of centralizer & type & tori \\ \hline
1a & 1 & 4245696 & unit & all \\ \hline
2a & 2 & 576 & ss. & $\varnothing, A_{1}, \tilde{A}_{1}, A_{1}\times\tilde{A}_{1}$ \\ \hline
3a & 3 & 5832 & unip. & -- \\ \hline
3b & 3 & 5832 & unip. & -- \\ \hline
3c & 3 & 729 & unip. & -- \\ \hline
3d & 3 & 162 & unip. & -- \\ \hline
3e & 3 & 162 & unip. & -- \\ \hline
4a & 4 & 96 & ss. & $\tilde{A}_{1}, A_{1}\times\tilde{A}_{1}$ \\ \hline
4b & 4 & 96 & ss. & $A_{1}, A_{1}\times\tilde{A}_{1}$  \\ \hline
6a & 6 & 72 & -- & -- \\ \hline
6b & 6 & 72 & -- & -- \\ \hline
6c & 6 & 18 & -- & -- \\ \hline
6d & 6 & 18 & -- & -- \\ \hline
7a & 7 & 7 & reg.\ ss. & $G_{2}$ \\ \hline
8a & 8 & 8 & reg.\ ss. & $\tilde{A}_{1}$ \\ \hline
8b & 8 & 8 & reg.\ ss. & $A_{1}$ \\ \hline
9a & 9 & 27 & unip. & -- \\ \hline
9b & 9 & 27 & unip. & -- \\ \hline
9c & 9 & 27 & unip. & -- \\ \hline
12a & 12 & 12 & -- & -- \\ \hline
12b & 12 & 12 & -- & -- \\ \hline
13a & 13 & 13 & reg.\ ss. & $A_{2}$ \\ \hline
13b & 13 & 13 & reg.\ ss. & $A_{2}$ \\ \hline
\end{tabular}
\end{table}

The character table of $G_{2}(\F_{3})$ is as in Table \ref{table:characters}; the 23 irreducible representations are named ``$X_{n}$'' in the decreasing order by their dimensions.
(This table is cited from GAP3 (\cite{GAP3}); if the reader is familiar with GAP3, Table \ref{table:characters} can be output just by typing:
\[
\verb|>gap DisplayCharTable( CharTable( "G2(3)" ) );|
\]

\begin{landscape}
\begin{table}[hbtp]
\caption{Character table of $G_{2}(\F_{3})$}\label{table:characters}
\begin{tabular}{|c|c|c|c|c|c|c|c|c|c|c|c|c|c|c|c|c|c|c|c|c|c|c|c|} \hline
&1a&2a&3a&3b&3c&3d&3e&4a&4b&6a&6b&6c&6d&7a&8a&8b&9a&9b &9c&12a&12b&13a&13b\\ \hline
$X_{1}$&$1$&$1$&$1$&$1$&$1$&$1$&$1$&$1$&$1$&$1$&$1$&$1$&$1$&$1$&$1$&$1$&$1$&$1$&$1$&$1$&$1$&$1$&$1$ \\ \hline
$X_{2}$&$14$&$-2$&$5$&$5$&$-4$&$2$&$-1$&$2$&$2$&$1$&$1$&$-2$&$1$&$\cdot$&$\cdot$&$\cdot$&$2$&$-1$&$-1$&$-1$&$-1$&$1$&$1$ \\ \hline
$X_{3}$&$64$&$\cdot$&$-8$&$-8$&$1$&$4$&$-2$&$\cdot$&$\cdot$&$\cdot$&$\cdot$&$\cdot$&$\cdot$&$1$&$\cdot$&$\cdot$&$1$&$\alpha$&$\beta$&$\cdot$&$\cdot$&$-1$&$-1$ \\ \hline
$X_{4}$&$64$&$\cdot$&$-8$&$-8$&$1$&$4$&$-2$&$\cdot$&$\cdot$&$\cdot$&$\cdot$&$\cdot$&$\cdot$&$1$&$\cdot$&$\cdot$&$1$&$\beta$&$\alpha$&$\cdot$&$\cdot$&$-1$&$-1$ \\ \hline
$X_{5}$&$78$&$-2$&$-3$&$-3$&$-3$&$-3$&$6$&$2$&$2$&$1$&$1$&$1$&$-2$&$1$&$\cdot$&$\cdot$&$\cdot$&$\cdot$&$\cdot$&$-1$&$-1$&$\cdot$&$\cdot$ \\ \hline
$X_{6}$&$91$&$-5$&$10$&$10$&$10$&$1$&$1$&$3$&$3$&$-2$&$-2$&$1$&$1$&$\cdot$&$-1$&$-1$&$1$&$1$&$1$&$\cdot$&$\cdot$&$\cdot$&$\cdot$ \\ \hline
$X_{7}$&$91$&$3$&$-8$&$19$&$1$&$4$&$-2$&$3$&$-1$&$\cdot$&$3$&$\cdot$&$\cdot$&$\cdot$&$1$&$-1$&$-2$&$1$&$1$&$\cdot$&$-1$&$\cdot$&$\cdot$ \\ \hline
$X_{8}$&$91$&$3$&$19$&$-8$&$1$&$4$&$-2$&$-1$&$3$&$3$&$\cdot$&$\cdot$&$\cdot$&$\cdot$&$-1$&$1$&$-2$&$1$&$1$&$-1$&$\cdot$&$\cdot$&$\cdot$ \\ \hline
$X_{9}$&$104$&$8$&$14$&$14$&$5$&$2$&$-1$&$\cdot$&$\cdot$&$2$&$2$&$2$&$-1$&$-1$&$\cdot$&$\cdot$&$2$&$-1$&$-1$&$\cdot$&$\cdot$&$\cdot$&$\cdot$ \\ \hline
$X_{10}$&$168$&$8$&$6$&$6$&$6$&$-3$&$6$&$\cdot$&$\cdot$&$2$&$2$&$-1$&$2$&$\cdot$&$\cdot$&$\cdot$&$\cdot$&$\cdot$&$\cdot$&$\cdot$&$\cdot$&$-1$&$-1$ \\ \hline
$X_{11}$&$182$&$6$&$20$&$-7$&$-7$&$2$&$2$&$2$&$2$&$\cdot$&$-3$&$\cdot$&$\cdot$&$\cdot$&$\cdot$&$\cdot$&$-1$&$-1$&$-1$&$2$&$-1$&$\cdot$&$\cdot$ \\ \hline
$X_{12}$&$182$&$6$&$-7$&$20$&$-7$&$2$&$2$&$2$&$2$&$-3$&$\cdot$&$\cdot$&$\cdot$&$\cdot$&$\cdot$&$\cdot$&$-1$&$-1$&$-1$&$-1$&$2$&$\cdot$&$\cdot$ \\ \hline
$X_{13}$&$273$&$-7$&$30$&$3$&$3$&$3$&$3$&$-3$&$1$&$2$&$-1$&$-1$&$-1$&$\cdot$&$1$&$-1$&$\cdot$&$\cdot$&$\cdot$&$\cdot$&$1$&$\cdot$&$\cdot$ \\ \hline
$X_{14}$&$273$&$-7$&$3$&$30$&$3$&$3$&$3$&$1$&$-3$&$-1$&$2$&$-1$&$-1$&$\cdot$&$-1$&$1$&$\cdot$&$\cdot$&$\cdot$&$1$&$\cdot$&$\cdot$&$\cdot$ \\ \hline
$X_{15}$&$448$&$\cdot$&$16$&$16$&$-11$&$-2$&$-2$&$\cdot$&$\cdot$&$\cdot$&$\cdot$&$\cdot$&$\cdot$&$\cdot$&$\cdot$&$\cdot$&$1$&$1$&$1$&$\cdot$&$\cdot$&$\gamma$&$\delta$ \\ \hline
$X_{16}$&$448$&$\cdot$&$16$&$16$&$-11$&$-2$&$-2$&$\cdot$&$\cdot$&$\cdot$&$\cdot$&$\cdot$&$\cdot$&$\cdot$&$\cdot$&$\cdot$&$1$&$1$&$1$&$\cdot$&$\cdot$&$\delta$&$\gamma$ \\ \hline
$X_{17}$&$546$&$2$&$-21$&$6$&$6$&$-3$&$-3$&$-2$&$6$&$-1$&$2$&$-1$&$-1$&$\cdot$&$\cdot$&$\cdot$&$\cdot$&$\cdot$&$\cdot$&$1$&$\cdot$&$\cdot$&$\cdot$ \\ \hline
$X_{18}$&$546$&$2$&$6$&$-21$&$6$&$-3$&$-3$&$6$&$-2$&$2$&$-1$&$-1$&$-1$&$\cdot$&$\cdot$&$\cdot$&$\cdot$&$\cdot$&$\cdot$&$\cdot$&$1$&$\cdot$&$\cdot$ \\ \hline
$X_{19}$&$728$&$-8$&$26$&$-28$&$-1$&$-1$&$-1$&$\cdot$&$\cdot$&$-2$&$4$&$1$&$1$&$\cdot$&$\cdot$&$\cdot$&$-1$&$-1$&$-1$&$\cdot$&$\cdot$&$\cdot$&$\cdot$ \\ \hline
$X_{20}$&$728$&$-8$&$-28$&$26$&$-1$&$-1$&$-1$&$\cdot$&$\cdot$&$4$&$-2$&$1$&$1$&$\cdot$&$\cdot$&$\cdot$&$-1$&$-1$&$-1$&$\cdot$&$\cdot$&$\cdot$&$\cdot$ \\ \hline
$X_{21}$&$729$&$9$&$\cdot$&$\cdot$&$\cdot$&$\cdot$&$\cdot$&$-3$&$-3$&$\cdot$&$\cdot$&$\cdot$&$\cdot$&$1$&$-1$&$-1$&$\cdot$&$\cdot$&$\cdot$&$\cdot$&$\cdot$&$1$&$1$ \\ \hline
$X_{22}$&$819$&$3$&$9$&$9$&$9$&$\cdot$&$\cdot$&$-1$&$-1$&$-3$&$-3$&$\cdot$&$\cdot$&$\cdot$&$1$&$1$&$\cdot$&$\cdot$&$\cdot$&$-1$&$-1$&$\cdot$&$\cdot$ \\ \hline
$X_{23}$&$832$&$\cdot$&$-32$&$-32$&$-5$&$4$&$4$&$\cdot$&$\cdot$&$\cdot$&$\cdot$&$\cdot$&$\cdot$&$-1$&$\cdot$&$\cdot$&$1$&$1$&$1$&$\cdot$&$\cdot$&$\cdot$&$\cdot$ \\ \hline
\end{tabular}
every dot denotes $0$, $\alpha:=\frac{-1+3\sqrt{-3}}{2}$, 
$\beta:=\frac{-1-3\sqrt{-3}}{2}$, 
$\gamma:=\frac{-1+\sqrt{13}}{2}$, 
$\delta:=\frac{-1-\sqrt{13}}{2}$.
\end{table}
\end{landscape}

We remark that among the 23 irreducible representations, the unipotent representations are
\[
X_{1}, X_{2}, X_{3}, X_{4}, X_{5}, X_{7}, X_{8}, X_{9}, X_{10}, X_{21}
\]
($X_{2}$, $X_{3}$, $X_{4}$, and $X_{5}$ are cuspidal unipotent representations).
This can be also seen by using GAP3: 
\[
\verb|gap> Display(UnipotentCharacters(CoxeterGroup("G",2)));|
\]
In the GAP3 output, the above unipotent representations are expressed as
\verb|phi{1,0}|, \verb|G2[1]|, \verb|G2[E3]|, \verb|G2[E3^2]|, \verb|G2[-1]|, \verb|phi{1,3}'|, \verb|phi{1,3}''|, \verb|phi{2,1}|, \verb|phi{2,2}|, \verb|phi{1,6}|.
(See \url{https://webusers.imj-prg.fr/~jean.michel/gap3/htm/chap098.htm} and also \cite[372 page]{Lus84}.)

\begin{table}[hbtp]
\caption{Unipotent representations of $G_{2}(\F_{q})$}\label{table:unipotent}
\begin{tabular}{|c|c|c|c|c|} \hline
GAP3 label & dimension & label ($q=3$) & dim ($q=3$) & label (Lusztig) \\ \hline
\verb|phi{1,0}| & $1$ & $X_{1}$ & $1$ & -- \\ \hline
\verb|phi{1,6}| & $q^{6}$ & $X_{21}$ & $729$ & -- \\ \hline
\verb|phi{1,3}'| & $q\phi_{3}(q)\phi_{6}(q)/3$ & $X_{7}$ & $91$ & $(1,r)$ \\ \hline
\verb|phi{1,3}''| & $q\phi_{3}(q)\phi_{6}(q)/3$ & $X_{8}$ & $91$ & $(g_{3},1)$ \\ \hline
\verb|phi{2,1}| & $q\phi_{2}^{2}(q)\phi_{3}(q)/6$ & $X_{9}$ & $104$ & $(1,1)$ \\ \hline
\verb|phi{2,2}| & $q\phi_{2}^{2}(q)\phi_{6}(q)/2$ & $X_{10}$ & $168$ & $(g_{2},1)$ \\ \hline
\verb|G2[-1]| & $q\phi_{1}^{2}(q)\phi_{3}(q)/2$ & $X_{5}$ & $78$ & $(g_{2},c)$ \\ \hline
\verb|G2[1]| & $q\phi_{1}^{2}(q)\phi_{6}(q)/6$ & $X_{2}$ & $14$ & $(1,c)$  \\ \hline
\verb|G2[E3]| & $q\phi_{1}^{2}(q)\phi_{2}^{2}(q)/3$ & $X_{3}$ & $64$ & $(g_{3},\theta)$  \\ \hline
\verb|G2[E3^2]| & $q\phi_{1}^{2}(q)\phi_{2}^{2}(q)/3$ & $X_{4}$ & $64$ & $(g_{3},\theta^{2})$ \\ \hline
\end{tabular}

$\phi_{1}(q)=q-1$, 
$\phi_{2}(q)=q+1$, 
$\phi_{3}(q)=q^{2}+q+1$, 
$\phi_{6}(q)=q^{2}-q+1$.
\end{table}

%

Now let us discuss a counterexample to Question \ref{Q:recover}.
We have $\bbT(\F_{3})\cong\Z/7\Z$ and $\bbT(\F_{3})_{\nreg}=\{1\}$.
Moreover, we can check that $W_{\bbT}$ acts on the set of regular semisimple elements of $\bbT(\F_{3})$ simply-transitively.
Thus we see that there exists only one regular character $\theta$ of $\bbT(\F_{3})$ up to conjugation.

We use the Deligne--Lusztig character formula at a semisimple element (see \cite[Corollary 7.2]{DL76}):
\[
\Theta_{R_{\bbT}^{\bbG}(\theta)}(s)
=
(-1)^{r(\bbG_{s})-r(\bbT)}\frac{1}{\dim\mathrm{St}_{\bbG_{s}} \cdot|\bbT(\F_{q})|}\sum_{\begin{subarray}{c} g \in \bbG(\F_{q}) \\ g^{-1}sg\in\bbT \end{subarray}} \theta(g^{-1}sg).
\]
Here, $s\in \bbG(\F_{q})$ is any semisimple element, $r(-)$ is the split rank, and $\mathrm{St}_{\bbG_{s}}$ denotes the Steinberg representation of $\bbG_{s}(\F_{q})$.
Especially, by choosing $s=1$, we get
\[
\dim R_{\bbT}^{\bbG}(\theta)
=
\frac{|\bbG(\F_{q})|}{\dim\mathrm{St}_{\bbG}\cdot|\bbT(\F_{q})|}
\]
(note that $r(\bbG)=2$ and $r(\bbT)=0$, hence $R_{\bbT}^{\bbG}(\theta)$ is a genuine representation).
Since we have 
\begin{itemize}
\item
$|\bbG(\F_{q})|=q^{6}\cdot (q^{2}-1)\cdot (q^{6}-1)$ (see \cite[Section 2.9]{Car85}), 
\item
$\dim\mathrm{St}_{\bbG}=q^{6}$ (see \cite[Proposition 6.4.4]{Car85}), 
\item
$|\bbT(\F_{q})|=q^{2}-q+1$, 
\end{itemize}
we have 
\[
\dim R_{\bbT}^{\bbG}(\theta)
=(q-1)^{2}\cdot(q+1)^{2}\cdot(q^{2}+q+1)
=832.
\]
Thus we conclude that $R_{\bbT}^{\bbG}(\theta)$ is the irreducible representation $X_{23}$.

By the above description of the group $\bbT(\F_{3})$ and the action of $W_{\bbT}$ on $\bbT(\F_{3})$, we see that 
\[
\sum_{w\in W_{\bbT}}\theta^{w}(t)=\sum_{i=1}^{6}\zeta_{7}^{i}=-1,
\]
where $\zeta_{7}$ is a primitive $7$th root of unity.
This means that the assumption in Question \ref{Q:recover} only says that
\begin{itemize}
\item
we have $\Theta_{\rho}(s)=0$ if the conjugacy class of $s$ is one of ``8a'', ``8b'', ``13a'', and ``13b'' (see Table \ref{table:conjugacy}) and
\item
we have $\Theta_{\rho}(s)=\pm1$ if the conjugacy class of $s$ is ``7a'' (see Table \ref{table:conjugacy}).
\end{itemize}

By looking at the character table (Table \ref{table:characters}), we can easily find that $X_{5}$ and $X_{9}$ satisfy these assumptions!

\subsubsection{The case of $G_{2}(\F_{5})$}
We finally consider the case of $G_{2}(\F_{5})$.
Here we go back to the proof of Theorem \ref{thm:vreg characterization}.
The key in the final step of our proof is the inequality
\[
\left|\sum_{s\in\bbT(\F_{q})_{\nreg}}\sum_{w\in W_{\bbT}}\theta(s)\overline{\theta^{w}(s)}\right|
<
\frac{|\bbT(\F_{q})|}{2}.
\]
To obtain this inequality, we applied the triangle inequality to the left-hand side and then used the strong Henniart inequality.
The idea is to check that this inequality holds also in the current setting without appealing to the triangle inequality but by computing the sum inside the absolute value directly.

We have $\bbT(\F_{5})\cong\Z/3\Z\times\Z/7\Z$ and $\bbT(\F_{5})_{\nreg}\cong\Z/3\Z$.
Moreover, we can check that the action of $W_{\bbT}$ on $\bbT(\F_{5})$ induces a bijective map $\Z/6\Z\xrightarrow{1:1}\mathrm{Aut}(\Z/7\Z)$ and a surjective map $\Z/6\Z\twoheadrightarrow\mathrm{Aut}(\Z/3\Z)$.
Thus, a character $\theta$ of $\bbT(\F_{5})\cong\Z/3\Z\times\Z/7\Z$ is regular if and only if $\theta|_{\Z/7\Z}$ is nontrivial.

We first suppose that a regular character $\theta$ has nontrivial restriction to $\bbT(\F_{5})_{\nreg}\cong\Z/3\Z$.
Let $\zeta_{3}$ be the primitive cube root of unity which is the image of $1\in \Z/3\Z$ under $\theta$.
Then, we have 
\[
\sum_{s\in\bbT(\F_{5})_{\nreg}}\sum_{w\in W_{\bbT}}\theta(s)\overline{\theta^{w}(s)}
=
\sum_{i=1}^{6}1
+\sum_{i=1}^{6}\zeta_{3}\overline{\zeta_{3}^{i}}
+\sum_{i=1}^{6}\zeta_{3}^{2}\overline{\zeta_{3}^{2i}}
=
6.
\]
Since $6$ is smaller than $|\bbT(\F_{5})|/2=21/2=10.5$, our proof still works.

We next suppose that a regular character $\theta$ has trivial restriction to $\bbT(\F_{5})_{\nreg}\cong\Z/3\Z$.
Then, we have 
\[
\sum_{s\in\bbT(\F_{5})_{\nreg}}\sum_{w\in W_{\bbT}}\theta(s)\overline{\theta^{w}(s)}
=|\bbT(\F_{5})_{\nreg}|\cdot|W_{\bbG}(\bbT)|
=3\cdot6
=18.
\]
Since $18$ is greater than $|\bbT(\F_{5})|/2=21/2=10.5$, our proof no longer works!
Indeed, we can find a counterexample Question \ref{Q:recover} in a similar way to the case of $G_{2}(\F_{3})$.
Because the character table of $G_{2}(\F_{5})$ is too large (although it can be easily displayed by using GAP3), we give up describing the details here; we only explain some points.

There exist $44$ conjugacy classes, hence irreducible representations, of $G_{2}(\F_{5})$.
They are labeled in the same way to the case of $G_{2}(\F_{3})$; it starts with ``1a'' and ends with ``31e''.
The regular semisimple conjugacy classes are 6c, 7a, 8a, 8b, 12a, 12b, 21a, 21b, 24a, 24b, 24c, 24d, 31a, 31b, 31c, 31d, and 31e.
Among these, only 7a, 21a, and 21b belong to the Coxeter maximal torus $\bbT(\F_{q})$.
By the same computation as in the $G_{2}(\F_{5})$ case, we see that 
\[
\dim R_{\bbT}^{\bbG}(\theta)
=(q-1)^{2}\cdot(q+1)^{2}\cdot(q^{2}+q+1)
=17856.
\]
There are three irreducible representations whose dimension is $17586$; $X_{40}$, $X_{41}$, and $X_{42}$.
By looking at the character value at 7a (note that this is a regular semisimple element of $\bbT(\F_{5})\cong\Z/3\Z\times\Z/7\Z$ whose $\Z/3\Z$-entry is trivial), we see that $X_{40}$ is our irreducible representation $R_{\bbT}^{\bbG}(\theta)$.
By looking at the character table of $G_{2}(\F_{5})$, we notice that $X_{5}$ (this is the unipotent representation \verb|phi{2,1}|, whose dimension is $930$) cannot be distinguished from $X_{40}$ only by its character values at regular semisimple elements.

\subsection{Unipotent representations}\label{sec:rss and unip}

One may notice that, in both cases of $G_{2}(\F_{3})$ and $G_{2}(\F_{5})$, a counterexample to Question \ref{Q:recover} is given by a unipotent representation.
In fact, this is not an accident.

Let $\bbG$ be a connected reductive group over $\F_{q}$, $\bbT$ an $\F_{q}$-rational maximal torus of $\bbG$, and $\theta$ a regular character of $\bbT(\F_{q})$.
We suppose that $\rho$ is an irreducible representation of $\bbG(\F_{q})$ satisfying the assumption of Question \ref{Q:recover}.

\begin{lem}\label{lem:theta-theta'}
Suppose that there exists a character $\theta'\colon\bbT(\F_{q})\rightarrow\C^{\times}$ such that
\begin{enumerate}
\item[(i)]
$\theta$ and $\theta'$ are not $W_{\bbT}$-conjugate, and
\item[(ii)]
$\theta|_{\bbT(\F_{q})_{\nreg}}\equiv\theta'|_{\bbT(\F_{q})_{\nreg}}$.
\end{enumerate}
Then we have either $\langle\rho,R_{\bbT}^{\bbG}(\theta)\rangle\neq0$ or $\langle\rho,R_{\bbT}^{\bbG}(\theta')\rangle\neq0$.
\end{lem}

\begin{proof}
By the regularity of $\theta$, we have
\begin{align}\label{eq:theta-theta'-1}
1
=\langle R_{\bbT}^{\bbG}(\theta),R_{\bbT}^{\bbG}(\theta)\rangle
=\langle R_{\bbT}^{\bbG}(\theta),R_{\bbT}^{\bbG}(\theta)\rangle_{\reg}+\langle R_{\bbT}^{\bbG}(\theta),R_{\bbT}^{\bbG}(\theta)\rangle_{\nreg}.
\end{align}
On the other hand, by the assumption (i) and Deligne--Lusztig's intertwining number formula \cite[Theorem 6.8]{DL76}, we have
\begin{align}\label{eq:theta-theta'-2}
0
=\langle R_{\bbT}^{\bbG}(\theta),R_{\bbT}^{\bbG}(\theta')\rangle
=\langle R_{\bbT}^{\bbG}(\theta),R_{\bbT}^{\bbG}(\theta')\rangle_{\reg}+\langle R_{\bbT}^{\bbG}(\theta),R_{\bbT}^{\bbG}(\theta')\rangle_{\nreg}.
\end{align}
By Deligne--Lusztig's character formula \cite[Theorem 4.2]{DL76}
\[
\Theta_{R_{\bbT}^{\bbG}(\theta)}(g)
=
\frac{1}{|\bbG_{s}(\F_{q})|}\cdot 
\sum_{\begin{subarray}{c}x\in \bbG(\F_{q}) \\ {}^{x}\bbT\subset\bbG_{s} \end{subarray}} Q_{{}^{x}\bbT}^{\bbG_{s}}(u)\cdot \theta(s^{x})
\]
(where $g=su$ denotes the Jordan decomposition), we see that the character $\Theta_{R_{\bbT}^{\bbG}(\theta)}$ on $\bbG(\F_{q})_{\nreg}$ depends only on $\theta|_{\bbT(\F_{q})_{\nreg}}$.
Since the same is true also for $\Theta_{R_{\bbT}^{\bbG}(\theta')}$, the assumption (ii) implies that $\Theta_{R_{\bbT}^{\bbG}(\theta)}$ equals $\Theta_{R_{\bbT}^{\bbG}(\theta')}$ on $\bbG(\F_{q})_{\nreg}$.
In particular, we have $\langle R_{\bbT}^{\bbG}(\theta),R_{\bbT}^{\bbG}(\theta)\rangle_{\nreg}=\langle R_{\bbT}^{\bbG}(\theta),R_{\bbT}^{\bbG}(\theta')\rangle_{\nreg}$.
Thus, by the equalities \eqref{eq:theta-theta'-1} and \eqref{eq:theta-theta'-2}, we get $\langle R_{\bbT}^{\bbG}(\theta),R_{\bbT}^{\bbG}(\theta)\rangle_{\reg}\neq\langle R_{\bbT}^{\bbG}(\theta),R_{\bbT}^{\bbG}(\theta')\rangle_{\reg}$.

We next look at the following two equalities:
\begin{align}\label{eq:theta-theta'-3}
\langle \rho,R_{\bbT}^{\bbG}(\theta)\rangle
=\langle \rho,R_{\bbT}^{\bbG}(\theta)\rangle_{\reg}+\langle \rho,R_{\bbT}^{\bbG}(\theta)\rangle_{\nreg},
\end{align}
\begin{align}\label{eq:theta-theta'-4}
\langle \rho,R_{\bbT}^{\bbG}(\theta')\rangle
=\langle \rho,R_{\bbT}^{\bbG}(\theta')\rangle_{\reg}+\langle \rho,R_{\bbT}^{\bbG}(\theta')\rangle_{\nreg}.
\end{align}
Again by the same observation as above, we have $\langle \rho,R_{\bbT}^{\bbG}(\theta)\rangle_{\nreg}=\langle \rho,R_{\bbT}^{\bbG}(\theta')\rangle_{\nreg}$.
Moreover, by the assumption on $\rho$, we have $\langle \rho,R_{\bbT}^{\bbG}(\theta)\rangle_{\reg}=\langle R_{\bbT}^{\bbG}(\theta),R_{\bbT}^{\bbG}(\theta)\rangle_{\reg}$ and $\langle \rho,R_{\bbT}^{\bbG}(\theta')\rangle_{\reg}=\langle R_{\bbT}^{\bbG}(\theta),R_{\bbT}^{\bbG}(\theta')\rangle_{\reg}$.
Since we have obtained $\langle R_{\bbT}^{\bbG}(\theta),R_{\bbT}^{\bbG}(\theta)\rangle_{\reg}\neq\langle R_{\bbT}^{\bbG}(\theta),R_{\bbT}^{\bbG}(\theta')\rangle_{\reg}$ in the previous paragraph, we have $\langle \rho,R_{\bbT}^{\bbG}(\theta)\rangle_{\reg}\neq\langle \rho,R_{\bbT}^{\bbG}(\theta')\rangle_{\reg}$.
Therefore, by combining these equalities with \eqref{eq:theta-theta'-3} and \eqref{eq:theta-theta'-4}, we get $\langle \rho,R_{\bbT}^{\bbG}(\theta)\rangle\neq\langle \rho,R_{\bbT}^{\bbG}(\theta')\rangle$.
In particular, at least one of $\langle \rho,R_{\bbT}^{\bbG}(\theta)\rangle$ and $\langle \rho,R_{\bbT}^{\bbG}(\theta')\rangle$ is not zero.
\end{proof}

Note that Lemma \ref{lem:theta-theta'} has the following immediate consequence (choose $\theta'$ to be the trivial character $\mathbbm{1}$ of $\bbT(\F_{q})$):

\begin{lem}\label{lem:theta-unip}
If $\theta|_{\bbT(\F_{q})_{\nreg}}\equiv\mathbbm{1}$, then we have either $\langle\rho,R_{\bbT}^{\bbG}(\theta)\rangle\neq0$ or $\langle\rho,R_{\bbT}^{\bbG}(\mathbbm{1})\rangle\neq0$.
\end{lem}

Hence we get the following theorem (note that this result requires \textit{no assumption on $q$}!):

\begin{thm}\label{thm:litmus with unip}
Let $\bbG$ be a connected reductive group over $\F_{q}$, $\bbT$ an $\F_{q}$-rational maximal torus of $\bbG$, and $\theta$ a regular character of $\bbT(\F_{q})$ whose restriction to $\bbT(\F_{q})_{\nreg}$ is trivial.
Suppose that $\rho$ is an irreducible representation of $\bbG(\F_{q})$ equipped with a sign $c$ such that, for any regular semisimple element $g\in\bbG(\F_{q})$, 
\[
\Theta_{\rho}(g)
=
c\cdot \Theta_{R_{\bbT}^{\bbG}(\theta)} (g).
\]
If $\rho$ is not unipotent, then we necessarily have $\rho\cong c R_{\bbT}^{\bbG}(\theta)$.
\end{thm}

\newcommand{\etalchar}[1]{$^{#1}$}
\providecommand{\bysame}{\leavevmode\hbox to3em{\hrulefill}\thinspace}
\providecommand{\MR}{\relax\ifhmode\unskip\space\fi MR }
\providecommand{\MRhref}[2]{%
  \href{http://www.ams.org/mathscinet-getitem?mr=#1}{#2}
}
\providecommand{\href}[2]{#2}

\end{document}